\documentclass[10pt]{article}

\usepackage{amsmath,amstext,amsfonts,amssymb,amsthm} 
\usepackage{fourier}
\usepackage[latin1]{inputenc}   
\usepackage[T1]{fontenc} 
\usepackage[english]{babel} 
\usepackage{hyperref}

\def \A {\mathbf{A}}
\def \Acal{\mathcal{A}}
\def \a {\mathbf{a}}
\def \B {\mathbf{B}}

\def \b {\mathbf{b}}

\def \Ccal{\mathcal{C}}
\def \Cbb {\mathbb{C}}
\def \D {\mathbf{D}}
\def \Dcal {\mathcal{D}}
\def \d {\mathbf{d}}
\def \drm {\mathrm{d}}

\def \Ecal {\mathcal{E}}
\def \e {\mathbf{e}}
\def \Ebb {\mathbb{E}}

\def \H {\mathbf{H}}

\def \I {\mathbf{I}}

\def \Kcal {\mathcal{K}}

\def \M {\mathbf{M}}

\def \Nbb{\mathbb{N}}
\def \Ncal{\mathcal{N}}

\def \Ocal{\mathcal{O}}

\def \Pbb {\mathbb{P}}
\def \Prm {\mathrm{P}}

\def \Q {\mathbf{Q}}
\def \Qrm {\mathrm{Q}}

\def \R {\mathbf{R}}
\def \Rbb {\mathbb{R}}
\def \Rcal {\mathcal{R}}
\def \s {\mathbf{s}}
\def \S {\mathbf{S}}
\def \Scal {\mathcal{S}}
\def \T {\mathbf{T}}

\def \Tcal{\mathcal{T}}

\def \v {\mathbf{v}}
\def \V {\mathbf{V}}

\def \W {\mathbf{W}}

\def \x {\mathbf{x}}
\def \X {\mathbf{X}}
\def \y {\mathbf{y}}
\def \Y {\mathbf{Y}}
\def \Z {\mathbf{Z}}
\def \Zbb {\mathbb{Z}}

\def \Deltabs {\boldsymbol{\Delta}}

\def \Lambdabs {\boldsymbol{\Lambda}}
\def \Omegabs {\boldsymbol{\Omega}}
\def \Pibs {\boldsymbol{\Pi}}
\def \Sigmabs {\boldsymbol{\Sigma}}
\def \Thetabs {\boldsymbol{\Theta}}
\def \thetabs {\boldsymbol{\theta}}
\def \mubs {\boldsymbol{\mu}}

\def \det {\mathrm{det} \ }
\def \Tr {\mathrm{Tr}\,}

\def \Prob {\mathbb{P}}
\def \dist {\mathrm{dist}}

\def \Vec{\mathrm{Vec}}
\def \adj{\mathrm{adj}}
\def \Int {\mathrm{Int}}

\renewcommand{\Im}{\mathrm{Im}}

\newcommand{\supp}{\mathrm{supp}}
\newcommand{\Exp}{\Ebb}
\newcommand{\Var}{\mathrm{Var}}

\newtheorem{theorem}{Theorem}

\newtheorem{lemma}{Lemma}
\newtheorem{remark}{Remark}

\newtheorem{proposition}{Proposition}

\renewenvironment{proof}{\noindent\textbf{Proof}:}{\hfill$\square$\bigbreak} 

\newcounter{countassum}
\newenvironment{assumption}
{
	\refstepcounter{countassum}
	\begin{flushleft}
	\noindent\textbf{Assumption A-\thecountassum}:
	\it
}
{
	\end{flushleft}
	
} 

\title{Large information plus noise random matrix models and consistent subspace estimation in large sensor networks}
\author
{
        W.Hachem$^{1}$, P.Loubaton$^{2}$, X.Mestre$^{3}$, J.Najim$^{1}$ and P.Vallet$^{2}$
        \\~\\
        $^1$ Telecom Paristech (CNRS UMR 5141), 46 rue Barrault 75634 Paris (France)
        \\
        $^2$ IGM (CNRS UMR 8049), 5 Bd. Descartes 77454 Marne-la-Vallée (France) 
        \\
        $^3$ CTTC, Av. Carl Friedrich Gauss 08860 Castelldefels, Barcelona (Spain) 
        \\
        \textsf{\{hachem, najim\}@telecom-paristech.fr, \{loubaton,vallet\}@univ-mlv.fr, xavier.mestre@cttc.cat}
}
\date{}

\begin{document}
\maketitle

\tableofcontents

\section{Introduction}

        \subsection{Motivation}
        
This paper is motivated by the problem of source localization using a large sensor network. In this context, the observation 
is a complex valued $M$-variate time series $({\bf y}_n)_{n \in \mathbb{Z}}$ ($M$ represents the number of sensors of the array) given by 
\begin{equation}        
        {\bf y}_n = \sum_{k=1}^{K}s_{k,n} \a(\theta_k) + {\bf v}_n = \A(\thetabs) {\bf s}_n + {\bf v}_n,
        \notag
\end{equation}
where
\begin{itemize}
         \item The $K < M$ scalar (in general complex valued) time series $((s_{k,n})_{n \in \Zbb}$ for $k=1,\ldots, K$ are non observable, and 
         represent the signals transmitted by $K$ transmitters. The vector $\s_n$ is given by $\s_n = (s_{1,n}, \ldots, s_{K,n})^{T}$. 
         \item For each $k$, $\theta_k$ is a scalar real parameter characterizing the direction of arrival of transmitter $k$. 
         $\theta \rightarrow \a(\theta)$ is a known $\Cbb^M$-valued function depending on the sensor network geometry, and matrix $\A(\thetabs)$ is defined as 
         $\A(\thetabs) = (\a(\theta_1), \ldots, \a(\theta_K))$.
        \item $(\v_n)_{n \in \Zbb}$ finally represents an additive complex Gaussian noise, i.e. $\v_n = (v_{1,n}, \ldots, v_{M,n})^{T}$ where the 
        $M$ time series $\left((v_{k,n})_{n \in \Zbb} \right)_{k=1, \ldots, M}$ are mutually independent identically distributed (i.i.d.) 
        sequences such that $\mathrm{Re}(v_{k,n})$ and $\mathrm{Im}(v_{k,n})$ are independent real Gaussian random variables with zero mean and variance $\sigma^{2}/2$. 
\end{itemize}
The classical source localization problem consists in estimating vector $\thetabs = (\theta_1, \ldots, \theta_K)^{T}$ 
from $N$ samples collected in the $M \times N$ matrix $\Y_N = ({\bf y}_1, \ldots, {\bf y}_N)$. 
This problem was extensively studied in the past (see e.g. \cite{stoica1989music} and the references therein). 
The so-called subspace estimator of $\thetabs = (\theta_1, \ldots, \theta_K)^{T}$ is based on the observation that if matrices $\A(\thetabs)$ and 
$\S_N = ({\bf s}_{1}, \ldots, {\bf s}_N)$ have both full rank $K$, then the angles $(\theta_k)_{k=1, \ldots, K}$ are solutions 
\footnote
{
	The $K$ angles are the unique solutions under certain assumptions on function $\theta \rightarrow \a(\theta)$
} 
of the equation $\a(\theta)^* \Pibs_N \a(\theta) = 0$, where $\Pibs_N$ represents the orthogonal projection matrix on the kernel of matrix 
$\A(\thetabs)\S_N\S_N^* \A(\thetabs)^*$. 
The existing subspace methods consist in estimating for each $\theta$ the quadratic form $\eta_N(\theta) = \a(\theta)^* \Pibs_N \a(\theta)$ of $\Pibs_N$ by a certain term
$\hat{\eta}_N(\theta)$, and then to estimate the $K$ angles as the argument of the $K$ most significant local minima of function 
$\theta \rightarrow \hat{\eta}_N(\theta)$. This approach has been extensively developed when $N \rightarrow +\infty$ and $M$ fixed. In this context, 
$\eta_N(\theta)$ can be estimated consistently for each $\theta$ by $ \hat{\eta}_N(\theta) = \a(\theta)^* \hat{\Pibs}_N \a(\theta)$ with $\hat{\Pibs}_N$ 
the orthogonal projection matrix on the eigenspace associated to the $M-K$ smallest eigenvalues of the empirical covariance matrix $\frac{1}{N} \Y_N \Y_N^{*}$. 
It clearly holds that $\sup_{\theta \in [-\pi, \pi]} \left| \hat{\eta}_N(\theta) - \eta_N(\theta) \right|$ converges torwards 0 almost surely, and this 
allows to prove that the corresponding estimators $(\hat{\theta}_k)_{k=1, \ldots, K}$ of the direction of arrivals are consistent.

If however $M$ and $N$ are of the same order of magnitude, a quite common situation if the number of sensors $M$ is large, then the above estimators show poor 
performances because $\hat{\Pibs}_N$ is no longer an accurate estimator of $\Pibs_N$. 
In order to study this context, Mestre \& Lagunas \cite{mestre2008modified} were the first to propose consistent estimators of $\eta_N(\theta)$ when 
$M,N \rightarrow +\infty$ in such a way that $c_N = \frac{M}{N} \rightarrow c$, with $c > 0$. 
In Mestre \& Lagunas \cite{mestre2008modified}, it is assumed that the source signals $(s_{k,n})_{k=1, \ldots, K}$ are mutually independent complex Gaussian i.i.d. 
time series with unit variance elements. Under this assumption, $\y_n$ can be written as 
\begin{align}
        \y_n = \R_y^{1/2} \x_n,
        \notag
\end{align}
where $\R_y = \A(\thetabs)\A(\thetabs)^* + \sigma^{2} \I_M$ represents the covariance matrix of the time series $(\y_n)_{n \in \Zbb}$ and $\x_n$ is a complex
standard Gaussian vector. 
Matrix $\Pibs_N$ coincides with the orthogonal projection matrix over the eigenspace of $\R_y $ associated to the eigenvalue $\sigma^{2}$, and Mestre \& Lagunas 
addressed the problem of estimating consistently any quadratic form of $\Pibs_N$ from the empirical covariance matrix $\frac{1}{N} \Y_N \Y_N^{*}$ 
where $\Y_N =\R_y^{1/2} \X_N$ and $\X_N = (\x_1, \ldots, {\bf x}_N)$. 
Mestre \& Lagunas \cite{mestre2008modified} used properties (see Silverstein \& Choi \cite{silversteinchoi1995}, Bai \& Silverstein \cite{bai1998no} \cite{bai1999exact}) 
of the empirical covariance matrix, and were able to exhibit a $M \times M$ matrix $\tilde{\Pibs}_{iid,N}$ such that 
$\a_N^* \tilde{\Pibs}_{iid,N} \a_N - \a_N^* \Pibs_N \a_N \rightarrow 0$ for each deterministic bounded sequence of vectors $(\a_N)$ when $M,N \rightarrow +\infty$ 
in such a way that $c_N = \frac{M}{N} \rightarrow c$, with $c > 0$. 
In some sense, matrix $\tilde{\Pibs}_{iid,N}$ can be viewed as a consistent estimate of $\Pibs_N$ but in a weak sense because in general, it does not hold that 
$\|\Pibs_{iid,N} - \Pibs_N\| \rightarrow 0$, where we have denoted by $\|.\|$ the usual spectral norm. 
Mestre \& Lagunas concluded that for each $\theta$, $\a(\theta)^* \tilde{\Pibs}_{iid,N} \a(\theta)$ is a consistent estimate of 
$\eta_N(\theta)$. 
However, the consistency of the angular estimates was not established. 
Note that these results do not require any hypothesis on $K$ which may scale with $N$ or not.

In Vallet et al. \cite{vallet2010sub}, a more general case was considered where the time series $(s_{k,n})_{k=1, \ldots, K}$ are deterministic signals for which 
the spectral norm of matrix $\frac{1}{\sqrt{N}} \A(\thetabs) {\bf S}_N$ is bounded w.r.t. the dimensions $M,N,K$. 
This time, random matrix $\Y_N$ is non zero mean, and corresponds to the so-called "Information plus Noise model" investigated in various works of 
Girko (see e.g. \cite{girko2001canonical}) and Dozier \& Silverstein (\cite{dozier2007empirical}, \cite{dozier2007analysis}). 
Using new results on the almost sure localization of the eigenvalues of the empirical covariance matrix $\frac{1}{N} \Y_N \Y_N^{*}$, 
\cite{vallet2010sub} generalized the estimator of Mestre \& Lagunas \cite{mestre2008modified}, and derived a "weakly consistent estimator"  $\tilde{\Pibs}_N$ of 
$\Pibs_N$, i.e.  $\a_N^* \tilde{\Pibs}_N \a_N - \a_N^* \Pibs_N \a_N \rightarrow 0$ for each deterministic bounded sequences 
of vectors $(\a_N)$. Therefore, it holds that for each $\theta$, $\tilde{\eta}_N(\theta) = \a(\theta)^* \tilde{\Pibs}_N \a(\theta)$ is a consistent 
estimate of $\eta_N(\theta)$ if $\a(\theta)$ is uniformly bounded in $N$. 

The goal of the present paper is to pursue the work \cite{vallet2010sub}, and to establish that the angle estimates defined as the $K$ most significant local minima of 
function $\theta \rightarrow \a(\theta)^* \tilde{\Pibs}_N \a(\theta)$ are consistent. 
As it will be shown below, the consistency of the angle estimates is based on the property
\begin{equation}
        \sup_{\theta \in [-\pi, \pi]} |\tilde{\eta}_N(\theta) -\eta_N(\theta)| \rightarrow 0
        \label{eq:cqfd}
\end{equation}
almost surely, that we shall refer to as the {\bf uniform consistency} of the estimate $\tilde{\eta}_N(\theta)$ of $\eta_N(\theta)$.

This paper is organized as follows. 
In Section \ref{sec:rappels}, we provide some background material on the asymptotic eigenvalue distribution of the large information plus noise model, 
on the almost sure localization of the eigenvalues of the empirical covariance matrix, and on the consistent estimator of $\a_N^* \Pibs_N \a_N$ proposed in 
\cite{vallet2010sub}. 
In Section \ref{sec:sup}, we prove the property of uniform consistency of estimator $\tilde{\eta}_N(\theta)$ (see \eqref{eq:cqfd}) when function $\a(\theta)$ 
is defined by  
\begin{equation}
        \a(\theta) = \frac{1}{\sqrt{M}}  \left(1, e^{i \theta}, \ldots, e^{i(M-1) \theta}\right)^T.
        \label{eq:modele-a}
\end{equation}
\eqref{eq:cqfd} of course holds for more general functions, but we believe that considering the typical example defined by \eqref{eq:modele-a}
is informative enough. The proof of \eqref{eq:cqfd} heavily relies on results concerning the probability that the eigenvalues 
of $\frac{\Y_N \Y_N^{*}}{N}$ escape from the intervals in which they are located almost surely for $N$ large enough. These results are believed to 
be of independent interest. 
Finally, we establish  in Section \ref{sec:consistency} the consistency of the $K$ most significant local minima of function 
$\theta \rightarrow \tilde{\eta}_N(\theta)$ by following the approach in \cite{hannan1973estimation}.  

         \subsection{General notations and useful results} 

 We now introduce various notations and results  used throughout the paper.
 \begin{itemize}
         \item If $E \subset \Rbb$, $\mathrm{Int}(E)$ and $\partial E$ represent the interior and the boundary of $E$ respectively.
        \item If $z \in \Cbb$, the complex conjugate of $z$ is denoted $\overline{z}$ or $z^*$. For a complex matrix $\A$, we denote its transpose by $\A^T$
        and its Hermitian adjoint by $\A^*$.
        \item We denote by ${\cal C}^{\infty}(\Rbb, \Rbb)$ (respectively ${\cal C}_c^{\infty}(\Rbb, \Rbb))$ the set of all 
        smooth real-valued functions (resp. compactly supported smooth real values functions).
        \item The quantity $C$ will represent a generic positive constant whose main feature is to be deterministic and independent of $M$ and $N$. 
        The value of $C$ may change from one line to another.
        \item Similarly, $\Prm_1$ and $\Prm_2$ will denote generic polynomials, independent of $M$ and $N$, with positive coefficients. 
        The polynomials may change from one line to another. 
        \item Complex Gaussian distribution: A complex valued random variable $Z = X + i Y$ follows the distribution $\Ccal\Ncal\left(\alpha + i \beta, \sigma^{2}\right)$
        if $X$ and $Y$ are independent real Gaussian random variables $\Ncal\left(\alpha, \frac{\sigma^{2}}{2}\right)$ and $\Ncal\left(\beta, \frac{\sigma^{2}}{2}\right)$ respectively. 
        The variance of $Z$, denoted as $\Var(Z)$ is defined as $\Var(Z) = \Ebb|Z-\Ebb[Z]|^{2} = \sigma^{2}$.
        \item Poincaré inequality (see Chen \cite{chen1982inequality}): let $Z_1 = X_1 + i Y_1, \ldots, Z_p= X_p + i Y_p$ be $p$ iid $\Ccal\Ncal(0,\sigma^{2})$ random variables and consider a function $\gamma$ 
        defined on $\Rbb^{2p}$ continuously differentiable with polynomially bounded partial derivatives. 
        Then, if $\X = (X_1, \ldots, X_p)^{T}$ and $\Y = (Y_1, \ldots, Y_p)^{T}$, the random variable $\gamma(\X, \Y)$ can be written as 
        $\gamma(\X, \Y) = \tilde{\gamma}(\Z, {\bf \overline{Z}})$ and 
        \begin{equation}
	\Var\left[\gamma(\X, \Y)\right] 
	= \Var\left[\tilde{\gamma}(\Z, {\bf \overline{Z}})\right] 
	\leq \sigma^{2} \sum_{i=1}^p 
	\left(
	        \Ebb \left|\frac{\partial \tilde{\gamma}(\Z, {\bf \overline{Z}})}{\partial z_i}\right|^{2} 
	        + \Ebb \left|\frac{\partial \tilde{\gamma}(\Z, {\bf \overline{Z}})}{\partial \overline{z}_i}\right|^{2}
	\right),
	\notag
        \end{equation}
        where we define as usual the differential operators 
        $\frac{\partial}{\partial z} = \frac{1}{2}\left(\frac{\partial}{\partial x} - i\frac{\partial}{\partial y} \right)$
        and $\frac{\partial}{\partial \overline{z}} = \frac{1}{2} \left(\frac{ \partial}{\partial x} + i\frac{ \partial}{\partial y} \right)$.
        If $\gamma$ is real-valued, it is clear that $\frac{\partial \tilde{\gamma}(\Z, {\bf \overline{Z}})}{\partial \overline{z}_i}$ coincides with the complex conjugate
        of $\frac{\partial \tilde{\gamma}(\Z, {\bf \overline{Z}})}{\partial z_i}$. In this case, the Poincaré inequality reduces to  
        \begin{equation}
	\Var(\gamma(\X, \Y)) \leq 2 \sigma^{2} \sum_{i} \Ebb \left| \frac{\partial \tilde{\gamma}(\Z, {\bf \overline{Z}})}{\partial z_i} \right|^{2}.
	\notag
        \end{equation}
        \item Stieltjes transform: Let $\mu$ be a positive finite measure on $\Rbb$. Its Stieltjes transform $m$ is the function defined by 
        \begin{align}
		m(z) = \int_{\Rbb} \frac{\drm\mu(\lambda)}{\lambda - z}, \quad \forall z \in \Cbb\backslash\supp(\mu),
	\notag
        \end{align}
        where $\supp(\mu)$ represents the support of measure $\mu$. 
        Function $m$ is holomorphic on $\Cbb\backslash\supp(\mu)$ and satisfies $\frac{\Im(m(z))}{\Im(z)} > 0$ for $z \in \Cbb\backslash\Rbb$ and 
        $m(iy) \rightarrow 0$ when $y \rightarrow +\infty$. 
        Moreover, $\supp(\mu) \subset \Rbb^{+}$ if and only if $\frac{\Im (z m(z))}{\Im (z)} > 0$ for $z \in \Cbb\backslash\Rbb$. 
        The mass of the measure $\mu$ can be evaluated through the formula 
        \begin{equation}
	\mu(\Rbb) = \lim_{y \rightarrow +\infty} -iy m(iy).
	\notag
        \end{equation}
        We also notice that if $m(z)$ is the Stieltjes transform of positive measure $\mu$, then it holds that 
        \begin{equation}
	\left|m(z)\right| \leq \frac{\mu(\Rbb)}{\dist\left(z,\supp(\mu)\right)} \leq \frac{\mu(\Rbb)}{|\Im(z)|},
	\notag
        \end{equation}
        and that $m'(z) = \int_{\Rbb} \frac{\drm\mu(\lambda)}{(\lambda - z)^{2}}$ satisfies 
        \begin{equation}
	\left|m'(z)\right|\leq\frac{\mu(\Rbb)}{\dist\left(z,\supp(\mu)\right)^{2}}\leq\frac{\mu(\Rbb)}{|\Im(z)|^{2}},
	\notag
        \end{equation}
        on $\Cbb\backslash\supp(\mu)$. We finally recall the following version of the inverse Stieltjes transform formula: For each function $\psi \in {\cal C}_c^{\infty}(\Rbb, \Rbb))$, 
        we have 
        \begin{equation}
		\label{eq:inverse-stieltjes-psi}
		\int_{\Rbb} \psi(\lambda) \drm \mu(\lambda) = 
		\frac{1}{\pi} \lim_{y \downarrow 0} \Im \left(\int_{\Rbb} \psi(\lambda) m(\lambda + iy)  \drm \lambda \right).
        \end{equation}
 
 \end{itemize}

\section{\texorpdfstring{Background on the Information plus Noise model and on the estimator of \cite{vallet2010sub}}{Background on the Information plus Noise model}}
\label{sec:rappels}
 
All along this paper, we consider integers $M, N, K \in \mathbb{N}^*$ such that $1 \leq K < M$, $K=K(N)$ and $M=M(N)$ are functions of $N$ with 
$c_N = \frac{M}{N} \to c > 0$ as $N \to \infty$. We assume that 
\begin{assumption}
        \label{assumption:c}
        $0 < c_N < 1 \; \mbox{and} \; 0 < c < 1$.
\end{assumption} 
In this section, $\boldsymbol{\Sigma}_N$ represents the complex valued $M \times N$ random matrix given by 
\begin{align}
	\boldsymbol{\Sigma}_N = \frac{\Y_N}{\sqrt{N}} = \B_N +  \W_N,
	\notag
\end{align}
where $\B_N = \frac{\A(\thetabs) \S_N}{\sqrt{N}}$ and $\W_N =  \frac{\V_N}{\sqrt{N}}$. 
Matrices $\B_N$ and $\W_N$ are assumed to satisfy the following assumptions
\begin{assumption}
	\label{assumption:norm_spec_BN}
	Matrix $\B_N$ is deterministic and satisfies $\sup_N \|\B_N\| < +\infty$
\end{assumption}
\begin{assumption}
      \label{assumption:rankB}
	$\mathrm{Rank} (\B_N \B_N^*) = K < M$ where $K$ may scale with $N$ or not.
\end{assumption}
\begin{assumption}
	\label{assumption:WN_gaussian}
	The entries of matrix $\W_N$ are i.i.d and follow the complex normal distribution $\mathcal{CN}(0,\frac{\sigma^{2}}{N})$.
\end{assumption}
We assume moreover that the non zero eigenvalues of $\B_N \B_N^*$ have multiplicities 1 in order to simplify the notations. In the following, 
we denote by  $0 = \lambda_{1,N} = \ldots = \lambda_{M-K,N} < \lambda_{M-K+1,N} < \ldots < \lambda_{M,N}$ and $({\bf u}_{k,N})_{k=1, \ldots, M}$ 
the ordered eigenvalues and associated eigenvectors of $\B_N \B_N^*$. The eigenvalues and the eigenvectors of matrix $\Sigmabs_N \Sigmabs_N^*$ are
denoted $(\hat{\lambda}_{k,N})_{k=1, \ldots, M}$ and $(\hat{{\bf u}}_{k,N})_{k=1, \ldots, M}$, and $\hat{\mu}_N$ represents the
empirical eigenvalue distribution of $\Sigmabs_N \Sigmabs_N^*$ defined by 
\begin{equation}
        \hat{\mu}_N = \frac{1}{M} \sum_{k=1}^{M} \delta_{\hat{\lambda}_{k,N}}.
        \notag
\end{equation}
As we assume $c_N < 1$, the joint probability distribution of $(\hat{\lambda}_{k,N})_{k=1,\ldots,M}$ is absolutely continuous (see e.g. James \cite{james1964distributions})
and it holds that the  $(\hat{\lambda}_{k,N})_{k=1, \ldots, M}$ have multiplicity 1 almost surely. 
We finally denote by $\Q_N(z)$ the resolvent of matrix $\Sigmabs_N \Sigmabs_N^*$, i.e. $\Q_N(z) = \left( \Sigmabs_N \Sigmabs_N^* - z \I_M \right)^{-1}$.

\subsection{\texorpdfstring{The asymptotic eigenvalue distribution $\mu_N$ of $\hat{\mu}_N$}{The asymptotic eigenvalue distribution}}

It is well-known (\cite[Th. 7.4]{girko2001canonical}, \cite[Th. 1.1]{dozier2007empirical}) that it exists a sequence of deterministic probability measures $(\mu_N)$ 
such that $\hat{\mu}_N - \mu_N \to_N 0$ weakly almost surely. Measure $\mu_N$ is characterized by its Stieltjes transform $m_N(z)$ which 
is known to satisfy the equation 
\begin{equation}
        m_N(z) = \frac{1}{M}\mathrm{Tr}
        \left[  
	-z(1+\sigma^{2}c_{N}m_{N}(z))\mathbf{I}_{M}+\sigma^{2}(1-c_{N})\mathbf{I}_{M}+\frac{\mathbf{B}_{N}\mathbf{B}_{N}^{*}}{1+\sigma^{2}c_{N}m_{N}(z)}
        \right]^{-1},
        \label{eq:canonique}
\end{equation}
for each $z \in \Cbb\backslash\Rbb^{+}$. In the following, we denote by $\Scal_N$ the support of $\mu_N$. 
As  $\hat{\mu}_N - \mu_N \to_N 0$ weakly almost surely, it holds that 
\begin{equation}
        \label{eq:convergence-hatmN}
        \hat{m}_N(z) - m_N(z) \rightarrow 0
\end{equation}
almost surely for each $z \in \Cbb\backslash\Rbb^{+}$. 
The following result will be of help.
\begin{lemma}[\cite{haagerup2005new}, \cite{capitaine2007freeness}]
        \label{lemma:haagerup_capitaine}
	Let $\psi \in \Ccal_c^{\infty}(\Rbb,\Rbb)$ and $(r_N)$ a sequence of holomorphic functions on $\Cbb\backslash\Rbb$ such that
	\begin{align}
		\left|r_N(z)\right| \leq \Prm_1(|z|) \Prm_2\left(\frac{1}{|\mathrm{Im}(z)|}\right),
		\notag
	\end{align}
	with $P_1$ and $P_2$ two polynomials with positive coefficients, independent of $N$.
	Then,
	\begin{align}
		\limsup_{y \downarrow 0} \left|\int_{\Rbb} \psi(x) r_N(x + i y) \drm x\right| \leq C < \infty,
		\notag
	\end{align}
	with $C$ a constant independent of $N$.
\end{lemma}
Taking into account the previous result, it is shown in \cite{vallet2010sub} that
\begin{equation}
        \label{eq:expre-Ehatmn}
        \Ebb[\hat{m}_N(z)] = m_N(z) + \frac{r_N(z)}{N^{2}},
\end{equation}
with $r_N$ as in Lemma \ref{lemma:haagerup_capitaine}. Using the inverse Stieltjes transform formula (\ref{eq:inverse-stieltjes-psi}), we obtain that for each function $\psi \in \Ccal_c^{\infty}(\Rbb,\Rbb)$, it holds that
\begin{equation}
        \label{eq:inegalite-haagerup}
        \frac{1}{M} \Ebb \left[\Tr \psi(\Sigmabs_N  \Sigmabs_N^{*}) \right] 
        = \frac{1}{M} \sum_{k=1}^{M} \Ebb\left[\psi(\hat{\lambda}_{k,N})\right]
        = \int_{\Scal_N} \psi(\lambda) \mu_N(d \lambda) + \Ocal\left(\frac{1}{N^{2}}\right).
\end{equation} 
If we denote by $\T_N(z)$ the matrix-valued function defined by 
\begin{equation}
        {\bf T}_N(z) = 	
        \left[  
	-z(1+\sigma^{2}c_{N}m_{N}(z))\mathbf{I}_{M}+\sigma^{2}(1-c_{N})\mathbf{I}_{M}+\frac{\mathbf{B}_{N}\mathbf{B}_{N}^{*}}{1+\sigma^{2}c_{N}m_{N}(z)}
        \right]^{-1},
        \notag
\end{equation}
then ${\bf T}_N$ coincides with the Stieltjes transform of a positive matrix valued measure $\mubs_N$ with support $\Scal_N$ such that 
$\mubs_N(\Scal_N) = \I_M$ (see Hachem et al \cite[Th. 2.4 \& Prop. 2.2]{hachem2007deterministic}), i.e. 
\begin{align}
        \T_N(z) =  \int_{\Scal_N} \frac{d \mubs_N(\lambda)}{\lambda - z}.
        \notag
\end{align}
As $m_N(z)$ verifies the equation \eqref{eq:canonique}, it is clear that $\frac{1}{M} \Tr \mubs_N = \mu_N$. 

In the remainder of the paper, we will make use of the following result proved in \cite{vallet2010sub} if ${\bf W}_N$ is complex Gaussian and in 
Hachem et al \cite{hachem2010bilinear} in the non Gaussian case. 
\begin{theorem}
        \label{prop:convergence-resolvente}
        Consider two sequences of deterministic vectors $(\b_N)$, $(\d_N)$ such that $\sup_N\|\b_N \| < +\infty$ and $\sup_N \| \d_N  \| < +\infty$. 
        Then, it holds that 
        \begin{equation}
		\label{eq:convergence-resolvente}
		\b_N^* \Q_N(z) \d_N - \b_N^* \T_N(z) \d_N \xrightarrow[N]{} 0,
        \end{equation}
        almost surely for each $z \in \Cbb\backslash\Rbb^{+}$. 
\end{theorem}

\subsection{\texorpdfstring{The characterization of the support $\Scal_N$ of $\mu_N$}{The characterization of the support}}

The support $\Scal_N$ of $\mu_N$ was first studied in Dozier \& Silverstein \cite{dozier2007analysis} and a more convenient characterization was presented 
in \cite{vallet2010sub}. We first recall (see \cite{dozier2007analysis}) that if $z \in \Cbb^{+}$ converges torwards $x \in \Rbb$, 
then, $m_N(z)$ converges torwards a finite limit still denoted $m_N(x)$. Function $x \rightarrow m_N(x)$ is continuous on $\Rbb$, continuously differentiable 
on $\Rbb\backslash\partial\Scal_N$, and verifies Eq. \eqref{eq:canonique} on $\Rbb\backslash\partial\Scal_N$. 
Moreover, $\mu_N$ is absolutely continuous and its density coincides with function $\frac{1}{\pi}\Im(m_N(x))$. 

In order to present the chacterization of $\Scal_N$, we first introduce the following notations. We denote by $f_N, \phi_N$ and $w_N$ the functions defined by 
\begin{align}
	f_N(w) &= \frac{1}{M} \Tr \left(\B_N \B_N^* - w \I_M\right)^{-1}, 
	\notag \\
	\phi_{N}(w) &= w \left(1 - \sigma^2 c_N f_N(w) \right)^2 + \sigma^2 (1-c_N) \left(1 - \sigma^2 c_N f_N(w) \right), 
	\notag \\
      	w_N(z) &  =  z (1 + \sigma^2 c_N m_N(z))^2 - \sigma^2 (1-c_N) (1 + \sigma^2 c_N m_N(z)). 
      	\label{eq:def-wN}  
\end{align}
We are now in position to characterize $\Scal_N$. 
\begin{theorem}
        \label{theo:support}
        The function $\phi_{N}$ admits $2 Q$ non-negative local extrema counting multiplicities (with $1\leq Q \leq K+1$) whose preimages are denoted 
        $w_{1,N}^-<0<w_{1,N}^+\leq w_{2,N}^-\ldots\leq w_{Q,N}^-<w_{Q,N}^+$. 
        Define $x_{q,N}^-=\phi_{N}(w_{q,N}^-)$ and $x_{q,N}^+=\phi_{N}( w_{q,N}^+)$ for $q=1\ldots Q$. 
        Then,
        \begin{align}
	x_{1,N}^- < x_{1,N}^+ \leq x_{2,N}^- < \ldots \leq x_{Q,N}^- < x_{Q,N}^+,
	\notag
        \end{align}
        and the support $\Scal_N$ of $\mu_N$ is given by
        \begin{align}
	\mathcal{S}_{N}=\bigcup_{q=1}^{Q}\left[x_{q,N}^{-},x_{q,N}^{+}\right].
	\notag
        \end{align}
        Moreover, for $q=1,\ldots, Q$, each interval $]w_{q,N}^-,w_{q,N}^+[$ contains at least an element of the set $\{ 0,\lambda_{M-K+1,N},\ldots,\lambda_{M,N}\}$ 
        and each eigenvalue of $\B_N \B_N^*$ belongs to one of these intervals.
\end{theorem}
The second statement of the theorem shows that each eigenvalue of $\B_N \B_N^*$ corresponds to a certain interval of $\Scal_N$. 
More precisely, an eigenvalue of $\B_N \B_N^*$ will be said to be associated to cluster $[x_{q,N}^{-}, x_{q,N}^{+}]$ if 
it belongs to the interval $(w_{q,N}^{-}, w_{q,N}^{+})$. We note that the eigenvalue $0$ is necessarily associated 
to the first cluster $[x_{1,N}^{-}, x_{1,N}^{+}]$.

We finally recall the useful properties of function $w_N$ defined by \eqref{eq:def-wN} (see \cite{vallet2010sub}). 
We still denote by $w_N(x)$ the limit of $w_N(z)$ when $z \in \Cbb^+$ converges torwards $x \in \Rbb$. 
\begin{proposition}
         \label{prop:w}
        Function $w_N:\Cbb \to \Cbb$ satisfies the following properties:      
        \begin{itemize}
	\item Function $x \rightarrow w_N(x)$ is continuous on $\Rbb$ and continuously differentiable on $\Rbb\backslash\partial \Scal_N$, 
	\item $\Im(w_N(z)) > 0$ if $\Im(z) > 0$,
	\item $w_N$ is real and strictly increasing on $\Rbb\backslash\Scal_N$,
	\item $w_N(x_{q,N}^-) = w_{q,N}^-$ and $w_N(x_{q,N}^+) = w_{q,N}^+$ for each $1 \leq q \leq Q$,
	\item $\Im(w_N(x)) > 0$ if and only if $x \in \Int(\Scal_N)$.
        \end{itemize}
\end{proposition}

\subsection{Some useful evaluations}
\label{subsec:evaluations}

In this paragraph, we gather some useful bounds related to certain Stieltjes transforms. We first recall that the inequality 
\begin{equation}
	\label{eq:inegalite-Reb}
	|1 + \sigma^{2} c_N m_N(z)| \geq \mathrm{Re}(1 + \sigma^{2} c_N m_N(z)) \geq 1/2
\end{equation}
holds for $z \in \Cbb$ (see Loubaton \& Vallet \cite{loubaton2010ejpsub}). 
We now consider function $z \rightarrow \frac{-1}{z(1+\sigma^{2}c_N m_N(z))}$. 
Proposition 2.2 in \cite{hachem2007deterministic} implies that it coincides with the Stieltjes transform of a probability measure carried by $\Rbb^{+}$. Moreover, (\ref{eq:inegalite-Reb}) 
shows that the support of this measure  is included in $\Scal_N \cup \{0\}$. Therefore, we obtain that  
\begin{equation}
        \label{eq:majoration1surbC+}
        \frac{1}{|1+\sigma^{2}c_N m_N(z))|} \leq \frac{|z|}{|\Im(z)|}
\end{equation}
for each $z \in \Cbb\backslash\Rbb$ as well as 
$$
	\frac{1}{|1+\sigma^{2}c_N m_N(z))|} \leq \frac{|z|}{\dist(z, \Scal_N)}
$$
for each $z \in \Cbb^*\backslash \Scal_N$. We also recall that matrix $\T_N(z)$ satisfies $\T_N(z) \T_N(z)^* \leq \frac{\I_M}{\Im(z)^2}$ for $z \in \Cbb^{+}$ 
(see \cite[Prop. 5.1]{hachem2007deterministic}). 
We now claim that the inequality 
\begin{equation}
        \label{eq:inegaliteT-C-S}
        \T_N(z) \T_N(z)^* \leq  \frac{\I_M}{\dist(z, \Scal_N)^2}
\end{equation}
also holds on $\Cbb \backslash \Scal_N$. In order to establish \eqref{eq:inegaliteT-C-S}, we follow the proof 
of Proposition 5.1 in \cite{hachem2007deterministic}. We first remark that function $\tilde{m}_N(z)$ defined by 
\begin{equation}
        \tilde{m}_N(z) = c_N m_N(z) - \frac{1-c_N}{z}
        \notag
\end{equation}
is the Stieltjes transform of probability measure $\tilde{\mu}_N = c_N \mu_N + (1 - c_N) \delta_0$. 
The support of $\tilde{\mu}_N$ thus coincides with $\Scal_N \cup \{ 0 \}$, and is included in $\Rbb^{+}$. 
Therefore, it holds that $\frac{\Im( z  \tilde{m}_N(z))}{\Im(z)} > 0$ if $z \in \Cbb\backslash\Rbb$. 
We remark that 
\begin{align}
        \frac{\T_N(z) - \T_N(z)^{*}}{2i} = \Im(z) \int_{\Scal_N} \frac{d \mubs_N(\lambda)}{|\lambda - z|^{2}}.
        \notag
\end{align}
By using the identity, $\T_N(z) - \T_N(z)^{*} = \T_N(z) \left( \T_N(z)^{-*} - \T_N(z)^{-1} \right) \T_N(z)^{*}$, we get after some algebra 
\begin{align}
        &\Im(z) \int_{\Scal_N} \frac{d \mubs_N(\lambda)}{|\lambda - z|^{2}} 
        =
        \notag \\
        &\qquad
        \Im(z) \T_N(z) \T_N(z)^* + \sigma^{2} \Im (z \tilde{m}_N(z))\T_N(z) \T_N(z)^* 
        + \frac{\sigma^{2} c_N}{|1+\sigma^{2} c_N m_N(z)|^{2}} \Im(m_N(z)) \T_N(z) \B_N \B_N^* \T_N(z)^{*},
        \notag
\end{align}
for each $z \in \Cbb\backslash\Rbb$, or equivalently
\begin{align}
        \int_{\Scal_N} \frac{d \mubs_N(\lambda)}{|\lambda - z|^{2}} =  \T_N(z) \T_N(z)^* + \sigma^{2} \frac{ \Im (z \tilde{m}_N(z))}{\Im(z)}  \, \T_N(z) \T_N(z)^* +
        \frac{\sigma^{2} c_N}{|1+\sigma^{2} c_N m_N(z)|^{2}} \frac{\Im(m_N(z))}{\Im(z)} \T_N(z) \B_N \B_N^* \T_N(z)^{*}.
        \notag
\end{align}
Consequently, we obtain that  
\begin{align}
        \T_N(z) \T_N(z)^* \leq \int_{\Scal_N} \frac{d \mubs_N(\lambda)}{|\lambda - z|^{2}}
        \notag
\end{align}
for $z \in \Cbb\backslash\Rbb$, but also for $z \in \Cbb\backslash\Scal_N$ because both members of above inequality are continuous on  
$\Cbb\backslash\Scal_N$. 
This immedialely leads to \eqref{eq:inegaliteT-C-S}. This inequality also implies that for each $z \in \Cbb\backslash\Scal_N$,
\begin{equation}
        \label{eq:distancewN-S}
        \min_{k=1, \ldots, M} \left| \lambda_{k,N} - w_N(z) \right| \geq \frac{1}{2} \dist(z, \Scal_N).
\end{equation}
Indeed, $\T_N(z)$ can be written as $\T_N(z) = (1 + \sigma^{2} c_N m_N(z)) \left( \B_N \B_N^* - w_N(z) \I_M \right)^{-1}$. 
Therefore, $\| \T_N(z) \|$ is equal to 
\begin{align}
        \| \T_N(z) \| = \frac{|1 + \sigma^{2} c_N m_N(z)|}{\min_{k=1, \ldots, M} \left| \lambda_{k,N} - w_N(z) \right|}
        \notag
\end{align}
so that \eqref{eq:distancewN-S} follows from \eqref{eq:inegaliteT-C-S} and \eqref{eq:inegalite-Reb}.

Since $\hat{m}_N(z)$ is the Stieltjes transform of the distribution $\frac{1}{M} \sum_{k=1}^{M} \delta_{\hat{\lambda}_{k,N}}$, it holds that 
\begin{equation}
        |\hat{m}_N(z)| \leq \frac{1}{\dist\left(z, \{\hat{\lambda}_{1,N}, \ldots, \hat{\lambda}_{M,N} \}\right)},
        \notag
\end{equation}
as well as
\begin{equation}
        |\hat{m}'_N(z)| \leq \frac{1}{\dist\left(z, \{\hat{\lambda}_{1,N}, \ldots, \hat{\lambda}_{M,N} \}\right)^{2}}.
        \notag
\end{equation}
We now consider the rational function $z \mapsto \frac{1}{1+\sigma^{2}c_N \hat{m}_N(z)}$ which will play an important role in the following. 
Its poles are solutions of the equation $1 + \sigma^{2} c_N \hat{m}_N(z) = 0$, and satisfy some useful properties. 
From now on, we denote by $\hat{\Lambdabs}_N$ the diagonal matrix 
$\hat{\Lambdabs}_N = \mathrm{Diag} \left( \hat{\lambda}_{1,N}, \ldots, \hat{\lambda}_{M,N} \right)$ and by $\hat{\Omegabs}_N$ the matrix
\begin{align}
	\hat{\Omegabs}_N = \hat{\Lambdabs}_N + \frac{\sigma^2 c_N}{M} \mathbf{1} \mathbf{1}^T,
\end{align}
where $\mathbf{1}$ denotes vector $\mathbf{1} = (1, 1, \ldots, 1)^{T}$. We denote $\hat{\omega}_{1,N} \leq \ldots \leq \hat{\omega}_{M,N}$ its eigenvalues. 
Then we have the following straighforward properties.
\begin{itemize}
        \item The zeros of $z \mapsto 1 + \sigma^{2} c_N \hat{m}_N(z)$ are included in the set $\{\hat{\omega}_{1,N},\ldots,\hat{\omega}_{M,N}\}$.
        \item If the eigenvalues $\hat{\lambda}_{1,N},\ldots,\hat{\lambda}_{M,N}$ of $\Sigmabs_N \Sigmabs_N^*$ have multiplicity one, 
        the equation $1 + \sigma^{2} c_N \hat{m}_N(z) = 0$ has $M$ multiplicity one solutions which coincide with the $(\hat{\omega}_{k,N})_{k=1, \ldots, M}$. 
        Moreover, $\hat{\lambda}_{1,N} < \hat{\omega}_{1,N} < \ldots < \hat{\lambda}_{M,N} < \hat{\omega}_{M,N}$.
        \item If the eigenvalue $\hat{\lambda}_{k,N}$ has multiplicity $p > 1$, i.e. $\hat{\lambda}_{k-1,N} < \hat{\lambda}_{k,N} = \hat{\lambda}_{k+p-1,N} < \hat{\lambda}_{k+p,N}$, 
        then, 
        \begin{align}
	\hat{\omega}_{k-1,N} <  \hat{\lambda}_{k,N} = \hat{\omega}_{k,N} 
	= \ldots 
	=  \hat{\lambda}_{k+p-2,N} = \hat{\omega}_{k+p-2,N} = \hat{\lambda}_{k+p-1,N} < \hat{\omega}_{k+p-1,N} < \hat{\lambda}_{k+p,N},
	\notag
        \end{align}
        and the $\hat{\omega}_{k,N}$ that do not coincide with some eigenvalues of $\Sigmabs_N \Sigmabs_N^*$ are zeros of $1 + \sigma^{2} c_N \hat{m}_N(z)$. 
\end{itemize}
\begin{remark}
        Since $c_N < 1$, we recall that the eigenvalues  $(\hat{\lambda}_{k,N})_{k=1, \ldots, M}$ have multiplicity 1 almost surely. However, in subsection \ref{subsec:escape-omega}, 
        it will be necessary to define properly the solutions of $1 + \sigma^{2} c_N \hat{m}_N(z) = 0$ everywhere. This explains why the case where some 
        of the  $(\hat{\lambda}_{k,N})_{k=1, \ldots, M}$ are multiple has to  be considered.
\end{remark}
Function $z \mapsto \frac{-1}{z(1+\sigma^{2} c_N \hat{m}_N(z))}$ is the Stieltjes transform of a probability measure whose support coincides with the set of all roots of 
the equation $z(1+\sigma^{2} c_N \hat{m}_N(z)) = 0$, which is included into
the set $\{ 0, \hat{\omega}_{1,N}, \ldots, \hat{\omega}_{M,N} \}$. Therefore, it holds that 
\begin{equation}
        \frac{1}{|1+\sigma^{2} c_N \hat{m}_N(z)|} \leq \frac{|z|}{\dist(z, \{ 0, \hat{\omega}_{1,N}, \ldots, \hat{\omega}_{M,N} \})}
        \notag
\end{equation}
for $z \in \Cbb\backslash \{ 0, \hat{\omega}_{1,N}, \ldots, \hat{\omega}_{M,N} \}$
and 
\begin{equation}
 \label{eq:borne-bhat-1}
        \frac{1}{|1+\sigma^{2} c_N \hat{m}_N(z)|} \leq \frac{|z|}{|\Im(z)|}
\end{equation}
for $z \in \Cbb\backslash\Rbb$. We eventually notice that 
\begin{equation}
        \left\| \Q_N(z) \right\| \leq  \frac{1}{\dist(z, \{\hat{\lambda}_{1,N}, \ldots, \hat{\lambda}_{M,N} \})}.
        \notag
\end{equation}

\subsection{\texorpdfstring{Almost sure localization of the eigenvalues $(\hat{\lambda}_{k,N})_{k=1, \ldots, M}$}{Almost sure localization of the eigenvalues}}

We recall the two following useful results of \cite{vallet2010sub} and \cite{loubaton2010ejpsub}. 
\begin{theorem}[\cite{vallet2010sub}]
	\label{theorem:no_eig}
	Assume assumptions \textbf{A-1} to \textbf{A-4} hold. Let $a,b \in \Rbb$, $\epsilon > 0$ and $N_0 \in \mathbb{N}$ such that
	\begin{align}
		]a - \epsilon, b + \epsilon[ \cap \mathcal{S}_N = \emptyset,
		\notag
	\end{align}
	for each $N > N_0$.
	Then, with probability one, no eigenvalue of $\boldsymbol{\Sigma}_N \boldsymbol{\Sigma}_N^*$ belongs to $[a,b]$ for $N$ large enough.
\end{theorem}
\begin{theorem}[\cite{loubaton2010ejpsub}]
	\label{theorem:loc_eig}
	Assume assumptions \textbf{A-1} to \textbf{A-4} hold. Let $a,b \in \Rbb$, $\epsilon > 0$, $N_0 \in \mathbb{N}$ such that $]a- \epsilon, b + \epsilon[ \cap \mathcal{S}_N = \emptyset$ for $N > N_0$. 
	Then, with probability 1,
	\begin{align}
		\mathrm{card}\{k: \hat{\lambda}_{k,N} < a \} &= \mathrm{card}\{k: \lambda_{k,N} < w_N(a)\} \label{eq:inferieur} \\
		\mathrm{card}\{k: \hat{\lambda}_{k,N} >  b \} &= \mathrm{card}\{k: \lambda_{k,N} > w_N(b)\} \label{eq:superieur} 
	\end{align}
	for $N$ large enough.
\end{theorem}
It is useful to mention that $\sup_{N} x_{Q_N,N}^{+} < +\infty$ and that these two theorems are still valid if $b = +\infty$ (see \cite{loubaton2010ejpsub}).

\subsection{\texorpdfstring{The consistent estimate of quadratic forms of $\Pibs_N$}{The consistent estimate of quadratic forms of Pi}}

Let $\Pibs_N$ be the orthogonal projection matrix on the kernel of $\B_N \B_N^*$ and let $(\a_N)_{N \in \mathbb{N}}$ be 
a sequence of deterministic $M$--dimensional vectors such that $\sup_{N} \|\a_N\| < \infty$. Then, 
\cite{vallet2010sub} proposed a consistent estimate of $\eta_N$ defined by 
\begin{equation}
        \eta_N  = \a_N^* \Pibs_N \a_N.
        \notag
\end{equation} 
The approach of \cite{vallet2010sub} is valid under the following assumptions.
\begin{assumption}
        \label{ass:separation1}
        For $N$ large enough, none of the strictly positive eigenvalues of $\B_N \B_N^*$ is associated to the first cluster $[x_{1,N}^{-}, x_{1,N}^{+}]$, i.e. 
$\lambda_{M-K+1,N} > w_N(x_{1,N}^{+})$ for $N$ large enough. 
\end{assumption}
\begin{assumption}
        \label{ass:separation2}
        It holds that 
        \begin{align}
	0  < \liminf_{N \rightarrow +\infty} x_{1,N}^{-}  < \limsup_{N \rightarrow +\infty} x_{1,N}^{+} < \liminf_{N \rightarrow +\infty} x_{2,N}^{-}.
	\notag
        \end{align}
\end{assumption} 
Using theorems \ref{theorem:no_eig} and \ref{theorem:loc_eig}, we deduce that if $t_1^{-}, t_1^+, t_2^{-}, t_2^+$ are real numbers independent of $N$ satisfying 
\begin{align}
        \label{eq:inegalites-t1t2}
        0  
        < t_1^{-} 
        < \liminf_{N \rightarrow +\infty} x_{1,N}^{-}  
        <  \limsup_{N \rightarrow +\infty} x_{1,N}^{+} 
        < \, t_1^{+} \, < \, t_2^{-} 
        < \liminf_{N \rightarrow +\infty} x_{2,N}^{-}
        \leq \limsup_{N \rightarrow +\infty} x_{Q_N,N}^{+}
        < t_2^+
\end{align}
then, almost surely, for $N$ large enough, it holds that 
\begin{equation}
        \label{eq:separation}
        0 < t_1^{-} < \hat{\lambda}_{1,N} < \ldots < \hat{\lambda}_{M-K,N} < t_1^{+} < t_2^{-} < \hat{\lambda}_{M-K+1,N} < \ldots <  \hat{\lambda}_{M,N} < t_2^+.
\end{equation}
Assumptions \ref{ass:separation1} and \ref{ass:separation2} thus imply that, almost surely, the smallest $M-K$ eigenvalues of 
$\Sigmabs_N \Sigmabs_N^*$ are separated from the $K$ greatest ones for $N$ large enough in the sense that the 2 sets of eigenvalues 
are included into 2 disjoint intervals that do not depend on $N$. It is interesting to remark that Assumptions \ref{ass:separation1} and 
\ref{ass:separation2} are "deterministic conditions" depending only on $\sigma^{2}, c_N = \frac{M}{N}$,and on the eigenvalues 
of $\B_N \B_N^*$. If $K$ remains fixed, recent results of Benaych-Rao \cite{benaych2011singular} (see also \cite{loubaton2010ejpsub}) imply that Assumptions 
\textbf{A-5} and \textbf{A-6} hold if and only if $\liminf_{N \rightarrow +\infty} \lambda_{M-K+1,N} > \sigma^{2} \sqrt{c}$. 
If however $K$ scales with $N$, the derivation of more explicit conditions equivalent to Assumptions \ref{ass:separation1} and \ref{ass:separation2} is still an open problem.

We are now in position to present the consistent estimator of $\eta_N$ proposed in \cite{vallet2010sub}. It is based on the observation that 
\begin{equation}
        \Pibs_N = \frac{1}{2 i \pi} \int_{{\cal C}^{-}} \left(\B_N \B_N^* - \lambda \I_M \right)^{-1}  \drm \lambda,
        \notag
\end{equation}
where ${\cal  C}$ represents a contour enclosing $0$ and not the strictly positive eigenvalues of $\B_N \B_N^*$, and the symbol 
${\cal C}^{-}$ means that the contour is oriented clockwise. The estimator of \cite{vallet2010sub} is based on the observation 
that under Assumptions  \ref{ass:separation1} and \ref{ass:separation2}, function $w_N(z)$ provides such a contour for $N$ large enough. 
In the following, for $y > 0$ and $\epsilon > 0$, $\epsilon < \frac{y}{3}$ small enough, we consider the rectangle $\Rcal_y$ defined by 
\begin{equation}
        \label{eq:def-Ry}
        \Rcal_y = \{ z = x + i v, 0 < t_1^{-} - 3 \epsilon \leq x \leq  t_1^{+} + 3 \epsilon < t_2^{-} - 3 \epsilon, -y \leq v \leq y \}
\end{equation}
and its boundary $\partial \Rcal_y$. 
Then, the properties of function $w_N(z)$ (see Proposition \ref{prop:w}) imply that for $N$ large enough, the set $w_N(\partial \Rcal_y)$ is a 
contour enclosing the origin, but not the other eigenvalues of $\B_N \B_N^*$. Therefore, $\Pibs_N$ can also be written as
$$
	\Pibs_N = \frac{1}{2 i \pi} \int_{{\partial \Rcal_y}^{-}} \left( \B_N \B_N^* - w_N(z) \I_M \right)^{-1} \; w_N'(z) \drm z
$$
or equivalently
\begin{equation}
        \Pibs_N = \frac{1}{2 i \pi} \int_{{\partial \Rcal_y}^{-}} {\bf T}_N(z)  \; \frac{w_N'(z)}{1+\sigma^{2} c_N m_N(z)} \drm z
\end{equation}
because $( \B_N \B_N^* - w_N(z) \I_M)^{-1} = \frac{{\bf T}_N(z)}{1+ \sigma^{2} c_N m_N(z)}$. 
Using \eqref{eq:convergence-hatmN} and \eqref{eq:convergence-resolvente} as well as the following lemma
\begin{lemma}
        \label{le:localization-omega}
        Almost surely, for $N$ large enough, the $M$ solutions  $(\hat{\omega}_{k,N})_{k=1, \ldots, M}$ of the equation $1 + \sigma^{2} c_N m_N(z) = 0$ satisfy
        \begin{align}
	t_1^{-} 
	< \hat{\lambda}_{1,N} 
	< \hat{\omega}_{1,N} 
	< \ldots 
	< \hat{\lambda}_{M-K,N} 
	< \hat{\omega}_{M-K,N} 
	< t_{1}^+ 
	< t_2^-
	< \hat{\lambda}_{M-K+1,N} 
	< \hat{\omega}_{M-K+1,N} 
	< \ldots 
	< \hat{\lambda}_{M,N} 
	< \hat{\omega}_{M,N} 
	< t_{2}^+
	\notag
        \end{align}
\end{lemma}
It is showed in \cite{vallet2010sub} that matrix $\tilde{\Pibs}_N$ defined by 
\begin{equation}
        \label{eq:def-tildePi}
        \tilde{\Pibs}_N = \frac{1}{2 i \pi} \int_{{\partial \Rcal_y}^{-}} \Q_N(z)  \frac{\hat{w}_N'(z)}{1+\sigma^{2} c_N \hat{m}_N(z)} \drm z
\end{equation}
where $\hat{w}_N(z) = z (1 + \sigma^{2} c_N \hat{m}_N(z))^{2} - \sigma^{2} (1 - c_N) (1 + \sigma^{2} c_N \hat{m}_N(z))$, satisfies
$\a_N^{*} \tilde{\Pibs}_N \a_N - \a_N^{*} \Pibs_N \a_N \rightarrow 0$ almost surely. 
We note that the poles of the integrand of the righthandside (r.h.s) of \eqref{eq:def-tildePi} coincide with the set 
$\{\hat{\lambda}_{k,N}, \hat{\omega}_{k,N}: k=1, \ldots,M \}$, 
which by \eqref{eq:separation} and Lemma \ref{le:localization-omega}, verifies 
\begin{equation}
        \label{eq:distance}
        \dist\left(\partial \Rcal_y, \{ (\hat{\lambda}_{k,N}, \hat{\omega}_{k,N})_{k=1, \ldots, M} \}\right) > 3 \epsilon
\end{equation}
almost surely for $N$ large enough. 
In pratice, the above estimator is quite easy to implement because, as the localization of the poles of the integrand in \eqref{eq:def-tildePi}  w.r.t. 
the contour $\partial \Rcal_y$ is known (see lemma \ref{le:localization-omega}), the contour integral in \eqref{eq:def-tildePi} can be solved, and expressed 
in closed form in terms of the $(\hat{{\bf u}}_{k,N}, \hat{\lambda}_{k,N}, \hat{\omega}_{k,N})_{k=1, \ldots, M}$.

\section{\texorpdfstring{Statement and proof of  the uniform consistency of estimate $\tilde{\eta}_N(\theta)$}{Statement and proof of the uniform consistency}}
\label{sec:sup}

From now on, we assume that vector $\a(\theta)$ is given by \eqref{eq:modele-a} and that assumptions \ref{ass:separation1} and \ref{ass:separation2} hold. 
We consider $t_1^{-}$, $t_1^{+}$, $t_2^{-}$ and $t_2^{+}$ satisfying \eqref{eq:inegalites-t1t2} as well a rectangle $\Rcal_y$ defined by \eqref{eq:def-Ry}.
We prove here the following result.
\begin{theorem}
	\label{theorem:uniform_consistency}
	Assume assumptions \textbf{A-1} to \textbf{A-6} hold. Then, we have
	\begin{align}
		\sup_{\theta \in [-\pi,\pi]} \left|\tilde{\eta}_N(\theta) - \eta_N(\theta)\right| \xrightarrow[N \to \infty]{} 0.
		\notag
	\end{align}
	with probability one.
\end{theorem}
In order to prove theorem \ref{theorem:uniform_consistency}, we show that it is sufficient to establish that for each $\alpha > 0$ and for each $\theta \in [-\pi, \pi]$,  
$\Prob(| \a(\theta)^* (\tilde{\Pibs}_N - \Pibs_N) \a(\theta) | > \alpha)$ decreases fast enough torwards $0$. For this, a tempting choice is to use the Markov 
inequality, and to establish that the moments of  $\a(\theta)^* (\tilde{\Pibs}_N - \Pibs_N) \a(\theta)$
decrease fast enough. 
However, the observation that (\ref{eq:distance}) holds for $N$ greater than a random integer does not necessarily imply the existence of the moments of  
$\a(\theta)^* \tilde{\Pibs}_N \a(\theta)$. 
In order to solve this technical problem, we establish that the probability that at least one element of 
$\{\hat{\lambda}_{k,N}, \hat{\omega}_{k,N}: k=1, \ldots,M \}$ escapes from 
$[t_{1}^{-} -  2 \epsilon, t_{1}^{+} +  2 \epsilon] \cup [t_{2}^{-} -  2 \epsilon, t_{2}^{+} +  2 \epsilon]$ 
decreases at rate  $\frac{1}{N^{l}}$ for any $l \in \mathbb{N}$, and prove that the moments of a convenient regularized version of 
$\a(\theta)^* (\tilde{\Pibs}_N - \Pibs_N) \a(\theta)$ converge fast enough torwards 0.

In the following,  we denote by $\Tcal_{\epsilon}$ the set  
\begin{equation}
        \Tcal_{\epsilon} = [t_{1}^{-} -  \epsilon, t_{1}^{+} +  \epsilon] \cup [t_{2}^{-} -  \epsilon, t_{2}^{+} +  \epsilon],
        \notag
\end{equation}
We first establish in Sections \ref{subsec:escape-lambda} and \ref{subsec:escape-omega} that the events $\Ecal_{1,N}$ and $\Ecal_{2,N}$  defined by 
\begin{align}
        \label{eq:def-E1}
        \Ecal_{1,N} = \{ \mbox{at least one of the $(\hat{\lambda}_{k,N})_{k=1, \ldots, M}$ escapes from $\Tcal_{\epsilon}$} \},
        \\
        \label{eq:def-E2}
        \Ecal_{2,N} = \{ \mbox{at least one of the $(\hat{\omega}_{k,N})_{k=1, \ldots, M}$ escapes from $\Tcal_{\epsilon}$} \}.
\end{align}
verify $\Prob\left(\Ecal_{i,N}\right) = \Ocal\left(\frac{1}{N^{l}}\right)$ for each $l \in \mathbb{N}$. 
Using this result, we introduce in Section \ref{subsec:end-of-proof} the regularization term, denoted $\chi_N$, defined as follows. 
We consider a function $\phi \in \Ccal_c^{\infty}(\Rbb,\Rbb^+)$ satisfying
\begin{align}
	\phi(\lambda) =
	\begin{cases}
		1 & \text{ for } \lambda \in \Tcal_{\epsilon}  \\
		0 & \text{ for } \lambda \in \Rbb\backslash\left( [t_1^- - 2\epsilon,t_1^+ + 2\epsilon] \cup  [t_2^- -2\epsilon,t_2^+ + 2\epsilon] \right)
	\end{cases}
\end{align}
and $\phi(\lambda) \in (0,1)$ elsewhere, and define the random variable
\begin{equation}
        \label{eq:def-chii}
        \chi_N = \det \phi(\Sigmabs_N\Sigmabs_N^*)\det\phi\left(\hat{\Omegabs}_N\right),
\end{equation}
which verifies $\mathbb{1}_{\Ecal_N^c} \leq \chi_N$ where $\Ecal_N =  \Ecal_{1,N} \cup \Ecal_{2,N} $. 
We will prove that, considered as a function of the real and imaginary part of the entries of $\W_N$, $\chi_N$ is a $\Ccal^1$ function, and using 
Poincaré inequality, we will establish that 
\begin{equation}
        \Ebb \left| \a(\theta)^* (\tilde{\Pibs}_N - \Pibs_N) \a(\theta) \chi_N \right|^{2l} = \Ocal\left(\frac{1}{N^{l}}\right),
        \notag
\end{equation}
for each integer $l$. The above mentioned properties eventually allow to prove the uniform consistency of estimator $\tilde{\eta}_N(\theta)$.

\subsection
{
        \texorpdfstring
        {Evaluation of the escape probability of $(\hat{\lambda}_{k,N})_{k=1, \ldots, N}$}
        {Evaluation of the escape probability of the eigenvalues}
}
\label{subsec:escape-lambda}

The purpose of this section is to prove the following technical result. 
\begin{proposition} 
        \label{lemma:escape-lambda}
        Under assumptions \textbf{A-1}-\textbf{A-6}, for each $l \in \mathbb{N}$, it holds that 
        \begin{equation}
		\Prob(\Ecal_{1,N}) = \Ocal\left(\frac{1}{N^{l}}\right).
		\notag
        \end{equation}
\end{proposition}
To prove this result, we consider a function $\psi_{0} \in \Ccal^\infty(\Rbb,\Rbb^+)$ such that
\begin{align}
        \label{eq:def-psi0}
        \psi_0(\lambda) =
        \begin{cases}
	1 & \text{ for } \lambda \in \Tcal_{\epsilon}^c, \\
	0 & \text{ for } \lambda \in [t_1^-,t_1^+] \bigcup [t_2^-,t_2^+].
        \end{cases}
\end{align}
and $\psi_0(\lambda) \in (0,1)$ elsewhere. From this definition, we clearly have
\begin{align}
	\Prob(\Ecal_{1,N})
	\leq 
	\Prob\left(\Tr \psi_{0}(\Sigmabs_N\Sigmabs_N^*) \geq 1\right)
	\leq
	\Ebb\left[\left(\Tr \psi_{0}(\Sigmabs_N\Sigmabs_N^*) \right)^{2l} \right]
	\notag
\end{align}
for $l \in \Nbb$. 
In order to establish Proposition \ref{lemma:escape-lambda}, it is therefore sufficient to prove that 
$\Ebb\left[\left(\Tr \psi_0(\Sigmabs_N\Sigmabs_N^*) \right)^{2l} \right] = \Ocal\left(\frac{1}{N^{2l}}\right)$
for each integer $l$ which is the object of the next lemma.
\begin{lemma}
        \label{lemma:utile}
        Assume assumptions \textbf{A-1} to \textbf{A-6} hold. 
        Then, for all function $\psi \in \Ccal^\infty(\Rbb,\Rbb)$ constant over the complementary of a compact interval and 
        which vanishes on the support $\Scal_N$ of $\mu_N$ for all $N$ large enough, it holds that
        \begin{equation}
	\label{eq:moment-trace-phi}
		\Ebb\left[\left(\Tr \psi(\Sigmabs_N\Sigmabs_N^*) \right)^{2l} \right] = \Ocal\left(\frac{1}{N^{2l}}\right)
        \end{equation}
        for each $l \in \mathbb{N}$. 
\end{lemma}
\begin{proof}
        We prove Lemma \ref{lemma:utile} by induction on $l$. We first consider the case $l=1$, and consider a function $\psi$ as above, and denote
        by $C$ the constant value taken by $\psi$ over the complementary of a certain compact interval. We follow 
        \cite{haagerup2005new} and write $\psi$ as $\psi = \tilde{\psi} + C$, where $\tilde{\psi} \in \Ccal_c^\infty(\Rbb,\Rbb)$, 
        and verifies $\tilde{\psi} = -C$ over $\Scal_N$ for $N$ large enough. 
        Using the technique developed in \cite{haagerup2005new} based on \eqref{eq:inegalite-haagerup} and Poincaré inequality, we have
        \begin{align}
		\Var\left[\Tr \psi(\Sigmabs_N\Sigmabs_N^*)\right] &= \Var\left[\Tr \tilde{\psi}(\Sigmabs_N\Sigmabs_N^*)\right] 
		&= \Ocal\left(\frac{1}{N^2}\right),
		\notag\\
                \Exp\left[\Tr \tilde{\psi}(\Sigmabs_N\Sigmabs_N^*)\right] 
                &=  
                M \int_{\Rbb} \tilde{\psi}(\lambda) \drm\mu_N(\lambda) +  \Ocal\left(\frac{1}{N}\right) 
                = -M C + \Ocal\left(\frac{1}{N}\right). 
                \notag
        \end{align}
        As $ \Exp\left[\Tr \psi(\Sigmabs_N\Sigmabs_N^*)\right] = CM + \Exp\left[\Tr \tilde{\psi}(\Sigmabs_N\Sigmabs_N^*)\right]$, this leads to 
        $\Exp\left[\Tr \psi(\Sigmabs_N\Sigmabs_N^*)\right] =  \Ocal\left(\frac{1}{N}\right)$. 
        As  
        \begin{align}
		\Exp\left[\left(\Tr \psi(\Sigmabs_N\Sigmabs_N^*)\right)^2\right]
		=  
		\left( \Exp\left[\Tr \psi(\Sigmabs_N\Sigmabs_N^*)\right] \right)^{2} + \Var\left[\Tr \psi(\Sigmabs_N\Sigmabs_N^*)\right]
        \end{align}
        we finally obtain that (\ref{eq:moment-trace-phi}) holds for $l=1$. 

        We now assume that (\ref{eq:moment-trace-phi}) holds until the order $l-1$
        for each function of $\Ccal^\infty(\Rbb,\Rbb)$ vanishing on $\Scal_N$ for $N$ large enough and constant over the complementary of a compact interval. We consider such a function $\psi$ and evaluate the behaviour of the $2l$--th order moment of $\Tr \psi(\Sigmabs_N\Sigmabs_N^*)$. We have
        \begin{align}
		\Ebb\left[\left(\Tr \psi(\Sigmabs_N\Sigmabs_N^*) \right)^{2l} \right] 
		&=
		\left(\Ebb\left[\left(\Tr \psi(\Sigmabs_N\Sigmabs_N^*) \right)^{l} \right]\right)^2
		+
		\Var\left[\left(\Tr \psi(\Sigmabs_N\Sigmabs_N^*) \right)^{l}\right].
		\label{eq:exp_tr_psi_pow_2l}
        \end{align}
        The first term of the r.h.s of (\ref{eq:exp_tr_psi_pow_2l}) can be upperbounded as follows
        \begin{align}
		\left(\Ebb\left[\left(\Tr \psi(\Sigmabs_N\Sigmabs_N^*) \right)^{l} \right]\right)^2
		\leq
		\Ebb\left[\left(\Tr \psi(\Sigmabs_N\Sigmabs_N^*)\right)^2 \right]
		\Ebb\left[\left(\Tr \psi(\Sigmabs_N\Sigmabs_N^*)\right)^{2(l-1)}\right]
		=
		\Ocal\left(\frac{1}{N^{2l}}\right),
		\notag
        \end{align} 
        using that (\ref{eq:moment-trace-phi}) holds until the order $l-1$. The second term of the righthandside of \eqref{eq:exp_tr_psi_pow_2l} can be evaluated using 
        the Poincaré inequality. 
        Using that the partial derivative of $\Tr \psi(\Sigmabs_N \Sigmabs_N^*)$ w.r.t. $W_{i,j,N}$ and $\overline{W}_{i,j,N}$ are equal respectively to 
        ${\bf e}_j^{T} \Sigmabs_N^* \psi'(\Sigmabs_N \Sigmabs_N^*) {\bf e}_i$ and ${\bf e}_i^{T} \psi'(\Sigmabs_N \Sigmabs_N^*) \Sigmabs_N {\bf e}_j$,
        we immediately obtain that 
        \begin{align}
		\Var\left[\left(\Tr \psi(\Sigmabs_N\Sigmabs_N^*) \right)^{l}\right]
		&\leq C
		\Ebb
		\left[
			\frac{1}{N} \Tr \left(\psi'(\Sigmabs_N\Sigmabs_N^*)^2 \Sigmabs_N\Sigmabs_N^*\right)
			\left(\Tr \psi(\Sigmabs_N\Sigmabs_N^*) \right)^{2l-2}
		\right].
		\notag
        \end{align}
        Using Hölder's inequality, we get immediately that 
        \begin{align}
		\Ebb
		\left[
			\frac{1}{N} \Tr \left(\psi'(\Sigmabs_N\Sigmabs_N^*)^2 \Sigmabs_N\Sigmabs_N^*\right)
			\left(\Tr \psi(\Sigmabs_N\Sigmabs_N^*) \right)^{2l-2}
		\right] 
		\leq
		C\left(\Ebb\left|\frac{1}{N} \Tr \left(\psi'(\Sigmabs_N\Sigmabs_N^*)^2 \Sigmabs_N\Sigmabs_N^*\right)\right|^l\right)^{\frac{1}{l}}
		\left(\Ebb\left[\left(\Tr \psi(\Sigmabs_N\Sigmabs_N^*) \right)^{2l}\right]\right)^{\frac{l-1}{l}}.
		\label{eq:Var_tr_psi_pow_l}
        \end{align}
        Since the function $\lambda \to \psi'(\lambda)^2\lambda$ belongs to $\Ccal_c^\infty(\Rbb,\Rbb)$ and has a support disjoint from $\Scal_N$ 
        for $N$ large enough, it holds that 
        \begin{align}
	\Ebb\left|\frac{1}{N} \Tr \left(\left[\psi'(\Sigmabs_N\Sigmabs_N^*)\right]^2 \Sigmabs_N\Sigmabs_N^*\right)\right|^l 
	&\leq
	\sqrt{\Ebb\left|\frac{1}{N} \Tr \left(\left[\psi'(\Sigmabs_N\Sigmabs_N^*)\right]^2 \Sigmabs_N\Sigmabs_N^*\right)\right|^2}
	\sqrt{\Ebb\left|\frac{1}{N} \Tr \left(\left[\psi'(\Sigmabs_N\Sigmabs_N^*)\right]^2 \Sigmabs_N\Sigmabs_N^*\right)\right|^{2(l-1)}}
	\notag\\
	&= \Ocal\left(\frac{1}{N^{2l}}\right).
	\notag
        \end{align}
        Plugging the previous estimates into \eqref{eq:Var_tr_psi_pow_l}, we get
        \begin{align}
	\Var\left[\left(\Tr \psi(\Sigmabs_N\Sigmabs_N^*) \right)^{l}\right]
	\leq
	\frac{C}{N^2} \left(\Ebb\left[\left(\Tr \psi(\Sigmabs_N\Sigmabs_N^*) \right)^{2l}\right]\right)^{\frac{l-1}{l}}.
	\notag
        \end{align}
        Define $x_N =\Ebb\left[\left(\Tr \psi(\Sigmabs_N\Sigmabs_N^*) \right)^{2l}\right]$ and $u_N = N^{2l} x_N$. 
        From \eqref{eq:exp_tr_psi_pow_2l}, we have the inequalities 
        $x_N \leq \frac{C_1}{N^2} x_N^{\frac{l-1}{l}} + \frac{C_2}{N^{2 l}}$ and $u_N \leq C_1 u_N^{\frac{l-1}{l}} + C_2$. 
        We claim that the sequence $(u_N)$ is bounded. 
        If this is not the case, it exists a subsequence $u_{k_N}$ extracted from $u_N$ which converges torwards $+\infty$. 
        However, the inequality $\frac{C_1}{u_{k_N}^{1/l}} + \frac{C_2}{u_{k_N}} \geq 1$ must holds for $N$ large enough. As  $u_{k_N} \rightarrow +\infty$, 
        this leads to a contradiction. Therefore, $u_N$ is bounded and $x_N \leq \frac{C}{N^{2l}}$ for $N$ large enough. This proves Lemma \ref{lemma:utile}. 
\end{proof}

\subsection
{
        \texorpdfstring
        {Evaluation of the escape probability of the $(\hat{\omega}_{k,N})_{k=1, \ldots, N}$}
        {Evaluation of the escape probability of the omega}
}
\label{subsec:escape-omega}

In this section, we will prove the following result. 
\begin{proposition} 
        \label{lemma:escape-omega}
        Assume assumptions \textbf{A-1} to \textbf{A-6} hold. For each $l \in \mathbb{N}$, it holds that 
        \begin{equation}
	\Prob(\Ecal_{2,N}) = \Ocal\left(\frac{1}{N^{l}}\right).
	\notag
        \end{equation}
\end{proposition}
We follow the same approach than in Section \ref{subsec:escape-lambda} and first prove that the $(\hat{\omega}_{k,N})_{k=1, \ldots, M}$ satisfy a property similar to \eqref{eq:inegalite-haagerup}. For this,
we study the behaviour of the Stieltjes transform $\hat{n}_N(z)$ of the distribution 
$\frac{1}{M} \sum_{k=1}^{M} \delta_{\hat{\omega}_{k,N}}$ defined by
\begin{align}
	\hat{n}_N(z) = \frac{1}{M} \Tr\left( \hat{\Omegabs}_N - z \I \right)^{-1}.
	\notag
\end{align}
and use Lemma \ref{lemma:haagerup_capitaine} as well as the inverse Stieltjes transform formula (\ref{eq:inverse-stieltjes-psi}). Our starting point is the following result showing that the empirical eigenvalue distribution of $\hat{\Omegabs}_N$ is very similar to the
distribution of the eigenvalues of $\Sigmabs_N \Sigmabs_N^*$. 
The following auxiliary result will be useful.
\begin{lemma}
        \label{le:expre-Ehatm'N}
        Assume assumptions \textbf{A-1} to \textbf{A-6} hold. It holds that
        \begin{equation}
	\label{eq:expre-Ehatm'N}
	\Ebb\left[\hat{m}'_N(z)\right] - m_N'(z) = \frac{t_N(z)}{N^{2}}
        \end{equation}
	where $t_N$ is analytic on $\Cbb\backslash\Rbb$ and can be upperbounded by $\Prm_1(|z|)  \Prm_2\left(\frac{1}{|\Im(z)|}\right)$ on $\Cbb\backslash\Rbb$. 
\end{lemma}
\begin{proof}
	The proof is given in Appendix \ref{subsec:proof-derivee}.
\end{proof}
We now prove the fundamental following result.
\begin{lemma}
        \label{le:fund}
        Assume assumptions \textbf{A-1} to \textbf{A-6} hold. For each $z \in \Cbb\backslash\Rbb$,
        \begin{align}
	\Ebb\left[\hat{n}_N(z)\right] = 
	\int_{\Scal_N} \frac{ \drm \mu_N(\lambda)}{\lambda - z} + \frac{1}{M}  \int_{\Scal_N} \frac{ d\kappa_N(\lambda)}{\lambda - z} + \frac{r_N(z)}{N^{2}},
	\notag
        \end{align}
        with $\kappa_N$ a finite signed measure carried by $\Scal_N$ such that $\kappa_N\left([x_{q,N}^{-}, x_{q,N}^{+}]\right) = 0$ for $q=1, \ldots, Q_N$, 
        and $r_N$ a holomorphic function on $\Cbb\backslash\Rbb$ satisfying
        \begin{align}
	\left|r_N(z)\right| \leq  \Prm_1(|z|) \Prm_2\left(\frac{1}{|\mathrm{Im}(z)|}\right),
	\notag
        \end{align}
        with $\Prm_1$, $\Prm_2$ two polynomials with positive coefficients independent of $N$.
\end{lemma}
\begin{proof}
        Using that $\hat{\Omegabs}_N$ is a rank 1 perturbation of $\hat{\Lambdabs}_N$, we obtain immediately that 
        \begin{align}
	\hat{n}_N(z) = \hat{m}_N(z) - \frac{1}{M} \frac{\sigma^{2} c_N \hat{m}_N'(z)}{1 + \sigma^{2} c_N \hat{m}_N(z)}.
	\notag
        \end{align}
        Therefore, for $z \in \Cbb\backslash\Rbb$, it holds that
        \begin{align}
	\label{eq:expre-Enhat}
	\Ebb\left[\hat{n}_N(z)\right] = 
	\Ebb\left[\hat{m}_N(z)\right] -\frac{1}{M} \Ebb\left[\frac{\sigma^{2} c_N \hat{m}_N'(z)}{1 + \sigma^{2} c_N \hat{m}_N(z)} \right].
        \end{align}
        We first establish that 
        \begin{equation}
		\label{eq:terme-supplementaire}
		\Ebb \left[\frac{\sigma^{2} c_N \hat{m}_N'(z)}{1 + \sigma^{2} c_N \hat{m}_N(z)} \right] = 
		\frac{\sigma^{2} c_N m_N'(z)}{1 + \sigma^{2} c_N m_N(z)} + \frac{c_N}{N} r_{N}(z), 
        \end{equation}
        where $r_N(z)$ is holomorphic on $\Cbb\backslash\Rbb$ and satisfies $|r_N(z)| \leq \Prm_1(|z|) \Prm_2\left(\frac{1}{|\Im(z)|}\right)$. 
        For this, we write 
        \begin{align}
		&\frac{\sigma^{2} c_N \hat{m}_N'(z)}{1 + \sigma^{2} c_N \hat{m}_N(z)} - \frac{\sigma^{2} c_N m_N'(z)}{1 + \sigma^{2} c_N m_N(z)} 
		=
		\notag \\
		&\qquad
		\frac{\sigma^{2} c_N (\hat{m}'_N(z) - m_N'(z))}{(1 +  \sigma^{2} c_N \hat{m}_N(z))(1 + \sigma^{2} c_N m_N(z))}
		+\frac{(\sigma^{2} c_N)^{2} \left( m_N(z) (\hat{m}_N'(z) - m_N'(z)) + m_N'(z) (m_N(z) - \hat{m}_N(z)) \right)}
		{(1 +  \sigma^{2} c_N \hat{m}_N(z))(1 + \sigma^{2} c_N m_N(z))}.
		\label{eq:expre-difference}
        \end{align}
        In order to study the expectation of this expression, we use \eqref{eq:majoration1surbC+} and \eqref{eq:borne-bhat-1}.
        Moreover, \eqref{eq:expre-Ehatmn} and a straightforward application of the Poincaré inequality to $\hat{m}_N(z)$ considered for $z$ fixed as a function of the entries of $\W_N$ 
        leads immediately to
        \begin{equation}
		\Ebb\left|\hat{m}_N(z) - m_N(z)\right|^{2} \leq \frac{1}{N^{2}} \Prm_1\left(|z|\right) \Prm_2\left(\frac{1}{|\Im(z)|}\right),
		\notag
        \end{equation}
        for some polynomials $\Prm_1$, $\Prm_2$ with positive coefficients and independent of $N$. Therefore,
        \begin{align}
	\Ebb\left|\hat{m}_N(z) - m_N(z)\right| \leq  \frac{1}{N} \left(\Prm_1\left(|z|\right) + \Prm_2\left(\frac{1}{|\Im(z)|}\right)\right).
	\notag
        \end{align}
        Applying also Poincaré inequality to bound $\Var[\hat{m}_N'(z)]$, together with Lemma \ref{le:expre-Ehatm'N}, we get
        \begin{equation}
	\Ebb\left|\hat{m}'_N(z) - m'_N(z)\right|^{2} \leq \frac{1}{N^{2}} \Prm_1\left(|z|\right) \Prm_2\left(\frac{1}{|\Im(z)|}\right).
	\notag
        \end{equation}
        Therefore, it holds that 
        \begin{align}
	\Ebb\left|\hat{m}'_N(z) - m'_N(z)\right| \leq  \frac{1}{N} \left(\Prm_1\left(|z|\right) + \Prm_2\left(\frac{1}{|\Im(z)|}\right)\right).
	\notag
        \end{align}        
        Using $|m_N(z)| \leq \frac{1}{|\Im(z)|}$, $|m_N'(z)| \leq \frac{1}{|\Im(z)|^{2}}$, as well as \eqref{eq:majoration1surbC+} and \eqref{eq:borne-bhat-1}, we eventually get 
        from \eqref{eq:expre-difference} that 
        \begin{align}
		\Ebb \left| \frac{\sigma^{2} c_N \hat{m}_N'(z)}{1 + \sigma^{2} c_N \hat{m}_N(z)} - \frac{\sigma^{2} c_N m_N'(z)}{1 + \sigma^{2} c_N m_N(z)} \right| 
		\leq  
		\frac{1}{N} \Prm_1\left(|z|\right) \Prm_2\left(\frac{1}{|\Im(z)|}\right).
		\notag
        \end{align}
        This immediately implies \eqref{eq:terme-supplementaire}. Now define the function $h_N(z)$ by
        \begin{align}
		h_N(z) = \frac{ \sigma^{2} c_N m_N'(z)}{1+ \sigma^{2} c_N m_N(z)}.
		\notag
        \end{align}
        This function coincides with the Stieltjes transform of a signed measure $\kappa_N$ satisfying the conditions of Lemma \ref{le:fund}: 
        Using \eqref{eq:inegalite-Reb}, we obtain that $|h_N(z)| \leq 2 \sigma^{2} c_N |m_N'(z)|$. As
        $|m_N'(z)| \leq  \frac{1}{\mathrm{dist}(z,\Kcal)^{2}}$ where ${\cal K}$ is a compact containing $\Scal_N$, it holds that 
        $|h_N(z)| \leq C \frac{1}{\mathrm{dist}(z, \Kcal)^{2}}$. 
        Using Theorem 4.3 in \cite{capitaine2009largest}, we obtain that $h_N(z)$ is the Stieltjes transform of a finite signed measure $\kappa_N$, 
        the support of which is the set of singular points of $h_N(z)$, i.e. $\Scal_N$. 
        In order to evaluate $\kappa_N([x_{q,N}^{-}, x_{q,N}^{+}])$, we use the inverse Stieltjes transform formula,
        \begin{equation}
		\kappa_N([x_{q,N}^{-}, x_{q,N}^{+}]) = \lim_{y \downarrow 0} \Im \left( \int_{[x_{q,N}^{-}, x_{q,N}^{+}]} h_N(x+iy) \drm x  \right).
		\notag
        \end{equation}
        It is clear that $h_N(x+iy) = \frac{\partial  \log (1 + \sigma^{2} c_N m_N(x+iy))}{\partial x}$, where the complex logarithm corresponds to the principal 
        determination defined on $\Cbb\backslash\Rbb^{-}$. We note that  \eqref{eq:inegalite-Reb} justifies the use of the principal determination. Therefore, 
        \begin{align}
		\int_{[x_{q,N}^{-}, x_{q,N}^{+}]} h_N(x+iy)  \drm x = 
		\log \left(1 + \sigma^{2} c_N m_N(x_{q,N}^{+} + iy) \right) - \log \left( 1 + \sigma^{2} c_N m_N(x_{q,N}^{-} + iy) \right).
		\notag
        \end{align}
        When $y \rightarrow 0$, this converges towards $\log(1 + \sigma^{2} c_N m_N(x_{q,N}^{+})) - \log(1 + \sigma^{2} c_N m_N(x_{q,N}^{-}))$, a real quantity
        because $x_{q,N}^{-}$ and $x_{q,N}^{+}$ belong to $\partial \Scal_N$. This shows that $\kappa_N([x_{q,N}^{-}, x_{q,N}^{+}]) = 0$. 
        Consequently,
        \begin{align}
		\Ebb \left[\frac{\sigma^{2} c_N \hat{m}_N'(z)}{1 + \sigma^{2} c_N \hat{m}_N(z)} \right] 
		= 
		\int_{\Scal_N} \frac{d \kappa_N(\lambda)}{\lambda - z} + \frac{c_N r_N(z)}{N},
		\notag
        \end{align}
        where $r_N(z)$ is holomorphic on $\Cbb\backslash\Rbb$ such that $|r_N(z)| \leq \Prm_1(|z|) \Prm_2(\frac{1}{|\Im(z)|})$. 
        Lemma \ref{le:fund} follows immediately from \eqref{eq:expre-Enhat}. 
\end{proof}
We now handle the proof of Proposition \ref{lemma:escape-omega}.
Although certain steps of the present proof are similar to the proof of Proposition \ref{lemma:escape-lambda}, more work is needed because matrix 
$\hat{\Omegabs}_N$ considered as a function of the entries of $\W_N$ is  more complicated than $\Sigmabs_N \Sigmabs_N^*$. We still consider function 
$\psi_{0} \in \Ccal^\infty(\Rbb,\Rbb^+)$ defined by (\ref{eq:def-psi0}) and remark that
\begin{align}
	\Prob(\Ecal_{2,N})
	\leq 
	\Prob\left(\Tr \psi_{0}(\hat{\Omegabs}_N) \geq 1\right)
	\leq
	\Ebb\left[\left(\Tr \psi_{0}(\hat{\Omegabs}_N) \right)^{2l} \right]
	\notag
\end{align}
for $l \in \Nbb$. In order to establish Proposition \ref{lemma:escape-omega}, it is therefore sufficient to prove that 
$\Ebb\left[\left(\Tr \psi_0(\hat{\Omegabs}_N) \right)^{2l} \right] = \Ocal\left(\frac{1}{N^{2l}}\right)$ for each integer $l$. 
For this, we still use the Poincaré inequality. 
However, in contrast with the context of Proposition \ref{lemma:escape-lambda}, the entries of $\hat{\Omegabs}_N$, considered as functions of the real and imaginary 
parts of the entries of $\W_N$, are not continuously differentiable on $\Rbb^{2MN}$ because function $\W_N \rightarrow \hat{\lambda}_{k,N}$ is not differentiable at 
points for which eigenvalue $\hat{\lambda}_{k,N}$ is multiple. 
The use of Poincaré inequality has therefore to be justified carefully. The following useful lemma is proved in the appendix. 
\begin{lemma}
        \label{le:differentiabilite-trpsihatomega}
        Assume assumptions \textbf{A-1} to \textbf{A-6} hold. 
        Let $\tilde{\psi}$ be a function of $\Ccal^{\infty}_c(\Rbb, \Rbb)$. Then, $\Tr \tilde{\psi}(\hat{\Omegabs}_N)$, 
        considered as a function of the real and imaginary parts of the entries of $\W_N$, is continuously differentiable. Moreover, 
        if the eigenvalues of $\Sigmabs_N \Sigmabs_N^{*}$ have multiplicity 1, it holds that 
	\begin{equation}
		\frac{\partial}{\partial W_{i,j,N}} \left\{\frac{1}{M} \mathrm{Tr}\left(\tilde{\psi}(\hat{\Omegabs}_N)\right)\right\}
		=
		\frac{1}{M} \left[\Sigmabs_N^{*} \sum_{l=1}^{M}  [\tilde{\psi}'(\hat{\Omegabs}_N)]_{l,l} \hat{\Pibs}_{l,N} \right]_{j,i}.
		\label{eq:expre-deriveetracepsi}
	\end{equation}
	where $\hat{\Pibs}_{l,N}$ represents the orthogonal projection matrix on the 1--dimensionnal eigenspace associated to the eigenvalue 
	$\hat{\lambda}_{l,N}$ of  $\Sigmabs_N \Sigmabs_N^{*}$.
\end{lemma}
We will also need that
\begin{lemma}
        \label{le:normeWW*}
        For each integer $p > 0$, it holds that 
        \begin{equation}
		\sup_{N} \Ebb \left[\|\W_N\W_N^*\|^p\right] < +\infty,
		\notag
        \end{equation}
\end{lemma}
a property also established in the appendix. We now prove the following result.
\begin{lemma}
        \label{lemma:utile-omega}
        Assume assumptions \textbf{A-1} to \textbf{A-6} hold. 
        For all function $\psi \in \Ccal^\infty(\Rbb,\Rbb)$ constant over the complementary of a compact interval and 
        which vanishes on the support $\Scal_N$ of $\mu_N$ for all $N$ large enough, it holds that
        \begin{equation}
		\label{eq:moment-trace-phi-omega}
		\Ebb\left[\left(\Tr \psi(\hat{\Omegabs}_N) \right)^{2l} \right] = \Ocal\left(\frac{1}{N^{2l}}\right)
        \end{equation}
	for each $l \in \mathbb{N}$. 
\end{lemma}
\begin{proof}
        As previously, we prove Lemma \ref{lemma:utile-omega} by induction on $l$. 
        We first consider the case $l=1$, and consider a function $\psi$ as above, and denote
        by $C$ the constant value taken by $\psi$ over the complementary of a certain compact interval, and by $\tilde{\psi}$ the function of 
        ${\cal C}^{\infty}_c(\Rbb, \Rbb)$ defined by $\tilde{\psi}(\lambda)  = \psi(\lambda) - C$, which, of course, is equal to $-C$ on $\Scal_N$. 
        Using Lemma \ref{lemma:haagerup_capitaine} and Lemma \ref{le:fund}, we obtain 
        \begin{equation}
		\label{eq:first-moment}
		\Ebb\left[\frac{1}{M} \Tr \tilde{\psi}(\hat{\Omegabs}_N)\right] =  
		\int_{\Scal_N} \tilde{\psi}(\lambda) d\mu_N(\lambda) 
		+ \frac{1}{M}  \int_{\Scal_N} \tilde{\psi}(\lambda) \drm\kappa_N(\lambda) 
		+ \Ocal\left(\frac{1}{N^{2}}\right).
        \end{equation}
        Using that $\kappa_N([x_{q,N}^{-}, x_{q,N}^{+}]) = 0$ for each $q=1, \ldots, Q_N$, we get that  $\int_{\Scal_N} \tilde{\psi}(\lambda) d\kappa_N(\lambda) = 0$
        and that 
        $$
		\Ebb\left[\frac{1}{M} \Tr \tilde{\psi}(\hat{\Omegabs}_N)\right] =  - C +  \Ocal\left(\frac{1}{N^{2}}\right).
        $$
        Therefore, it holds that 
        \begin{equation}
		\Ebb\left[\frac{1}{M} \Tr \psi(\hat{\Omegabs}_N)\right] = \Ocal\left(\frac{1}{N^{2}}\right).
		\notag
        \end{equation}
        Moreover, we prove the following lemma.
        \begin{lemma}
		\label{le:varpsi}
		Assume assumptions \textbf{A-1} to \textbf{A-6} hold. It holds that 
		\begin{align}
			\label{eq:varpsi}
			\mathrm{Var}\left[\frac{1}{M} \mathrm{Tr}\left(\psi\left(\hat{\Omegabs}_N\right)\right)\right] = \Ocal\left(\frac{1}{N^{4}}\right).
		\end{align}
	\end{lemma}
\begin{proof}
        We first note that, considered as a function of $(\mathrm{Re}(W_{i,j,N}), \mathrm{Im}(W_{i,j,N}))_{1 \leq i \leq M,1 \leq j \leq N}$, function 
        $\frac{1}{M} \Tr \tilde{\psi}(\hat{\Omegabs}_N)$ is continuously differentiable by Lemma \ref{le:differentiabilite-trpsihatomega}. 
        Therefore, function $\frac{1}{M} \Tr \psi(\hat{\Omegabs}_N)$ is continuously differentiable as well. 
        It is thus possible to use the Poincaré inequality to evaluate the lefthandside of \eqref{eq:varpsi}. 
        Furthermore, as the probability that the eigenvalues $(\hat{\lambda}_{k,N})_{k=1, \ldots, M}$ have multiplicity one is equal to 1, it is sufficient to 
        evaluate the partial derivatives of function $\frac{1}{M} \Tr \psi(\hat{\Omegabs}_N)$ when $\W_N$ is such that the $(\hat{\lambda}_{k,N})_{k=1, \ldots, M}$ 
        have multiplicity 1.
        As the derivative of $\psi$ coincides with $\tilde{\psi}^{'}$, \eqref{eq:expre-deriveetracepsi} and Poincaré inequality lead to 
        \begin{equation}
		\mathrm{Var}\left[\frac{1}{M} \mathrm{Tr}\left(\psi(\hat{\Omegabs}_N)  \right) \right] 
		\leq \frac{C}{N^{2}} 
		\Ebb 
		\left[
			\frac{1}{M} \mathrm{Tr}\left(\Sigmabs_N \Sigmabs_N^{*} \sum_{l=1}^{M} \left|[\psi'(\hat{\Omegabs}_N)]_{l,l}\right|^{2} \hat{\Pi}_{l,N}\right) 
		\right],
		\notag
        \end{equation}
        or equivalently, 
        \begin{equation}
	\mathrm{Var}\left[\frac{1}{M} \mathrm{Tr} \left(\psi(\hat{\Omegabs}_N) \right) \right] 
	\leq \frac{C}{N^{2}} \Ebb \left[\frac{1}{M} \sum_{l=1}^{M} \hat{\lambda}_{l,N}   \left|[\psi'(\hat{\Omegabs}_N)]_{l,l}\right|^{2} \right].
	\notag
        \end{equation}
        We claim that
        \begin{equation}
		\label{eq:jensen-psiomegall}
		\left|[\psi'(\hat{\Omegabs}_N)]_{l,l}\right|^2 \leq \left([\psi'(\hat{\Omegabs}_N)]^{2}\right)_{l,l}.
        \end{equation}
        Indeed, if $(\hat{{\bf v}}_{k,N})_{k=1, \ldots, M}$ represent the eigenvectors of $\hat{\Omegabs}$, then
        \begin{align}
		\left[\psi'(\hat{\Omegabs}_N)\right]_{l,l} = \sum_{k=1}^{M} \psi'(\hat{\omega}_{k,N}) |{\bf e}_l^{T} \hat{{\bf v}}_{k,N}|^{2}.
		\notag
	\end{align}
        As $\sum_{k=1}^{M}  |{\bf e}_l^{T} \hat{{\bf v}}_{k,N}|^{2} = 1$, Jensen's inequality yields to \eqref{eq:jensen-psiomegall}. 
        Therefore, it holds that 
        \begin{align}
		\mathrm{Var}\left[\frac{1}{M} \mathrm{Tr}\left(\psi(\hat{\Omegabs}_N) \right) \right]
		\leq 
		\frac{C}{N^{2}}  \Ebb\left[\left\| \Sigmabs_N  \Sigmabs_N^{*} \right\| \frac{1}{M} \sum_{l=1}^{M} \psi'(\hat{\omega}_l)^{2} \right].
		\label{eq:majoration-varpsi-final1}
        \end{align}
        As $\sup_N \| \B_N \B_N^* \| < +\infty$, we get using lemma \ref{le:normeWW*} that
        \begin{equation}
		\sup_{N} \Ebb \left[  \| \Sigmabs_N  \Sigmabs_N^{*} \|^{p} \right] < +\infty.
		\notag
        \end{equation}
        We remark that $  \| \Sigmabs_N  \Sigmabs_N^{*} \| < t_2^{+} + \epsilon$  on the set $\Ecal_{1,N}^{c}$,and write the righthandside of 
        \eqref{eq:majoration-varpsi-final1} as
        \begin{align}
		\frac{C}{N^{2}} \Ebb\left[\left\| \Sigmabs_N  \Sigmabs_N^{*} \right\| (\mathbb{1}_{\Ecal_{1,N}} + \mathbb{1}_{\Ecal_{1,N}^{c}}) 
		\frac{1}{M} \sum_{l=1}^{M} \psi'(\hat{\omega}_l)^{2} \right].
		\notag
	\end{align}
        It holds that 
        \begin{align}
		\Ebb\left[\left\| \Sigmabs_N  \Sigmabs_N^{*}  \right\| \mathbb{1}_{\Ecal_{1,N}^{c}} 
		\frac{1}{M} \sum_{l=1}^{M} \psi'(\hat{\omega}_l)^{2} \right] \leq 
		\left( t_2^{+} + \epsilon \right) \Ebb \left[ \frac{1}{M} \Tr (\psi'(\hat{\Omegabs}_N))^{2} \right].
		\notag
	\end{align}
        Function $\psi'^{2}$ belongs to ${\cal C}^{\infty}_c(\Rbb, \Rbb)$ and vanishes on $\Scal_N$. Therefore, 
        lemma \ref{le:fund} implies that $\Ebb \left[ \frac{1}{M} \Tr (\psi'(\hat{\Omegabs}_N))^{2} \right] = \Ocal(\frac{1}{N^{2}})$ 
        (see Eq. \eqref{eq:first-moment}). 
        Moreover, as $\frac{1}{M} \Tr \left(\psi'(\hat{\Omegabs}_N)^{2}\right) \leq \sup_{\lambda} \psi'(\lambda)^{2} < C$, we have 
        \begin{align}
		\Ebb\left[\left\| \Sigmabs_N  \Sigmabs_N^{*}  \right\| \mathbb{1}_{\Ecal_{1,N}} \frac{1}{M} \sum_{l=1}^{M} \psi'(\hat{\omega}_l)^{2} \right]
		< C \Ebb\left[\left\| \Sigmabs_N  \Sigmabs_N^{*}  \right\| \mathbb{1}_{\Ecal_{1,N}} \right],
		\notag
        \end{align}
        which is itself upperbounded by 
        \begin{align}
		C \left(\Ebb\left[\|\Sigmabs_N  \Sigmabs_N^{*} \|^{2}\right]\right)^{1/2} \Pbb(\Ecal_{1,N})^{1/2} = \Ocal\left(\frac{1}{N^{p}}\right),
		\notag
        \end{align}
        for each integer $p$. This completes the proof of lemma \ref{le:varpsi}.
\end{proof}      
        Assume that \eqref{eq:moment-trace-phi-omega} holds until integer $l-1$.
        We  write as previously that
        \begin{align}
	\label{eq:moment-2l-trpsi-omega}
	 \Ebb\left[\left(\Tr \psi(\hat{\Omegabs}_N)\right)^{2l}\right]
	 =
	 \left(\Ebb\left[\left(\Tr \psi(\hat{\Omegabs}_N)\right)^{l}\right]\right)^2
	 +
	 \Var\left[\left(\Tr \psi(\hat{\Omegabs}_N)\right)^{l}\right].
        \end{align}
        The Cauchy-Schwarz inequality leads immediately to 
        \begin{align}
	\label{eq:schwartz-Tr-psi-l}
	 \left(\Ebb\left[\left(\Tr \psi(\hat{\Omegabs}_N)\right)^{l}\right]\right)^2
	 \leq
	 \Ebb\left[\left(\Tr \psi(\hat{\Omegabs}_N)\right)^{2}\right]
	 \Ebb\left[\left(\Tr \psi(\hat{\Omegabs}_N)\right)^{2l-2}\right]
	 =
	 \Ocal\left(\frac{1}{N^{2l}}\right).
        \end{align}
        As for the second term of the r.h.s. of \eqref{eq:moment-2l-trpsi-omega}, we use Poincaré inequality and Hölder's inequality to obtain
        \begin{align}
	 \Var\left[\left(\Tr \psi(\hat{\Omegabs}_N)\right)^{l}\right]
	 &\leq
	 C \Ebb
	 \left[
		\left(\Tr \psi(\hat{\Omegabs}_N)\right)^{2(l-1)} 
		\frac{1}{M} \sum_{k=1}^M \hat{\lambda}_{k,N} \left([\psi'(\hat{\Omegabs}_N)]_{k,k}\right)^2
	\right],
	 \notag\\
	 &\leq
	 C \left(\Ebb\left|  \frac{1}{M} \sum_{k=1}^M \hat{\lambda}_{k,N}\left([\psi'(\hat{\Omegabs}_N)]_{k,k}\right)^2\right|^l\right)^{\frac{1}{l}}
	 \left(\Ebb\left[\left(\Tr \psi(\hat{\Omegabs}_N)\right)^{2l}\right]\right)^{\frac{l-1}{l}}.
	 \notag
        \end{align}
        Jensen's inequality leads again to 
        \begin{align}
	 \Ebb\left|  \frac{1}{M} \sum_{k=1}^M \hat{\lambda}_{k,N}\left([\psi'(\hat{\Omegabs}_N)]_{k,k}\right)^2\right|^l
	 &\leq
	 \Ebb\left| \|\Sigmabs_N\Sigmabs_N^*\|  \frac{1}{M} \Tr \psi'(\hat{\Omegabs}_N)^2\right|^l.
	 \notag
        \end{align}
        We write again that 
        \begin{align}
	  \Ebb\left| \|\Sigmabs_N\Sigmabs_N^*\|  \frac{1}{M} \Tr \psi'(\hat{\Omegabs}_N)^2\right|^l
	  &=
	  \Ebb\left| \|\Sigmabs_N\Sigmabs_N^*\|  \frac{1}{M} \Tr \psi'(\hat{\Omegabs}_N)^2 \mathbb{1}_{\Ecal_{1,N}}\right|^l
	  +
	  \Ebb\left| \|\Sigmabs_N\Sigmabs_N^*\|  \frac{1}{M} \Tr \psi'(\hat{\Omegabs}_N)^2 \mathbb{1}_{\Ecal_{1,N}^c}\right|^l,
	  \notag
        \end{align}
        and obtain as previously that 
        \begin{align} 
	\Ebb\left| \|\Sigmabs_N\Sigmabs_N^*\|  \frac{1}{M} \Tr \psi'(\hat{\Omegabs}_N)^2\right|^l 
	\leq  C \left(\Ebb\left|\frac{1}{M} \Tr \psi'(\hat{\Omegabs}_N)^2\right|^l + \left(\Prob(\Ecal_{1,N})\right)^{1/2} \right).
	\notag
        \end{align}
        But, applying Cauchy-Schwarz inequality as in \eqref{eq:schwartz-Tr-psi-l} to $\Ebb\left| \Tr \psi'(\hat{\Omegabs}_N)^2\right|^l$ 
        leads to $\Ebb\left|\frac{1}{M} \Tr \psi'(\hat{\Omegabs}_N)^2\right|^l  = \Ocal\left(\frac{1}{N^{2l}}\right)$. 
        Gathering all the previous inequalities, we find that
        \begin{align}
	\Ebb\left[\left(\Tr \psi(\hat{\Omegabs}_N)\right)^{2l}\right]
	\leq
	\frac{C}{N^2} \left( \Ebb\left[\left(\Tr \psi(\hat{\Omegabs}_N)\right)^{2l}\right]\right)^{\frac{l-1}{l}}
	+\Ocal\left(\frac{1}{N^{2l}}\right),
	\notag
        \end{align}
        and in the same way as in the proof of Proposition \ref{lemma:escape-lambda}, we obtain  
	$\Ebb\left[\left(\Tr \psi(\hat{\Omegabs}_N)\right)^{2l}\right]=\Ocal\left(\frac{1}{N^{2l}}\right)$.
        This concludes the proof of Lemma \ref{lemma:utile-omega}.  
\end{proof}

\subsection{\texorpdfstring{End of the proof of theorem \ref{theorem:uniform_consistency}}{End of the proof of uniform consistency}}
\label{subsec:end-of-proof}

We now complete the proof of Theorem \ref{theorem:uniform_consistency} when function $\theta \rightarrow \a(\theta)$ is given by 
\begin{align}
	\a(\theta) = \frac{1}{\sqrt{M}} \left[1,e^{i \theta},\ldots,e^{i  (M-1)\theta}\right]^{T},
	\notag
\end{align}
for $\theta \in [-\pi,\pi]$.
We recall that $\Ecal_N$ is defined by  
\begin{equation}
        \Ecal_N = \Ecal_{1,N} \bigcup \Ecal_{2,N},
        \notag
\end{equation}
where $(\Ecal_{i,N})_{i=1,2}$ are defined by \eqref{eq:def-E1} and \eqref{eq:def-E2}, and that $\mathbb{1}_{\Ecal_N^{c}} \leq \chi_N$
where  $\chi_N = \det \phi(\Sigmabs_N \Sigmabs_N^{*}) \det \phi(\hat{\Omegabs}_N)$.
We first give a useful lemma which appears as a straighforward consequence of the evaluations of Section \ref{subsec:evaluations}
\begin{lemma}
        \label{le:bornes-utiles}
        Assume assumptions \textbf{A-1} to \textbf{A-6} hold. For each $N$, it holds that 
        \begin{align}
		\sup_{z \in \partial \Rcal_y} \| \T_N(z) \|  \leq  C, \nonumber \\
		\sup_{z \in \partial  \Rcal_y} \left| \frac{1}{1+\sigma^{2} c_N m_N(z)} \right|  \leq  C, \nonumber \\
		\sup_{z \in \partial  \Rcal_y} \left| \frac{w_N'(z)}{1+\sigma^{2} c_N m_N(z)} \right|  \leq  C, \nonumber
        \end{align}
        and for $N$ large enough, we have
        \begin{align}
		\sup_{z \in \partial \Rcal_y} \| \Q_N(z) \| \chi_N  \leq  C, \nonumber \\
		\sup_{z \in \partial  \Rcal_y} \left| \frac{\chi_N}{1+\sigma^{2} c_N \hat{m}_N(z)} \right|  \leq  C, \nonumber \\
		\sup_{z \in \partial  \Rcal_y} \left| \frac{\hat{w}_N'(z)}{1+\sigma^{2} c_N \hat{m}_N(z)} \right| \; \chi_N  \leq  C. \nonumber
        \end{align}
\end{lemma} 
We consider the set 
\begin{equation}
        \vartheta_N = \left\{-\pi + \frac{2 (k-1) \pi}{N^{2}}: k=1,\ldots,N^{2} \right\},
        \notag
\end{equation}
and remark that for each $\theta \in [-\pi,\pi]$ and for each $N$, there exists $\theta_N \in \vartheta_N$ such that 
$|\theta - \theta_N| \leq \frac{2 \pi}{N^{2}}$. For each $\theta \in [-\pi,\pi]$, it holds that 
\begin{align}
	\label{eq:decomposition-a}
	\tilde{\eta}_N(\theta) - \eta_N(\theta) 
	&=
	\left[\tilde{\eta}_N(\theta) - \tilde{\eta}_N(\theta_N)\right]
	+
	\left[\tilde{\eta}_N(\theta_N) - \eta_N(\theta_N)\right]
	+
	\left[\eta_N(\theta_N) - \eta_N(\theta)\right].
\end{align}
It is easy to check that the third term of the r.h.s. of \eqref{eq:decomposition-a} satisfies
\begin{align}
	\label{eq:bornesupa}
	\sup_{\theta \in [-\pi,\pi]} \left|\eta_N(\theta_N) - \eta_N(\theta)\right| 
	\leq 2 \sup_{\theta \in [-\pi,\pi]} \left\|\a(\theta) - \a(\theta_N)\right\| 
	= \Ocal\left(\frac{1}{N}\right).
\end{align}
In order to evaluate the behaviour of the supremum over $\theta$ of the first term of the r.h.s. of \eqref{eq:decomposition-a}, 
we prove that for each $\alpha > 0$, 
\begin{equation}
        \Prob\left( \sup_{\theta \in [-\pi, \pi]} |\tilde{\eta}_N(\theta) -  \tilde{\eta}_N(\theta_N) | > \alpha \right) = \Ocal\left(\frac{1}{N^{1+\beta}}\right),
        \notag
\end{equation}
where $\beta > 0$. We first remark that for each $l \in \mathbb{N}$, it holds that 
\begin{align}
	\Prob\left(\sup_{\theta \in [-\pi,\pi]} \left|\tilde{\eta}_N(\theta) - \tilde{\eta}_N(\theta_N)\right| > \alpha \right)
	&\leq
	\Prob\left(\sup_{\theta \in [-\pi,\pi]} \left|\tilde{\eta}_N(\theta) - \tilde{\eta}_N(\theta_N)\right| \mathbb{1}_{\Ecal_{N}^{c}} > \alpha  \right)
	+ \Pbb\left(\Ecal_N \right)
	\notag\\
	&\leq
	\frac{1}{\alpha^{l}} \Ebb\left[\sup_{\theta \in [-\pi,\pi]} \left|\tilde{\eta}_N(\theta) - \tilde{\eta}_N(\theta_N)\right|^l  \mathbb{1}_{\Ecal_N^{c}} \right]
	+ \Ocal\left(\frac{1}{N^l}\right).
	\notag
\end{align}
Moreover,
\begin{align}
        \left|\tilde{\eta}_N(\theta) - \tilde{\eta}_N(\theta_N)\right|^l  \mathbb{1}_{\Ecal_N^{c}}
        \leq
        C \oint_{\partial \Rcal_y^-}
        \left|
	  \left(\a(\theta)-\a(\theta_N)\right)^* 
	  \Q_N(z) \frac{\hat{w}'_N(z)}{1+\sigma^2 c_N \hat{m}_N(z)} \a(\theta_N)
	  \right|^{l} 
	  \mathbb{1}_{\Ecal_N^{c}} 
        \left|\drm z\right|.
        \notag
\end{align}
Lemma \ref{le:bornes-utiles} and the inequality $\mathbb{1}_{\Ecal_N^{c}} \leq \chi_N$ imply that 
\begin{equation}
        \label{eq:borne-terme-regularise}
        \sup_{z \in \partial \Rcal_y} \| \Q_N(z) \| \left| \frac{\hat{w}'_N(z)}{1+\sigma^2 c_N \hat{m}_N(z)} \right| \mathbb{1}_{\Ecal_N^{c}} < C
\end{equation}
for some constant term $C$. 
Inequality \eqref{eq:bornesupa} thus implies that 
$$
\sup_{\theta \in [-\pi, \pi]} \left|
	  \left(\a(\theta)-\a(\theta_N)\right)^* 
	  \Q_N(z) \frac{\hat{w}'_N(z)}{1+\sigma^2 c_N \hat{m}_N(z)} \a(\theta_N)
	  \right|^{l} 
	  \mathbb{1}_{\Ecal_N^{c}} \leq  \frac{C}{N^{l}}
$$
thus showing that 
$$\Prob\left(\sup_{\theta \in [-\pi,\pi]} \left|\tilde{\eta}_N(\theta) - \tilde{\eta}_N(\theta_N)\right| > \alpha \right) = \Ocal\left(\frac{1}{N^{l}}\right)$$
for each integer $l$. Borel-Cantelli's lemma eventually implies that 
$$\sup_{\theta \in [-\pi,\pi]} \left|\tilde{\eta}_N(\theta) - \tilde{\eta}_N(\theta_N) \right| \rightarrow 0$$
almost surely. 

We finally study the supremum of the second term of \eqref{eq:decomposition-a}. 
We denote by $\nu_{k,N}$ the elements of $\vartheta_N$.
Let $\alpha > 0$, then 
\begin{align}
	\Prob\left(\sup_{\theta \in [-\pi,\pi]} \left|\tilde{\eta}_N(\theta_N) - \eta_N(\theta_N)\right| > \alpha\right)
	 &\leq
	\Prob\left(\sup_{k=1,\ldots,N^2}\left|\tilde{\eta}_N(\nu_{k,N}) - \eta_N(\nu_{k,N})\right| > \alpha \right)
	\notag\\
	&\leq
	\sum_{k=1}^{N^2}
	\Prob\left(\left|\tilde{\eta}_N(\nu_{k,N}) - \eta_N(\nu_{k,N})\right| > \alpha \right)
	\notag\\
	&\leq  
	\sum_{k=1}^{N^{2}} \left[ \Prob\left(\left\{\left|\tilde{\eta}_N(\nu_{k,N}) - \eta_N(\nu_{k,N})\right| > \alpha \right\} \cap \Ecal_N^{c} \right) \right]
	+ \Ocal\left(\frac{1}{N^l}\right)
	\notag
\end{align}
for each integer $l$. We now introduce in the above term the regularization term $\chi_N = \det \phi(\Sigmabs_N \Sigmabs_N^*) \det \phi(\hat{\Omegabs}_N)$ defined in \eqref{eq:def-chii}. 
As  $\chi_N$ is equal to 1 on $ \Ecal_N^{c}$, it holds that 
\begin{align}
	\Prob\left(\left\{\left|\tilde{\eta}_N(\nu_{k,N}) - \eta_N(\nu_{k,N})\right| > \alpha \right\} \cap  \Ecal_N^{c} \right)
	&= 
	\Prob\left(\left\{\left|\tilde{\eta}_N(\nu_{k,N}) - \eta_N(\nu_{k,N})\right| \chi_N^{2} > \alpha \right\} \cap  \Ecal_N^{c} \right)
	\notag\\
	&\leq
	\Prob\left(\left|\tilde{\eta}_N(\nu_{k,N}) - \eta_N(\nu_{k,N})\right|  \chi_N^{2} > \alpha \right)
	\notag\\
	&\leq
	\frac{1}{\alpha^{2 l}} \Ebb\left|\left(\tilde{\eta}_N(\nu_{k,N}) - \eta_N(\nu_{k,N})\right) \chi_N^{2} \right|^{2l}.
	\notag
\end{align}
The introduction of $\chi_N$ is in part motivated by the observation that the moments of $\tilde{\eta}_N(\nu_{k,N})  \chi_N^{2}$.
are finite. Moreover, it holds that 
\begin{align}
	&\Ebb\left|\left(\tilde{\eta}_N(\nu_{k,N}) - \eta_N(\nu_{k,N})\right) \chi_N^{2} \right|^{2l}
	\notag\\
	&\quad\leq
	C
	  \oint_{\partial \Rcal_y^-}
	   \Ebb
	\left[ \left|
	    \a(\nu_{k,N})^* \left(\Q_N(z) \frac{\hat{w}'_N(z)}{1+\sigma^2 c_N \hat{m}_N(z)} - \T_N(z) \frac{w'_N(z)}{1+\sigma^2 c_N m_N(z)}\right) \a(\nu_{k,N})
	  \chi_N^{2}  \right|^{2l}
	  \right] \; \left|\drm z\right| 
\end{align}
In order to complete the proof of Theorem \ref{theorem:uniform_consistency}, we establish the following proposition. 
\begin{proposition}
        \label{prop:moment2l-terme2}
        Assume assumptions \textbf{A-1} to \textbf{A-6} hold.
        If $(\a_N)_ {N \in \mathbb{N}}$ is sequence of deterministic vectors satisfying $\| \a_N \|=1 $, then, for each integer $l$, it holds that 
        \begin{equation}
	\label{eq:moment2l-terme2}
	 \sup_{z \in \partial \Rcal_y} \; \Ebb \left[ \left|
	 \a_N^* \left(\Q_N(z) \frac{\hat{w}'_N(z)}{1+\sigma^2 c_N \hat{m}_N(z)} - \T_N(z) \frac{w'_N(z)}{1+\sigma^2 c_N m_N(z)}\right) \a_N \chi_N^{2} 
	 \right|^{2l} \right] \leq \frac{C}{N^{l}}
        \end{equation}
        where the constant $C$ does not depend on the sequence $(\a_N)$.
\end{proposition}
\begin{proof} 
        In order to shorten the notations, we denote by $\hat{g}_N(z)$ and $g_N(z)$ the functions defined by 
        \begin{align}
		\hat{g}_N(z) =  \a_N^* \Q_N(z) \a_N  \; \frac{\hat{w}'_N(z)}{1+\sigma^2 c_N \hat{m}_N(z)},
		\notag
        \end{align}
        and 
        \begin{align}
		g_N(z) =  \a_N^*  \T_N(z)  \a_N  \; \frac{w'_N(z)}{1+\sigma^2 c_N m_N(z)}.
		\notag
        \end{align}
        In order to evaluate $\Ebb|\hat{g}_N(z) - g_N(z) \chi_N^{2}|^{2l}$, we use the Poincaré inequality. 
        For this, we first state the following lemma proved in the appendix. We recall that if $\H$ a hermitian matrix with a spectral decomposition
        $\H = \sum_l \gamma_l {\bf x}_l {\bf x}_l^*$, its adjoint (i.e. the transpose of its cofactor matrix) denoted by $\adj(\H)$ is given by
        $\adj(\H) = \sum_l (\prod_{k\neq l} \gamma_k) \x_l \x_l^*$. When $\H$ is invertible, $\adj(\H) = \det(\H)\H^{-1}$.
        Next, we state the following lemma proved in the appendix.
        \begin{lemma}
		\label{le:differentiability-chi}
		Assume assumptions \textbf{A-1} to \textbf{A-6} hold. 
		Considered as functions of the real and imaginary parts of the entries of $\W_N$, functions $\det\phi(\Sigmabs_N \Sigmabs_N^*)$ and 
		$\det\phi(\hat{\Omegabs}_N)$ belong to $\Ccal^1(\Rbb^{2MN})$, and their partial derivatives w.r.t. $W_{i,j,N}$ denoted by  
		\begin{align}
			[\D_1]_{i,j,N} &:= \frac{\partial}{\partial W_{i,j,N}} \left\{\det \phi(\Sigmabs_N\Sigmabs_N^*)\right\}, 
			\notag\\
			[\D_2]_{i,j,N} &:= \frac{\partial}{\partial W_{i,j,N}} \left\{\det \phi(\hat{\Omegabs}_N)\right\}, 
			\notag
		\end{align}
		are given almost surely by 
		\begin{align} 
			[\D_1]_{i,j,N}
			&= 
			\e_j^* \Sigmabs_N^* 
			\adj\left( \phi(\Sigmabs_N \Sigmabs_N^*) \right) 
			\phi'(\Sigmabs_N \Sigmabs_N^*) \e_i ,
			\label{eq:deriv-detSigma}  \\
			[\D_2]_{i,j,N}&= 
			\left[
			\Sigmabs_N^* \sum_{l=1}^M \left[ \adj\left( \phi(\hat{\Omegabs}_N) \right) \phi'(\hat{\Omegabs}_N)\right]_{ll} \hat{\Pibs}_{l,N} 
			\right]_{ji} 
			\label{eq:deriv-detOmega} 
		\end{align} 
		If we denote by $ \Acal_{1,N}$ and $ \Acal_{2,N}$ the events defined by 	
		\begin{align}
			\Acal_{1,N}  &=
			\left\{\exists k: \hat{\lambda}_{k,N} \not\in \Tcal_{\epsilon} \right\}\cap \left\{\hat{\lambda}_{1,N},\ldots,\hat{\lambda}_{M,N} \in \supp(\phi)\right\},
			\notag\\
			\Acal_{2,N} &=
			\left\{\exists k: \hat{\omega}_{k,N} \not\in \Tcal_{\epsilon} \right\}\cap \left\{\hat{\omega}_{1,N},\ldots,\hat{\omega}_{M,N} \in  \supp(\phi)\right\}.
			\notag
		\end{align}
		then $ [\D_1]_{i,j,N} = 0$ on $ \Acal_{1,N}^{c}$ and  $ [\D_2]_{i,j,N} = 0$ on $ \Acal_{2,N}^{c}$.  
        \end{lemma} 
        We now establish \eqref{eq:moment2l-terme2} by induction on $l$, and first consider the case $l=1$. We write the second moment of 
        $(\hat{g}_N(z) - g_N(z)) \chi_N^{2}$ as
        $$
	\Ebb\left|(\hat{g}_N(z) - g_N(z)) \chi_N^{2} \right|^{2} = 
	\left|\Ebb \left( (\hat{g}_N(z) -g_N(z)) \chi_N^{2} \right)\right|^{2} + \Var\left( (\hat{g}_N(z) - g_N(z))\chi_N^{2} \right).
        $$
        We evaluate $ \Var\left[ (\hat{g}_N(z) - g_N(z)) \chi_N^{2} \right]$ using the Poincaré inequality and get
        \begin{align}
		\label{eq:var-gchi2}
		\Var \left[(\hat{g}_N(z) - g_N(z)) \chi_N^{2} \right] 
		\leq 
		\frac{\sigma^{2}}{N} \sum_{i,j}\Ebb\left[\chi_N^{4}\left(\left| \frac{\partial \hat{g}_N(z)}{\partial W_{i,j,N}}  \right|^{2} + 
		\left|\frac{\partial \hat{g}_N(z)}{\overline{W}_{i,j,N}}\right|^{2} \right) \right]+ 2 
		\Ebb \left[|\hat{g}_N(z)-g_N(z)|^{2} \left| \frac{\partial \chi_N^{2}}{\partial W_{i,j,N}} \right|^{2}\right].
        \end{align}
        It is clear that 
        \begin{align}
		\frac{\partial \hat{g}_N(z)}{\partial W_{i,j,N}} 
		=  \a_N^* \frac{\partial \Q_N(z)}{\partial W_{i,j,N}} \a_N \frac{\hat{w}'_N(z)}{1+\sigma^2 c_N \hat{m}_N(z)} 
		+ \a_N^* \Q_N(z) \a_N  \frac{\partial }{\partial W_{i,j,N}} \left\{ \frac{\hat{w}_N'(z)}{1+\sigma^{2} c_N \hat{m}_N(z)} \right\}.
		\notag
        \end{align}
        We verify that 
        $$
		\a_N^* \frac{\partial \Q_N(z)}{\partial W_{i,j,N}} \a_N = - \a_N^* \Q_N(z) {\bf e}_i \; {\bf e}_j \Sigmabs_N^* \Q_N(z) \a_N,
        $$
        so that 
        $$
		\sum_{i,j}  \left|\a_N^* \frac{\partial \Q_N(z)}{\partial W_{i,j,N}} \a_N \right|^{2} = 
		\a_N^* \Q_N(z) \Q_N(z)^* \a_N \; \a_N^* \Q_N(z) \Sigmabs_N  \Sigmabs_N^* \Q_N(z)^* \a_N.
        $$
        As $\chi_N \neq 0$ implies that $ \hat{\lambda}_{M,N} = \| \Sigmabs_N  \Sigmabs_N^* \| \leq t_2^+ + 2 \epsilon$, Lemma \ref{le:bornes-utiles} implies that
        $$
		\sup_{z \in \partial  \Rcal_y} \chi_N^{2} \, \a_N^* \Q_N(z) \Q_N(z)^* \a_N \; \a_N^* \Q_N(z) \Sigmabs_N  \Sigmabs_N^* \Q_N(z)^* \a_N \leq  C.
        $$
        Using again Lemma \ref{le:bornes-utiles}, we get that 
        \begin{equation}
		\label{eq:borne-terme1-w}
		\sup_{z \in \partial \Rcal_y}  \chi_N^4 \left| \frac{\hat{w}_N'(z)}{1+\sigma^{2} c_N \hat{m}_N(z)} \right|^{2}
		\sum_{i,j}\left|\a_N^* \frac{\partial \Q_N(z)}{\partial W_{i,j,N}} \a_N \right|^{2} 
		\leq  C.
        \end{equation}
        We obtain similarly that 
        \begin{align}
		\sup_{z \in \partial \Rcal_y} \chi_N^4  |\a_N^* \Q_N(z) \a_N|^{2}  \sum_{i,j}\left| 
		\frac{\partial }{\partial W_{i,j,N}} 
		\left\{\frac{\hat{w}_N'(z)}{1+\sigma^{2} c_N \hat{m}_N(z)} \right\}\right|^{2} 
		\leq \frac{C}{N}.
		\label{eq:borne-terme2-w}
        \end{align}
        The same conclusions hold when the derivatives w.r.t. variables $\overline{W}_{i,j,N}$ are considered. This shows that the first term of the r.h.s. 
        of \eqref{eq:var-gchi2} is a $\Ocal\left(\frac{1}{N}\right)$ term. We now evaluate the behaviour of the second term of the r.h.s. of \eqref{eq:var-gchi2}, 
        and establish that 
        \begin{equation}
		\label{eq:var-gchi2-derivee-chi}
		\sup_{z \in \partial \Rcal_y} \Ebb \left[ |\hat{g}_N(z) - g_N(z)|^{2}  
		\sum_{i,j}  \left| \frac{\partial  \chi_N^{2}}{\partial W_{i,j,N}} \right|^{2}  \right] = \Ocal\left(\frac{1}{N^{p}}\right)
        \end{equation}
        for each integer $p$.
        We express $\frac{\partial  \chi_N^{2}}{\partial W_{i,j,N}}$ as $2 \chi_N \frac{\partial  \chi_N}{\partial W_{i,j,N}}$.  
        Lemma \ref{le:bornes-utiles} implies that $\sup_{z  \in \partial \Rcal_y} \chi_N^{2} |\hat{g}_N(z) - g_N(z)|^{2} < C$. 
        Therefore, it is sufficient to check that 
        \begin{equation}
		\Ebb \left[  \sum_{i,j}  \left| \frac{\partial  \chi_N}{W_{i,j,N}} \right|^{2} \right] = \Ocal\left(\frac{1}{N^p}\right)
		\notag
        \end{equation}  
        for each integer $p$. $\frac{\partial\chi_N}{\partial W_{i,j,N}}$ can be written as 
        $$
		\frac{\partial  \chi_N}{\partial W_{i,j,N}} 
		= [\D_1]_{i,j,N} \det \phi(\hat{\Omegabs}_N) +  [\D_2]_{i,j,N} \det \phi( \Sigmabs_N\Sigmabs_N^*).
        $$
        It holds that
        $$
		\Ebb \left[ \sum_{i,j} \left| [\D_1]_{i,j,N} \det\phi(\hat{\Omegabs}_N)) \right|^{2} \right]
		= 
		\Ebb 
		\left[  
			\det \phi(\hat{\Omegabs}_N)^{2} 
			\Tr\left(\Sigmabs_N \Sigmabs_N^* \phi'(\Sigmabs_N \Sigmabs_N^*)^2 \adj\left( \phi({\Sigmabs}_N {\Sigmabs_N}^*) \right)^{2} \right)
			\mathbb{1}_{\Acal_{1,N}} 
		\right]
        $$
        Moreover, we can write
        \begin{align}
		\Tr\left(\Sigmabs_N \Sigmabs_N^* \phi'(\Sigmabs_N \Sigmabs_N^*)^2  \adj\left( \phi({\Sigmabs}_N {\Sigmabs_N}^*) \right)^{2}\right)
		& = \sum_{k}  \hat{\lambda}_{k,N} \phi'(\hat{\lambda}_{k,N})^{2} 
		\prod_{l \neq k}  \phi(\hat{\lambda}_{l,N})^{2} 
		\notag\\
		& \leq \Tr\left(\Sigmabs_N \Sigmabs_N^* \phi'(\Sigmabs_N \Sigmabs_N^*)^{2} \right),
		\notag
        \end{align}
        because $\phi(\lambda) \leq 1$ on $\Rbb$. 
        Therefore, it holds that 
        \begin{align} 
		\Ebb\left[\sum_{i,j} \left| [\D_1]_{i,j,N} \det\phi(\hat{\Omegabs}_N) \right|^{2} \right] 
		\leq 
		\Ebb\left[\det\phi(\hat{\Omegabs}_N)^2\Tr\left( \Sigmabs_N \Sigmabs_N^* \phi'(\Sigmabs_N \Sigmabs_N^*)^{2} \right)
		\mathbb{1}_{\Acal_{1,N}} \right] 
		\leq C N \Pbb(\Acal_{1,N}),
		\notag
        \end{align}
        because $\det\phi(\hat{\Omegabs}_N)) \leq 1$, and $ \Tr\left(\Sigmabs_N \Sigmabs_N^* \phi'(\Sigmabs_N \Sigmabs_N^*)^{2} \right)\leq C N$ on $\Acal_{1,N}$. 
        As $\Acal_{1,N} \subset \Ecal_{1,N}$, proposition \ref{lemma:escape-lambda} implies that
        \begin{align}
		\Ebb \left[ \sum_{i,j} \left| [\D_1]_{i,j,N} \det\phi(\hat{\Omegabs}_N)) \right|^{2} \right] = \Ocal\left(\frac{1}{N^{p}}\right)
		\notag
        \end{align}
        for each integer $p$. Using similar calculations and Proposition  \ref{lemma:escape-omega}, we obtain that
        \begin{align}
		\Ebb \left[ \sum_{i,j} \left| [\D_2]_{i,j,N} \det \phi( \Sigmabs_N\Sigmabs_N^*)) \right|^{2} \right] = \Ocal\left(\frac{1}{N^{p}}\right)
		\notag
        \end{align}
        for each integer $p$. This completes the proof of \eqref{eq:var-gchi2-derivee-chi} and establishes that 
        \begin{equation}
		\sup_{z \in \partial \Rcal_y} \Var \left[ (\hat{g}_N(z)-g_N(z)) \chi_N^{2} \right) = \Ocal\left(\frac{1}{N}\right].
		\notag
        \end{equation}
        In order to evaluate the term $\left| \Ebb\left[(\hat{g}_N(z) - g_N(z)) \chi_N^{2}\right]\right|^{2}$, we also need the following auxilliary lemma proved in 
        the appendix.
        \begin{lemma}
		\label{le:biais-gchi2}
		Assume assumptions \textbf{A-1} to \textbf{A-6} hold. It holds that 
		\begin{align}
			\sup_{z \in \partial \Rcal_y}\left|\Ebb\left[\a_N^{*} \Q_N(z) \a_N \chi_N -  \a_N^{*} \T_N(z) \a_N \right] \right| 
			& = \Ocal \left(\frac{1}{N^{3/2}}\right),
			\label{eq:biaisfqchi}
			\\	        
			\sup_{z \in \partial \Rcal_y} \left| \Ebb \left[ \hat{m}_N(z) \chi_N - m_N(z) \right] \right|
			& =  \Ocal \left(\frac{1}{N^{2}} \right),
			\label{eq:biaismnchi}
			\\
			\sup_{z \in \partial \Rcal_y} \left| \Ebb \left[ \hat{m}_N'(z) \chi_N - m_N'(z) \right] \right|
			& =  \Ocal \left( \frac{1}{N^{2}} \right).
			\label{eq:biaismn'chi}
		\end{align}
        \end{lemma}
        We express $(\hat{g}_N(z) - g_N(z)) \chi_N^{2}$ as $\beta_{1,N}(z) + \beta_{2,N}(z)$ where 
        $$
		\beta_{1,N}(z) = 
		\chi_N \left( \a_N^* \Q_N(z) \a_N  - \a_N^* \T_N(z) \a_N \right) \frac{\hat{w}_N'(z) \chi_N}{1+\sigma^{2} c_N \hat{m}_N(z)}
        $$
        and 
        $$
		\beta_{2,N}(z) = 
		\chi_N^{2} \a_N^* \T_N(z) \a_N \left( \frac{\hat{w}_N'(z)}{1+\sigma^{2} c_N \hat{m}_N(z)} - \frac{w_N'(z)}{1+\sigma^{2} c_N m_N(z)} \right), 
        $$
        and establish that 
        \begin{equation}
		\sup_{z \in \partial \Rcal_y} \Ebb|\beta_{1,N}|^{2} =  \Ocal\left(\frac{1}{N} \right)
		\quad\text{and}\quad
		\sup_{z \in \partial \Rcal_y} \Ebb|\beta_{2,N}|^{2} =  \Ocal\left(\frac{1}{N^{2}} \right).
		\label{eq:supEbetai}
        \end{equation}
        Using Lemma \ref{le:bornes-utiles}, \eqref{eq:supEbetai} for $\beta_{1,N}$ will be established if we show that 
        \begin{align}
		\sup_{z \in \partial \Rcal_y} \Ebb\left| \chi_N \left(\a_N^* \Q_N(z) \a_N  - \a_N^* \T_N(z) \a_N \right) \right|^{2}
		= \Ocal\left(\frac{1}{N}\right).
		\notag
	\end{align}
        For this, we write that
        \begin{align}
		\Ebb\left| \chi_N (\a_N^* \Q_N(z) \a_N  - \a_N^* \T_N(z) \a_N) \right|^{2} 
		= 
		\Var(\chi_N \a_N^* \Q_N(z) \a_N ) + \left| \Ebb \left( \chi_N (\a_N^* \Q_N(z) \a_N - \a_N^* \T_N(z) \a_N) \right) \right|^{2}.
		\notag
        \end{align}
        The above calculations prove that $\sup_{z \in \partial \Rcal_y} \Var[\chi_N \a_N^* \Q_N(z) \a_N] =  \Ocal\left(\frac{1}{N}\right)$, 
        while \eqref{eq:biaisfqchi} and $1 - \Ebb(\chi_N) = \Ocal(\frac{1}{N^{p}})$ for each $p$ imply that 
        \begin{align}
		\Ebb \left[\chi_N (\a_N^* \Q_N(z) \a_N - \a_N^* \T_N(z) \a_N)\right] = \Ocal\left( \frac{1}{N^{3/2}} \right).
		\notag
	\end{align}
        This completes the proof of \eqref{eq:supEbetai} for $\beta_{1,N}$. 
        In order to show  \eqref{eq:supEbetai} for $\beta_{2,N}$, we first remark that by Lemma \ref{le:bornes-utiles}, $|\a_N^* \T_N(z) \a_N|$ is uniformly bounded 
        on $\partial \Rcal_y$, and write that 
        \begin{align}
		&\chi_N^{2}  \left( \frac{\hat{w}_N'(z)}{1+\sigma^{2} c_N \hat{m}_N(z)} - \frac{w_N'(z)}{1+\sigma^{2} c_N m_N(z)} \right) 
		=
		\notag\\
		&\quad
		\sigma^{2} c_N \chi_N^{2} \left( \hat{m}_N(z) - m_N(z) \right) + 2 z \sigma^{2} c_N \chi_N^{2} \left( \hat{m}_N'(z) - m_N'(z) \right) 
		- \sigma^4 c_N (1 - c_N) \chi_N^{2} \left( \frac{\hat{m}_N'(z)}{1+\sigma^{2} c_N \hat{m}_N(z)} - \frac{m_N'(z)}{1+\sigma^{2} c_N m_N(z)} \right), 
		\notag
        \end{align}
        or equivalently that 
        \begin{align}
		&\chi_N^{2}  \left( \frac{\hat{w}_N'(z)}{1+\sigma^{2} c_N \hat{m}_N(z)} - \frac{w_N'(z)}{1+\sigma^{2} c_N m_N(z)} \right) 
		=
		\notag\\
		&\quad
		\sigma^{2} c_N \chi_N^{2} \left( \hat{m}_N(z) - m_N(z) \right) + 2 z \sigma^{2} c_N \chi_N^{2} \left( \hat{m}_N'(z) - m_N'(z) \right) 
		\notag \\ 
		&\quad
		- (\sigma^{2} c_N)^{2} \sigma^{2} (1 - c_N) \frac{\chi_N}{(1+\sigma^{2} c_N \hat{m}_N(z))(1+\sigma^{2} c_N m_N(z))} 
		\left[ m_N(z) \chi_N (\hat{m}_N'(z) - m_N'(z))-  m_N'(z) \chi_N (\hat{m}_N(z) - m_N(z)) \right] 
		\notag\\
		&\quad
		-\sigma^{2} c_N \sigma^{2} (1 - c_N) \frac{\chi_N}{(1+\sigma^{2} c_N \hat{m}_N(z))(1+\sigma^{2} c_N m_N(z))} 
		\left[ \chi_N \left( \hat{m}_N'(z) - m_N'(z) \right) \right]. 
		\notag
        \end{align}
        The Poincaré inequality and Lemma \ref{le:biais-gchi2} imply that 
        $$ \sup_{z \in \partial \Rcal_y} \Ebb \left| \chi_N ( \hat{m}_N(z) - m_N(z)) \right|^{2} =  \Ocal\left( \frac{1}{N^{2}} \right)$$
        and 
        $$ \sup_{z \in \partial \Rcal_y} \Ebb \left| \chi_N ( \hat{m}_N'(z) - m_N'(z)) \right|^{2} =  \Ocal\left( \frac{1}{N^{2}} \right).$$
        Eq. \eqref{eq:supEbetai} follows immediately from
        \begin{align}
		\sup_{z \in \partial \Rcal_y}  \left| \frac{\chi_N}{(1+\sigma^{2} c_N \hat{m}_N(z))(1+\sigma^{2} c_N m_N(z))} \right| \leq C,
		\notag
	\end{align}
        for some deterministic constant $C$ (see Lemma \ref{le:bornes-utiles}). This completes the proof of \eqref{eq:moment2l-terme2}  for $l=1$. 

        We now assume that  \eqref{eq:moment2l-terme2} holds until integer $l-1$ and write that 
        $$
		\Ebb \left| \chi_N^{2} \left( \hat{g}_N(z) - g_N(z) \right) \right|^{2l} 
		= 
		\left| \Ebb \left[\left(\chi_N^{2} (\hat{g}_N(z) - g_N(z)) \right)^{l}\right] \right|^{2} 
		+ \Var \left[ \left(\chi_N^{2} (\hat{g}_N(z) - g_N(z)) \right)^{l} \right].
        $$
        The Cauchy-Schwarz inequality implies that 
        $$
		\left| \Ebb \left[\left(\chi_N^{2} (\hat{g}_N(z) - g_N(z)) \right)^{l}\right] \right|^{2} 
		\leq  
		\Ebb \left| \chi_N^{2} ( \hat{g}_N(z) - g_N(z)) \right|^{2} 
		\Ebb \left| \chi_N^{2} ( \hat{g}_N(z) - g_N(z)) \right|^{2(l-1)},
        $$
        and shows that 
        $$
		\sup_{z \in \partial {\cal  R}_y} \left| \Ebb \left( \chi_N^{2} ( \hat{g}_N(z) - g_N(z)) \right)^{l} \right|^{2} 
		=  \Ocal\left( \frac{1}{N^{l}} \right).
        $$
        The Poincaré inequality gives 
        \begin{eqnarray}
		\label{eq:vargchil}
		\Var \left[ \left(\chi_N^{2} (\hat{g}_N(z) - g_N(z)) \right)^{l} \right] 
		&\leq 
		\frac{\sigma^{2} l^{2} }{N} 
		\Ebb 
		\left[
			\left|\chi_N^{2}(\hat{g}_N(z) - g_N(z))\right|^{2(l-1)}
			\chi_N^{4}\sum_{i,j}
			\left(
				\left| \frac{\partial \hat{g}_N(z)}{\partial W_{i,j,N}} \right|^{2} 
				+\left| \frac{\partial \hat{g}_N(z)}{\partial \overline{W}_{i,j,N}} \right|^{2} 
			\right) 
		\right]
		\nonumber \\ 
		&
		+ \frac{8 \sigma^{2} l^{2}}{N}  
		\Ebb 
		\left[
			\left|\hat{g}_N(z) - g_N(z)) \right|^{2l}  
			\chi_N^{4l-2} \sum_{i,j} 
			\left| \frac{\partial \chi_N}{\partial W_{i,j,N}} \right|^{2} 
		\right].
		\notag
        \end{eqnarray}
        Finally, \eqref{eq:borne-terme1-w} and \eqref{eq:borne-terme2-w} imply that 
        $$
		\sup_{z \in \partial \Rcal_y} 
		\chi_N^{4}\sum_{i,j}
		\left(
			\left|\frac{\partial \hat{g}_N(z)}{\partial W_{i,j,N}} \right|^{2} 
			+\left| \frac{\partial \hat{g}_N(z)}{\partial \overline{W}_{i,j,N}} \right|^{2} 
		\right) 
		\leq C,
        $$ 
        for some deterministic constant $C$. 
        Therefore, the supremum over $z \in  \partial \Rcal_y$ of first term of the r.h.s. of \eqref{eq:vargchil} is a $\Ocal\left(\frac{1}{N^{l}}\right)$. 
        Moreover, it can be shown as in the case $l=1$ that the supremum over $z \in  \partial \Rcal_y$ of the second term of the righthandside of \eqref{eq:vargchil} 
        is a  $\Ocal\left(\frac{1}{N^{p}}\right)$ for each integer $p$. 
        This completes the proof of Proposition \ref{prop:moment2l-terme2} and of the uniform consistency of estimator $\tilde{\eta}_N(\theta)$.
\end{proof}

\section{Consistency of the angular estimates}
\label{sec:consistency}

We now adress the consistency of the DoA estimates defined as the local minima of function $\theta \rightarrow \tilde{\eta}_N(\theta)$. 
For this, we assume that the number of sources $K$ is fixed, i.e. that $K$ does not scale with $N$. 
In other words, model $\Sigmabs_N = \B_N + \W_N$ corresponds to a finite rank perturbation of the complex Gaussian i.i.d. matrix $\W_N$. 

\begin{remark}
        \label{re:spike}
        In this context, it is possible to derive in a simpler way than above an alternative consistent estimator, say $\theta \rightarrow \hat{\eta}_{N,spike}(\theta)$ 
        of function $\theta \mapsto \eta_N(\theta)$. 
        This estimator is obtained by assuming from the very beginning that $K$ is fixed, and is based on the recent work of Benaych \& Nadakuditi 
        \cite{benaych2011singular}; see \cite{vallet2011improved} for more details. 
        However, as shown in \cite{vallet2011improved}, estimator $\tilde{\eta}_N(\theta)$ always leads in practice to the same performance as 
        $\hat{\eta}_{N,spike}(\theta)$ if $\frac{K}{M} << 1$ (typical value $\frac{1}{10}$ in \cite{vallet2011improved}), 
        but outperforms  $\hat{\eta}_{N,spike}(\theta)$ for greater values of $\frac{K}{M}$ (typical value $\frac{1}{4}$ in \cite{vallet2011improved}). 
        Therefore, the use of estimator $\tilde{\eta}_N(\theta)$ appears in practice more relevant than $\hat{\eta}_{N,spike}(\theta)$. 
\end{remark}
In order to define the estimators of $\theta_1,\ldots,\theta_K$ properly, we consider $K$ disjoint intervals $I_1,\ldots,I_K$, such that $\theta_k \in I_k$,
and define for each $k$ the estimator  $\tilde{\theta}_{k,N}$ of $\theta_k$ by 
$
        \tilde{\theta}_{k,N} = \arg\min_{\theta \in I_k} |\tilde{\eta}_N(\theta)|.
$
We prove the following result. 
\begin{proposition}
        For $k=1,\ldots,K$, with probability one,
        \begin{align}
        	N (\tilde{\theta}_{k,N} - \theta_k) \xrightarrow[N \to \infty]{} 0.
        	\notag
        \end{align}
\end{proposition}
In order to establish the proposition, we follow a classical approach initiated by Hannan \cite{hannan1973estimation} to study sinusoid frequency estimates. 
For this, we first recall the following useful lemma. 
\begin{lemma}
        \label{lemma:sumexp}
        Let $(\alpha_M)$ a real-valued sequence of a compact subset of $(-0.5, 0.5]$, and converging to $\alpha$ as $M \to \infty$.
        Define
        $
        	q_M(\alpha_M) = \frac{1}{M} \sum_{k=1}^M e^{-i 2 \pi k \alpha_M }.
        $
        If $\alpha \neq 0$ or if $\alpha = 0$ and $M |\alpha_M| \to \infty$, then
        $
        	q_M(\alpha_M) \to 0
        $.
        If $\alpha=0$ and $M \alpha_M \xrightarrow[M\to\infty]{} \beta \in \Rbb$, then
        $
        	q_M(\alpha_M) \to e^{i \beta} \frac{\sin \beta}{\beta}
        $.
\end{lemma}
        We denote by $\A$ the matrix $\A(\Thetabs)$ corresponding the true angles $\Thetabs = (\theta_1, \ldots, \theta_K)^{T}$. 
        It is clear that 
        $
        	\eta_N(\theta) = 1 - \a(\theta)^* \A(\A^*\A)^{-1}\A^* \a(\theta)
        $.
        By the very definition of $\tilde{\theta}_{k,N}$, $|\tilde{\eta}_N(\tilde{\theta}_{k,N})|\leq |\tilde{\eta}_N(\theta_k)|$.
        From \eqref{eq:cqfd} and the equality $\eta_N(\theta_k)=0$, we have $|\tilde{\eta}_N(\tilde{\theta}_{k,N})|\to 0$ w.p.1., as $N \to \infty$.
        Consequently,
        \begin{align}
        	|\eta_N(\tilde{\theta}_{k,N})| 
        	&\leq |\eta_N(\tilde{\theta}_{k,N}) - \tilde{\eta}_N(\tilde{\theta}_{k,N})| + |\tilde{\eta}_N(\tilde{\theta}_{k,N})|
        	\notag\\
        	&\leq \sup_{\theta \in [-\pi,\pi]} |\eta_N(\theta) - \tilde{\eta}_N(\theta)|+ |\tilde{\eta}_N(\tilde{\theta}_{k,N})|
        	\notag\\
        	&\xrightarrow[N \to \infty]{a.s.} 0.
        	\label{eq:conv_eta_hattheta}
        \end{align}
        From Lemma \ref{lemma:sumexp}, $(\A^*\A)^{-1}$ converges to $\I_K$ as $N \to \infty$.
        Since $(\tilde{\theta}_{k,N})$ is bounded, we can extract a converging subsequence $\left(\tilde{\theta}_{k,\varphi(N)}\right)$.
        Let $\alpha_N = \tilde{\theta}_{k,\varphi(N)} - \theta_k$.
        From Lemma \ref{lemma:sumexp}, if $\alpha_N \to \alpha \neq 0$ as $N \to \infty$,
        then 
        \begin{align}
	\label{eq:contradiction}
        	\a(\tilde{\theta}_{k,\varphi(N)})^* \A(\A^*\A)^{-1}\A^* \a(\tilde{\theta}_{k,\varphi(N)}) \xrightarrow[N\to\infty]{a.s.} 0,        	
        \end{align}
        and thus $\eta_N(\tilde{\theta}_{k,\varphi(N)}) \to 1$, a contradiction with \eqref{eq:conv_eta_hattheta}.
        This implies that the whole sequence  $(\tilde{\theta}_{k,N})$ converges torwards $\theta_k$. 
        If $N |\tilde{\theta}_{k,N} - \theta_k|$ is not bounded, we can extract a subsequence such that 
        $N |\tilde{\theta}_{k,\phi(N)} - \theta_k| \rightarrow +\infty$ and Lemma 
        \ref{lemma:sumexp} again implies that \eqref{eq:contradiction} holds, a contradiction.  
        $N |\tilde{\theta}_{k,N} - \theta_k|$ is thus bounded, and we consider a subsequence such that  
        $N (\tilde{\theta}_{k,\varphi(N)} - \theta_k) \rightarrow \beta$ where $\beta \in [-\pi,\pi]$.  
        From Lemma \ref{lemma:sumexp}, if $\beta \neq 0$, we get
        \begin{align}
        	\eta_{\varphi(N)}(\tilde{\theta}_{k,\varphi(N)}) \xrightarrow[N\to\infty]{a.s.} 1- \left( \frac{\sin \beta}{\beta} \right)^{2} > 0,
        	\notag
        \end{align}
        which is again in contradiction with \eqref{eq:conv_eta_hattheta}. Therefore, $\beta=0$ and all converging subsequences of 
        $\left(N |\hat{\theta}_{k,\varphi(N)} - \theta_k|\right)$ converge to $0$, which of course implies that the whole sequence
        $(N |\hat{\theta}_{k,N} - \theta_k|)$ converges to $0$. 
        We finally end up with
        $
        	N (\tilde{\theta}_{k,N} - \theta_k)\to 0
        $
        w.p.1., as $N \to \infty$.

\section{Appendix}

        \subsection{\texorpdfstring{Proof of Lemma \ref{le:expre-Ehatm'N}: estimate of $\Ebb[\hat{m}'_N(z)]$}{Estimate of E[hat(m)']}}
        \label{subsec:proof-derivee}

We first give the following useful technical result. 
Its proof, based on Poincaré's inequality, is elementary and therefore omitted.
\begin{lemma}
	\label{lemma:var_no_reg}
	Let $\left(\M_N(z)\right)$ a sequence of deterministic complex $M \times M$ matrix-valued functions defined on $\Cbb\backslash\Rbb$ such that 
	\begin{align}
		\|\M_N(z)\| \leq \Prm_1(|z|) \Prm_2(|\Im(z)|^{-1}).
		\notag
	\end{align}
	Then, 
	\begin{align}
		\Var\left[\frac{1}{N} \Tr \Q_N(z) \M_N(z)\right] &\leq \frac{1}{N^2}\Prm_1(|z|) \Prm_2(|\Im(z)|^{-1}),
		\notag \\
		\Var\left[\frac{1}{N} \Tr \Sigmabs_N^*\Q_N(z) \M_N(z)\right] &\leq \frac{1}{N^2}\Prm_1(|z|) \Prm_2(|\Im(z)|^{-1}).
		\notag
	\end{align}
	Moreover, the same results still hold when $\Q_N(z)$ is replaced by  $\Q_N(z)^2$.
\end{lemma}
We are now in position to establish Lemma \ref{le:expre-Ehatm'N}. We have to establish that 
\begin{align}
        \left|\Ebb\left[\hat{m}'_N(z)\right] - m_N'(z)\right| \leq \frac{1}{N^2} \Prm_1(|z|)\Prm_2\left(|\Im(z)|^{-1}\right).
        \label{eq:ineqmder}
\end{align}
For clarity, we recall  results from \cite{dumont2009capacity}, \cite{vallet2010sub} and \cite{hachem2010bilinear}, on which the proof 
heavily relies.
We have first to introduce some new notations extensively used in \cite{dumont2009capacity}, \cite{vallet2010sub} and \cite{hachem2010bilinear}. 
We define $\delta_N(z) = \sigma c_N m_N(z) = \sigma \frac{1}{N} \Tr(\T_N(z))$, as well as $\tilde{\delta}_N(z) = \delta_N(z) - \frac{\sigma(1-c_N)}{z}$ 
which coincides with the Stieltjes transform of finite measures $c_N \mu_N + (1 - c_N) \delta_0$. In the following, matrix 
$\tilde{\T}_N(z)$ is defined by  
$$
	\tilde{\T}_N(z) = \left(-z(1+\sigma \delta_N(z)) \I_M + \frac{\B_N^* \B_N}{1+\sigma\tilde{\delta}_N(z)}\right)^{-1},
$$
and is related to $\tilde{\delta}_N(z)$ through the equation $\tilde{\delta}_N(z) = \sigma \frac{1}{N} \Tr(\tilde{\T}_N(z))$ 
(cf \cite{hachem2007deterministic}, \cite{vallet2010sub}). We also remark that matrix $\T_N(z)$ can  be written as
$$
	\T_N(z) = \left(-z(1+\sigma \tilde{\delta}_N(z)) \I_M + \frac{\B_N \B_N^*}{1+\sigma\delta_N(z)}\right)^{-1},
$$
and that $w_N(z)$ coincides with $z(1 + \sigma \delta_N(z))(1 + \sigma \tilde{\delta}_N(z))$. 
We also denote $\tilde{\Q}_N(z)$ the resolvent of matrix $\Sigmabs_N^* \Sigmabs_N$, i.e. 
$$
\tilde{\Q}_N(z) = \left( \Sigmabs_N^* \Sigmabs_N - z \I_N \right)^{-1}
$$
and define
$\alpha_N(z)= \Ebb\left[\frac{\sigma}{N} \Tr \Q_N(z)\right]$,
$\tilde{\alpha}_N(z) = \Ebb\left[\frac{\sigma}{N} \Tr \tilde{\Q}_N(z)\right]$, and the matrices
\begin{align}
	\R_N(z) &= \left(-z(1+\sigma \tilde{\alpha}_N(z)) \I_M + \frac{\B_N \B_N^*}{1+\sigma\alpha_N(z)}\right)^{-1},
	\notag \\
	\tilde{\R}_N(z) &= \left(-z(1+\sigma \alpha_N(z)) \I_N + \frac{\B_N^* \B_N}{1+\sigma\tilde{\alpha}_N(z)}\right)^{-1}. 
	\notag
\end{align}
It is shown in  \cite{dumont2009capacity} and \cite{vallet2010sub} that the entries of $\Q_N(z)$ (resp. $\tilde{\Q}_N(z)$) have the same 
behaviour as the entries of $\R_N(z)$ and $\T_N(z)$ (resp. of $\tilde{\R}_N(z)$ and $\tilde{\T}_N(z)$).  It is also useful to recall that $|\alpha_N(z)|$, 
$|\tilde{\alpha}_N(z)|$,  $\left|-z(1+\sigma\alpha_N(z))\right|^{-1}$, $\left|-z(1+\sigma\tilde{\alpha}_N(z))\right|^{-1}$,
$\|\T_N(z)\|$, $\|\tilde{\T}_N(z)\|$, $\|\R_N(z)\|$ and $\|\tilde{\R}_N(z)\|$ are bounded on $\Cbb\backslash\Rbb$ by $\Prm_1(|z|)\Prm_2(|\Im(z)|^{-1})$. 
We remark that our new notations are symetrical w.r.t. the substitution $\Sigmabs_N \rightarrow \Sigmabs_N^{*}$, and are easier to use in the forthcoming calculations. 

We first notice that \eqref{eq:ineqmder} is equivalent to 
\begin{equation}
        \label{eq:ineqmder-bis}
        \left|\alpha_N'(z) - \delta_N'(z) \right| \leq  \frac{1}{N^2} \Prm_1(|z|)\Prm_2\left(|\Im(z)|^{-1}\right). 
\end{equation}
In order to prove \eqref{eq:ineqmder-bis}, we first show that
\begin{equation}
        \label{eq:terme1}
        \left|\alpha_N'(z) - \frac{\sigma}{N} \Tr \R_N'(z)\right| \leq \frac{1}{N^2} \Prm_1(|z|)\Prm_2\left(|\Im(z)|^{-1}\right) 
\end{equation}
and deduce from this that \eqref{eq:ineqmder-bis} holds. 
Using results on the behaviour of $\alpha_N(z) - \frac{\sigma}{N} \Tr \R_N(z)$ established in \cite{dumont2009capacity}, \cite{vallet2010sub} and \cite{hachem2010bilinear}, 
we first establish that \eqref{eq:terme1} holds. For this, we recall the following lemma.
\begin{lemma}[\cite{dumont2009capacity}, \text{\cite[proof of Prop.6]{vallet2010sub}}]
        \label{lemma:alphaminustrR}
        For $z \in \Cbb\backslash\Rbb$, it holds that
        \begin{equation}
        \label{eq:master-1}
	\Ebb\left[\Q_N(z)\right] =
	\R_N(z) + \Deltabs_N(z)\R_N(z) + \left(\frac{\sigma^2}{N} \Tr \Deltabs_N(z)\right)  \Ebb\left[\Q_N(z)  \right] \R_N(z)
        \end{equation}
         where $\Deltabs_N(z)$ is given by $\Deltabs_N(z) = \Deltabs_{1,N}(z) + \Deltabs_{2,N}(z) + \Deltabs_{3,N}(z)$ with
        \begin{align}
        	\Deltabs_{1,N}(z) &= 
        	\frac{\sigma}{1+\sigma \alpha_N(z)}\Ebb\left[\Q_N(z)\Sigmabs_N\Sigmabs_N^*\frac{\sigma}{N}\Tr\left(\Q_N(z)-\Ebb\left[\Q_N(z)\right]\right)\right], 
        	\notag\\
        	\Deltabs_{2,N}(z) &= 
        	\frac{\sigma^2}{1+\sigma\alpha_N(z)} \Ebb\left[\left(\Q_N(z) - \Ebb\left[\Q_N(z)\right]\right) \frac{\sigma}{N} \Tr \Sigmabs_N^*\Q_N(z)\B_N\right],
        	\notag\\
        	\Deltabs_{3,N}(z) &= 
        	-\frac{\sigma^2}{(1+\sigma \alpha_N(z))^2} 
        	\Ebb\left[\Q_N(z)\right]  \Ebb\left[\frac{\sigma}{N} \Tr \left(\Q_N(z) - \Ebb\left[\Q_N(z)\right]\right) \frac{\sigma}{N} \Tr \Sigmabs_N^* \Q_N(z)\B_N\right].
        	\notag
        \end{align}
        If $\M_N(z)$ is a sequence of deterministic complex matrix-valued functions defined on $\Cbb\backslash\Rbb$
        such that 
        \begin{align}
	\|\M_N(z)\| \leq \Prm_1(|z|)\Prm_2(|\Im(z)|^{-1}),
	\notag
        \end{align}
        then $\frac{1}{N}\Tr \left(\Deltabs_{i,N}(z) \M_N(z)\right)$, for $i=1,2,3$, is bounded by $\frac{1}{N^2}\Prm_1(|z|)\Prm_2(|\Im(z)|^{-1})$. Therefore, 
        \begin{align}
	\alpha_N(z) = 
	        \frac{\sigma}{N} \Tr \R_N(z)
	        + \frac{\epsilon_{1,N}(z)}{N^2}.
	\label{eq:alphaN}
        \end{align}
        where $\epsilon_{1,N}(z)$ is the holomorphic function on $\Cbb\backslash\Rbb$ defined by
        \begin{align}
	\frac{\epsilon_{1,N}(z)}{N^{2}} = 
	\left(\frac{\sigma}{N} \Tr \Deltabs_N(z) \R_N(z) + \frac{\sigma}{N} \Tr \Ebb\left[\Q_N(z)\right] \R_N(z)\frac{\sigma^2}{N} \Tr \Deltabs_N(z)\right),
	\notag
        \end{align}
        and satisfies $|\epsilon_{1,N}(z)| \leq  \Prm_1(|z|)\Prm_2(|\Im(z)|^{-1})$. 
        Finally, $\tilde{\alpha}_N(z)$ can also be written as
        \begin{align}
	\tilde{\alpha}_N(z) = 
	        \frac{\sigma}{N} \Tr \tilde{\R}_N(z)
	        + \frac{\tilde{\epsilon}_{1,N}(z)}{N^2}.
	\label{eq:tildealphaN}
        \end{align}
        where $\tilde{\epsilon}_{1,N}(z)$ is equal to 
        \begin{equation}
	\tilde{\epsilon}_{1,N}(z) = \frac{1 + \sigma \tilde{\alpha}_N(z)}{1+ \sigma \alpha_N(z)} \epsilon_{1,N}(z),
	\notag
        \end{equation}
        and satisfies $|\tilde{\epsilon}_{1,N}(z)| \leq  \Prm_1(|z|)\Prm_2(|\Im(z)|^{-1})$.
\end{lemma}
In order to evaluate the behaviour of $\alpha'_N(z) - \frac{\sigma}{N} \Tr \R'_N(z)$, we differentiate \eqref{eq:alphaN} w.r.t. $z$ 
and get  the following result. 
\begin{proposition}
        \label{prop:alphaderminustrRder}
        For $z \in \Cbb\backslash\Rbb$,  it holds that the derivatives $\epsilon_{1,N}'(z)$ and $\tilde{\epsilon}_{1,N}'(z)$ of 
        $\epsilon_{1,N}(z)$ and $\tilde{\epsilon}_{1,N}(z)$ w.r.t. $z$ satisfy $|\epsilon_{1,N}'(z)| \leq \Prm_1(|z|)\Prm_2(|\Im(z)|^{-1})$
        and $|\tilde{\epsilon}_{1,N}'(z)| \leq \Prm_1(|z|)\Prm_2(|\Im(z)|^{-1})$.
\end{proposition}
\begin{proof}
        The proof uses Lemma \ref{lemma:var_no_reg} and the observation that the spectral norms $\| \R'_N(z) \|$ and$\| \tilde{\R}'_N(z) \|$ 
        are bounded by $\Prm_1(|z|)\Prm_2(|\Im(z)|^{-1})$. The details are omitted.  
\end{proof}
In order to complete the proof of the lemma, we establish that
\begin{proposition}
        \label{prop:deltaderminustrRder}
        For $z \in \Cbb\backslash\Rbb$,
        \begin{align}
	\alpha'_N(z) &= \delta_N'(z) + \frac{\epsilon_{2,N}(z)}{N^2},
	\notag\\
	\tilde{\alpha}'_N(z) &= \tilde{\delta}_N'(z) + \frac{\tilde{\epsilon}_{2,N}(z)}{N^2},
	\notag
        \end{align}
        where $|\epsilon_{2,N}(z)|$ and $|\tilde{\epsilon}_{2,N}(z)|$ are both bounded by $\Prm_1(|z|) \Prm_2\left(|\Im(z)|^{-1}\right)$.
\end{proposition}
\begin{proof}
        We first observe that \eqref{eq:alphaN} and \eqref{eq:tildealphaN} imply that
        \begin{align}
	\label{eq:eq1-approx}
	\alpha'_N(z) - \delta_N'(z) = \sigma \frac{1}{N} \Tr(\R_N'(z)) - \delta_N'(z) + \frac{\epsilon_{1,N}'(z)}{N^2},  \\
	\tilde{\alpha}'_N(z) - \tilde{\delta}_N'(z) = \sigma \frac{1}{N} \Tr(\tilde{\R}_N'(z)) - \tilde{\delta}_N'(z) + \frac{\tilde{\epsilon}'_{1,N}(z)}{N^2}. 
	\label{eq:eq2-approx}
        \end{align}
        We start with the classical identities 
        \begin{align}
	\R_N(z) - \T_N(z) = \R_N(z) \left(\T_N(z)^{-1} - \R_N(z)^{-1}\right) \T_N(z), \nonumber \\
	\tilde{\R}_N(z) - \tilde{\T}_N(z) = \tilde{\R}_N(z) \left(\tilde{\T}_N(z)^{-1} - \tilde{\R}_N(z)^{-1}\right)\tilde{\T}_N(z), \nonumber
        \end{align}
        and get that
        \begin{align}
	\frac{\sigma}{N} \Tr \left(\R_N(z) - \T_N(z)\right) = 
	\left(\tilde{\alpha}_N(z)-\tilde{\delta}_N(z)\right) z v_N(z)
	+\left(\alpha_N(z)-\delta_N(z)\right) u_N(z), 
	\label{eq:first-equation} \\
               \frac{\sigma}{N} \Tr \left(\tilde{\R}_N(z) - \tilde{\T}_N(z)\right) = 
	\left(\tilde{\alpha}_N(z)-\tilde{\delta}_N(z)\right) \tilde{u}_N(z)
	+\left(\alpha_N(z)-\delta_N(z)\right) z \tilde{v}_N(z),
	\label{eq:second-equation}
        \end{align}
        with 
        \begin{align}
	u_N(z) = \frac{\sigma^2}{N} \Tr \frac{\R_N(z)\B_N\B_N^*\T_N(z)}{(1+\sigma\alpha_N(z))(1+\sigma\delta_N(z))},
	\quad\quad
	\tilde{u}_N(z) = \frac{\sigma^2}{N} \Tr \frac{\tilde{\R}_N(z)\B_N^*\B_N\tilde{\T}_N(z)}{(1+\sigma\tilde{\alpha}_N(z))(1+\sigma\tilde{\delta}_N(z))},
	\notag
        \end{align}
        and
        \begin{align}
	v_N(z) = \frac{\sigma^2}{N} \Tr \R_N(z)\T_N(z)
	\quad\quad
	\tilde{v}_N(z) = \frac{\sigma^2}{N} \Tr \tilde{\R}_N(z)\tilde{\T}_N(z).
	\notag
        \end{align}
        Note that it is easy to check that $u_N(z) = \tilde{u}_N(z)$. 
        We differentiate \eqref{eq:first-equation}, \eqref{eq:second-equation}) w.r.t. $z$, we use \eqref{eq:eq1-approx}, \eqref{eq:eq2-approx}  and Proposition 
        \ref{prop:alphaderminustrRder}, and recall that both $|\alpha_N(z) - \delta_N(z)|$ and $|\tilde{\alpha}_N(z) - \tilde{\delta}_N(z)|$ are bounded that 
        $\frac{1}{N^{2}} \Prm_1(|z|) \Prm_2\left(|\Im(z)|^{-1}\right)$ (see \cite{vallet2010sub}). 
        We check that $u_N(z)$, $z v_N(z)$, $z \tilde{v}_N(z)$ are their derivatives are bounded by $\Prm_1(|z|) \Prm_2\left(|\Im(z)|^{-1}\right)$, and obtain 
        eventually that 
        \begin{align}
	\begin{bmatrix}
		\alpha'_N(z) - \delta'_N(z)
		\\
		\tilde{\alpha}'_N(z) - \tilde{\delta}'_N(z)
	\end{bmatrix}
	=
	\begin{bmatrix}
		u_N(z) & z v_N(z)
		\\
		z \tilde{v}_N(z) & u_N(z)
	\end{bmatrix}
	\begin{bmatrix}
		\alpha'_N(z) - \delta'_N(z)
		\\
		\tilde{\alpha}'_N(z) - \tilde{\delta}'_N(z)
	\end{bmatrix}
	+
	\frac{1}{N^2}
	\begin{bmatrix}
		\epsilon_{3,N}(z)
		\\
		\tilde{\epsilon}_{3,N}(z)
	\end{bmatrix},
	\notag
        \end{align}
        with $|\epsilon_{3,N}(z)|$, $|\tilde{\epsilon}_{3,N}(z)|$ bounded by $\Prm_1(|z|) \Prm_2\left(|\Im(z)|^{-1}\right)$.
        We denote by $\Delta_N(z)$  the determinant of the above system, i.e. 
        \begin{equation}
	\label{eq:def-determinant}
	\Delta_N(z) = (1-u_N(z))^2 - z v_N(z) \tilde{v}_N(z).
        \end{equation}
        The determinant $\Delta_N(z)$ was studied in \cite{hachem2010bilinear} and in \cite{vallet2010sub} where it was proved that 
        $\left|\Delta_N(z)^{-1}\right| \leq \Prm_1(|z|) \Prm_2\left(|\Im(z)|^{-1}\right)$ on a subset $\Dcal_N$ of $\Cbb$ defined as 
        $$\Dcal_N = \left\{ z \in \Cbb-\Rbb, \frac{1}{N^{2}} \Qrm_1(|z|) \Qrm_2\left(|\Im(z)|^{-1}\right) < 1 \right\}$$
	where $ \Qrm_1$ and $ \Qrm_2$ are 2 polynomials independent of $N$. Thus, we can invert the previous system on $\Dcal_N$ to get
	\begin{align}
		\begin{bmatrix}
			\alpha'_N(z) - \delta'_N(z)
			\\
			\tilde{\alpha}'_N(z) - \tilde{\delta}'_N(z)
		\end{bmatrix}
		=
	               \frac{1}{\Delta_N(z)}
		\begin{bmatrix}
			1-u_N(z) & z v_N(z)
			\\
			z \tilde{v}_N(z)  & 1-u_N(z)
		\end{bmatrix}
                       \frac{1}{N^{2}}
		\begin{bmatrix}
			\epsilon_{3,N}(z)
			\\
			\tilde{\epsilon}_{3,N}(z)
		\end{bmatrix}.
		\notag
	\end{align}
	This implies that $|\alpha'_N(z) - \delta'_N(z)|$ is bounded by  $ \frac{1}{N^{2}} \Prm_1(|z|) \Prm_2\left(|\Im(z)|^{-1}\right)$
	on $\Dcal_N$. If $z \in \Cbb\backslash\left\{\Rbb\cup\Dcal_N\right\}$, we use the trick in \cite{haagerup2005new}. 
	We remark that 
	$$
	        |\alpha'_N(z) - \delta'_N(z)| \leq |\alpha'_N(z)|+ |\delta'_N(z)| \leq \frac{C}{|\Im z|}, 
	$$
	for each $z$, and that $1 \leq  \frac{1}{N^{2}} \Qrm_1(|z|) \Qrm_2\left(|\Im(z)|^{-1}\right)$ on $\Cbb\backslash\left\{\Rbb\cup\Dcal_N\right\}$. 
	Therefore, 
	$$
	        |\alpha'_N(z) - \delta'_N(z)| \leq 
	        \frac{C}{|\Im z|}  \frac{1}{N^{2}}   \Qrm_1(|z|) \Qrm_2\left(|\Im(z)|^{-1}\right) \leq \frac{1}{N^{2}} \Prm_1(|z|) \Prm_2\left(|\Im(z)|^{-1}\right)
	$$
	on $ \Cbb\backslash\left\{\Rbb\cup\Dcal_N\right\}$. This in turn shows that \eqref{eq:ineqmder-bis} holds on $\Cbb\backslash\Rbb$. 
\end{proof}

	\subsection{\texorpdfstring{Proof of lemma \ref{le:differentiabilite-trpsihatomega}: differentiability of $\frac{1}{M} \Tr \tilde{\Psi}(\hat{\Omegabs}_N)$}{Differentiability of the trace of psi}}
	 \label{sec:proof_diff}

We first need to establish the following useful Lemma.  
\begin{lemma}
\label{lm:extension} 
Given an integer $D > 0$, let $f$ be a continuous real function on $\Rbb^D$. 
Let $\Ocal$ be an open set of $\Rbb^D$ such that $\Rbb^D \backslash \Ocal$
has a zero Lebesgue measure. Assume that $f$ is a $\Ccal^1$ function on 
$\Ocal$ and that its gradient $f'$ on $\Ocal$ can be continuously
extended to $\Rbb^D$. Then $f$ is $\Ccal^1$ on the whole $\Rbb^D$ with 
gradient $f'$. 
\end{lemma} 
\begin{proof}
We only need to prove that for any $x \in \Rbb^D - {\mathcal O}$ and any 
sequence $x_n \to x$, 
\[
f(x_n) - f(x) = \langle f'(x), x_n - x \rangle + o(d_n)  . 
\] 
where $d_n = \| x_n - x \|$. Since $f$ is uniformly continuous on any small 
neighborhood of $x$, there exists a sequence $\delta_n$ 
such that for every $y$ and $y'$ in this neighborhood for which 
$\| y - y' \| < \delta_n$, $| f(y) - f(y') | \leq d_n^2$. Since
$\Rbb^D - {\mathcal O}$ has a zero Lebesgue measure, 
there exists $y_n$ and $z_n$ in ${\mathcal O}$ such that 
\[
\| x_n - y_n \| < \min(\delta_n, d_n^2) \ 
\text{and} \ 
\| x - z_n \| < \min(\delta_n, d_n^2)  . 
\]
Therefore, it holds that  $\max(|f(x_n) - f(y_n) | , |f(z_n) - f(x) | ) < d_n^2$.
Writing $f(x_n) - f(x) = f(x_n) - f(y_n) + f(y_n) - f(z_n) + f(z_n) - f(x)$,
we obtain that $ f(x_n) - f(x) = f(y_n) - f(z_n) + o(d_n)$. 
By differentiability of $f$ on ${\mathcal O}$ and continuity of $f'$ at $x$, 
\[
f(y_n) - f(z_n) = \langle f'(z_n), y_n - z_n \rangle + o(\| y_n - z_n \|) 
= \langle f'(x), x_n - x \rangle + o(d_n)
\]
which proves the lemma. 
\end{proof}	 
We now complete the proof of the Lemma. 
We consider $\tilde{\Psi} \in \Ccal_c^{\infty}(\Rbb,\Rbb)$, and establish that, considered as a function of the real and imaginary parts of 
$\W_N$, function $\frac{1}{M} \Tr \tilde{\Psi}(\hat{\Omegabs}_N)$ is continuously differentiable on $\Rbb^{2MN}$, i.e. that 
for each pair $(i,j)$, the partial derivatives 
$$
        \frac{\partial}{\partial W_{i,j,N}} \left\{\frac{1}{M} \Tr \tilde{\Psi} \left(\hat{\Omegabs}_N\right)\right\}
$$
exist, and are continuous  
\footnote
{
	$\tilde{\Psi}$ is real valued, the partial derivatives w.r.t. $\overline{W}_{i,j,N}$ thus coincide 
	with the complex conjugate of the partial derivative w.r.t. $W_{i,j,N}$. It is therefore sufficient to consider these derivatives.
} 
We denote by $\Ocal$ the open subset of $\Rbb^{2MN}$ for which the eigenvalues $(\hat{\lambda}_{l,N})_{l=1, \ldots, M}$ 
of $\Sigmabs_N \Sigmabs_N^{*}$ have multiplicity 1. It is clear that  $\Rbb^{2MN} \backslash \Ocal$ has a zero Lebesgue measure. On 
$\Ocal$, it is standard that the eigenvalues $(\hat{\lambda}_{l,N})_{l=1, \ldots, M}$ are $\Ccal^1$ functions and that 
\begin{equation}
	\label{eq:deriv-vap} 
	\frac{\partial \hat\lambda_{l,N}} {\partial W_{i,j,N}} 
	= 
	\left[ {\Sigmabs}_N^* \hat{\Pibs}_{l,N} \right]_{j,i} . 
\end{equation} 
Using Lemma 4.6 in Haagerup-Thorbjornsen \cite{haagerup2005new}, we obtain
\begin{equation}
	\label{eq:deriv}
	\frac{\partial}{\partial W_{i,j,N}} 
	\left\{ \Tr \tilde\psi(\hat{\Omegabs}_N) \right\} = 
	\Tr\left( 
	\tilde\psi'(\hat{\Omegabs}_N) 
	\frac{\partial}{\partial W_{i,j,N}} \{ \hat{\Omegabs}_N \} \right) 
	= \left[
	{\Sigmabs}_N^* \sum_{l=1}^M [ \tilde\psi'(\hat{\Omegabs}_N) ]_{ll} 
	\hat{\Pibs}_{l,N}
	\right]_{j,i}
\end{equation}
and get that  $\frac{1}{M} \Tr \tilde{\Psi}(\hat{\Omegabs}_N)$ is a  $\Ccal^1$ on $\Ocal$. By Lemma \ref{lm:extension}, 
it remains to establish that the righthandside of \eqref{eq:deriv} can be continuously extended to any point 
$\W_{N}^{0}$ of  $\Rbb^{2MN} \backslash \Ocal$. For this, we first prove the following useful result. 
\begin{lemma}
	\label{lm-psi-cst}
	If $\hat\lambda_{k,N} = \hat\lambda_{l,N}$, then 
	$[ \tilde\psi(\hat{\Omegabs}_N) ]_{kk} = 
	[ \tilde\psi(\hat{\Omegabs}_N) ]_{ll}$. 
\end{lemma}
\begin{proof}
	We start by observing that for any integers $m_1, m_2, \ldots, m_t$, 
	matrix 
	$\A = \hat{\Lambdabs}_N^{m_1} {\bf 11}^T \hat{\Lambdabs}_N^{m_2} \cdots 
	{\bf 11}^T \hat{\Lambdabs}_N^{m_t}$ writes
	\[
	\A =  
	\begin{bmatrix} 
		\hat\lambda_{1,N}^{m_1} & \cdots & \hat\lambda_{1,N}^{m_1} \\
		\vdots & \vdots & \vdots \\
		\hat\lambda_{M,N}^{m_1} & \cdots & \hat\lambda_{M,N}^{m_1} 
	\end{bmatrix} 
	\cdots 
	\begin{bmatrix} 
		\hat\lambda_{1,N}^{m_{t-1}} & \cdots & \hat\lambda_{1,N}^{m_{t-1}} \\
		\vdots & \vdots & \vdots \\
		\hat\lambda_{M,N}^{m_{t-1}} & \cdots & \hat\lambda_{M,N}^{m_{t-1}} 
	\end{bmatrix} 
	\hat{\Lambdabs}_N^{m_t}
	\]
	hence $[\A]_{kk} = [\A]_{ll}$ if 
	$\hat\lambda_{k,N} = \hat\lambda_{l,N}$. The same can be said about 
	${\bf 11}^T \A$ and $\A {\bf 11}^T$. Consequently, the result of
	the lemma is true when $\tilde\psi$ is a polynomial. 
	Since any continuous function $\tilde\psi$ is the uniform limit of a sequence of 
	polynomials on compact subsets of $\Rbb$, the result is true for such 
	$\tilde\psi$. 
\end{proof} 
We consider an element $\W_N^{0}$ of $\Rbb^{2MN} \backslash \Ocal$, and denote by $m_1,\ldots,m_L$, with $M = \sum_{l=1}^L m_l$, the respective  multiplicities of 
the eigenvalues of  $\Sigmabs_{N}^{0} \Sigmabs_{N}^{0*}$ where $\Sigmabs_{N}^{0} = \B_N + \W_N^{0}$. 
We also denote by $(\overline{\Pibs}_{l,N})_{l=1, \ldots, L}$ the orthogonal projection matrices over the corresponding eigenspaces. 
Lemma \ref{lm-psi-cst} implies that for each $i=1, \ldots, L$, 
\begin{align}
	\left[ \tilde{\psi}^{'}(\hat{\Omegabs}) \right]_{m_1+\ldots + m_i,m_1+\ldots + m_i} 
	=  
	\ldots 
	= 
	\left[ \tilde{\psi}^{'}(\hat{\Omegabs}) \right]_{m_1+\ldots + m_i+m_{i+1}-1,m_1+\ldots + m_i+ m_{i+1}-1} 
	= 
	\kappa_i.
	\notag
\end{align}
Therefore, for any sequence $(\W_{N,n})_{n \in \mathbb{N}}$ converging torward $\W_N^{0}$, it holds that 
\begin{align}
	\lim_{n \to \infty} 
	\left. 
		\frac{\partial}{\partial W_{i,j,N}} \left\{\frac{1}{M} \Tr \tilde{\Psi} \left(\hat{\Omegabs}_N\right)\right\}
	\right|_{\W_N = \W_{N,n}} 
	= 
	\left[\Sigmabs_N^* \sum_{l=1}^L \kappa_l \overline{\Pibs}_{l,N} \right]_{j,i}.
	\notag
\end{align}
This completes the proof of Lemma \ref{le:differentiabilite-trpsihatomega}. 

\subsection{\texorpdfstring{Proof of lemma \ref{le:normeWW*}: uniform boundedness of $\Exp[\|\W_N\|^p]$}{Uniform boundedness of the norm of W}}
\label{sec:bound_E_W}

It is clear that it is sufficient to prove the boundedness of $\Exp[\|\W_N\|^p]$ if the entries of $\W_N$ are real.  
We thus consider the case of real matrices and denote by $X_N$ the largest singular value of $\frac{\W_N}{\sigma}$. 
The following concentration result is well-known.
\begin{theorem}[\text{\cite[Th. II.13]{davidson2001local}}]
	\label{theorem:davidson}
	It holds that $\Ebb\left[X_N\right] \leq 1 + \sqrt{c_N}$ and for all $t > 0$, $\Pbb\left(X_N > 1+\sqrt{c_N}+t\right) \leq \exp\left(- N t^2 /2\right)$.
\end{theorem}
Using Theorem \ref{theorem:davidson} and for $p \geq 2$ the inequality,
\begin{align}
	\Ebb[X_N^p]
	= \int_{0}^{+\infty}\Pbb\left(X_N \geq t\right) p t^{p-1} \drm t
	\leq p \left(1 + \sqrt{c_N}\right)^p +  \int_{0}^{+\infty}\Pbb\left(X_N \geq t + 1 + \sqrt{c_N}\right) p (t+ 1 + \sqrt{c_N})^{p-1} \drm t,
	\notag
\end{align}
we easily obtain $\Ebb[X_N^p] \leq K < \infty$, with $K$ a constant independent of $N$, for all $p \in \Nbb$.

\subsection{Proof of Lemma \ref{le:differentiability-chi}: differentiability of the regularization factor}

We first establish that $\det \phi(\Sigmabs_N \Sigmabs_N^*)$ is a $\Ccal^1$ function, and that \eqref{eq:deriv-detSigma} holds. 
We use the same approach as in Haagerup \& Thorbjornsen \cite[Lem. 4.6]{haagerup2005new}.
We start begin by showing that the differential of $\det\phi(X)$ is given by 
\begin{equation} 
	\label{eq-diff-psi-intermediate} 
	\det\phi(\X)'.\H = \Tr\left( \adj(\phi(\X)) \phi'(\X) \H \right). 
\end{equation} 
As $\det(\X)'.\H = \Tr(\adj(\X) \H)$ and $(\X^n)'.\H = \sum_{i=0}^{n-1} \X^i \H \X^{n-1-i}$ for any $n\in \Nbb$, we have  
\begin{align}
	\det(\X^n)'.\H = \Tr\left( \adj(\X^n) (n \X^{n-1}) \H \right)
	\notag
\end{align}
since $\adj(\X^n)$ and $\X$ commute. 
So \eqref{eq-diff-psi-intermediate} is true when $\phi$ is a polynomial. 
By choosing a sequence of polynomials $P_n$ such that $P_n \to \phi$ and $P'_n \to \phi'$ uniformly on compact subsets of $\Rbb$, we generalize 
\eqref{eq-diff-psi-intermediate} to any $\phi \in {\cal C}_1$. 
Now one can check that 
\begin{align}
	\frac{\partial (\Sigmabs_N \Sigmabs_N^* )}{\partial W_{i,j,N}},
	= \e_i \e_j^* \Sigmabs_{N}^*, 
\end{align}
and it remains to apply the composition formula for differentials to obtain \eqref{eq:deriv-detSigma}. 

We also remark that at a point ${\bf W}_N$ for which there exists a 
$\hat\lambda_{l,N} \not\in \supp(\phi)$, we have 
\[
	\adj\left( \phi({\Sigmabs}_N {\Sigmabs_N}^*) \right) 
	\phi'({\Sigmabs}_N {\Sigmabs_N}^*) 
	= 
	\sum_{l=1}^M \Bigl(\prod_{k\neq l} \phi(\hat\lambda_{k,N}) \Bigr) 
	\phi'(\hat\lambda_{l,N}) {\bf u}_l {\bf u}_l^* = {\bf 0} 
\]
hence the derivative \eqref{eq:deriv-detSigma} is zero on $\Acal_{1,N}^{c}$. \\

It is easy to check that $\det \phi(\hat{\Omegabs}_N \hat{\Omegabs}_N^*)$ is a ${\cal C}_1$ function on the open set $\Ocal$ of all matrices 
$\W_N$ for which the eigenvalues of $\Sigmabs_N \Sigmabs_N^*$ are simple, and that (\ref{eq:deriv-detOmega}) holds if $\W_N \in \Ocal$, 
i.e. on a set of probability 1. In order to show that  $\det \phi(\hat{\Omegabs}_N \hat{\Omegabs}_N^*)$ is a  ${\cal C}_1$ function
on $\Rbb^{2MN} \backslash \Ocal$, we use again Lemma \ref{lm:extension}, and verify that \eqref{eq:deriv-detOmega} can be continuously extended
to  $\Rbb^{2MN} \backslash \Ocal$. For this, we claim that 
\begin{equation}
	\label{eq:egalitekl}
	\left[ \adj(\phi(\hat{\Omegabs}_N)) \phi'(\hat{\Omegabs}_N) \right]_{k,k} 
	= 
	\left[ \adj(\phi(\hat{\Omegabs}_N)) \phi'(\hat{\Omegabs}_N) \right]_{l,l}
\end{equation}
if $\hat{\lambda}_{k,N} = \hat{\lambda}_{l,N}$. Indeed, given $\varepsilon > 0$, let $\phi_\varepsilon(x) = \phi(x) + \varepsilon$. 
Since $\phi_\varepsilon(\hat{\Omegabs}_N) > 0$,  
\[
	\adj(\phi_\varepsilon(\hat{\Omegabs}_N)) 
	\phi_\varepsilon'(\hat{\Omegabs}_N) = 
	\det(\phi_\varepsilon(\hat{\Omegabs}_N)) 
	\phi_\varepsilon^{-1}(\hat{\Omegabs}_N) 
	\phi_\varepsilon'(\hat{\Omegabs}_N) .
\]
Applying Lemma \ref{lm-psi-cst} to $\tilde\psi = \phi_\varepsilon^{-1} 
\times \phi_\varepsilon'$, we obtain that 
\[
	\left[ \adj(\phi_\varepsilon(\hat{\Omegabs}_N)) 
	\phi_\varepsilon'(\hat{\Omegabs}_N) \right]_{kk} 
	= 
	\left[ \adj(\phi_\varepsilon(\hat{\Omegabs}_N)) 
	\phi_\varepsilon'(\hat{\Omegabs}_N) \right]_{ll} 
	\quad 
	\text{if} \ \hat\lambda_{k,N} = \hat\lambda_{l,N}  
\]
and letting $\varepsilon \to 0$, we obtain the same result for $\adj(\phi(\hat{\Omegabs}_N)) \phi'(\hat{\Omegabs}_N)$. 
Similarly to the proof of Lemma \ref{le:differentiabilite-trpsihatomega}, this proves that \eqref{eq:deriv-detOmega} can be continuously extended
to  $\Rbb^{2MN} \backslash \Ocal$.

\subsection{Proof of lemma \ref{le:biais-gchi2}: various estimates}

In this section, we denote by $\alpha_{r,N}(z), \tilde{\alpha}_{r,N}(z), \R_{r,N}(z)$ and $\tilde{\R}_{r,N}(z)$ the regularized versions
of the respective functions $\alpha_{N}(z), \tilde{\alpha}_{N}(z), \R_{N}(z)$ and $\tilde{\R}_{N}(z)$ defined in Section \ref{subsec:proof-derivee}, i.e.
\begin{align}
	\alpha_{r,N}(z) = \sigma \Ebb\left( \frac{1}{N} \Tr(\Q_N(z)) \chi_N \right)
	\quad\text{and}\quad
	\tilde{\alpha}_{r,N}(z) = \sigma \Ebb\left( \frac{1}{N} \Tr(\tilde{\Q}_N(z)) \chi_N \right),
	\notag
\end{align}
and
\begin{align}
	\R_{r,N}(z) = \left(\frac{\B_N\B_N^*}{1+\sigma\alpha_{r,N}(z)} - z(1+\sigma\tilde{\alpha}_{r,N}(z))\right)^{-1},
	\notag
	\tilde{\R}_{r,N}(z)=\left(\frac{\B_N^*\B_N}{1+\sigma\tilde{\alpha}_{r,N}(z)} - z(1+\sigma\alpha_{r,N}(z))\right)^{-1}.
	\notag
\end{align}
It is clear that $\alpha_{r,N}$ and $\tilde{\alpha}_{r,N}$ are the Stieltjes transforms of positive measures carried by $\Cbb\backslash\supp(\phi)$
and $\Cbb^{*}\backslash\supp(\phi)$ respectively and with mass $\sigma c_N \Ebb[\chi_N]$ and $\sigma  \Ebb[\chi_N]$. 
This implies that the following uniform bounds hold: Let $\Kcal$ and $\tilde{\Kcal}$ be compact subsets of $\Cbb\backslash\supp(\phi)$ and  
$\Cbb^{*}\backslash\supp(\phi)$ respectively, then we have 
\begin{equation}
        \label{eq:borne-alphar}
        \sup_{z \in {\cal K}} |\alpha_{r,N}(z)| < C \quad\text{and}\quad \sup_{z \in \tilde{{\cal K}}} |\tilde{\alpha}_{r,N}(z)| < C.
\end{equation}
In order to establish Lemma \ref{le:biais-gchi2}, it is necessary to show that similar bounds hold for functions $\frac{1}{1+\sigma \alpha_{r,N}(z)}$, 
$\| \R_{r,N}(z) \|$ and  $\| \tilde{\R}_{r,N}(z) \|$. For this, we introduce function 
$w_{r,N}(z) = z(1+\sigma \alpha_{r,N}(z))(1+\sigma \tilde{\alpha}_{r,N}(z))$  and prove the following lemma
\begin{lemma}
        \label{le:bornes-regularisees}
        For any compact subset $\Kcal$ of $\Cbb\backslash\supp(\phi)$, it holds that
        \begin{align}
	\label{eq:cvuni-alphar}
	\sup_{z\in {\cal K}} \left|\alpha_{r,N}(z)-\delta_N(z)\right| \xrightarrow[N \to \infty]{} 0, 
	\\
	\inf_{z \in {\cal K}} \min_{k=1, \ldots, M} \left| \lambda_{k,N} - w_{r,N}(z) \right| > C > 0.
	\label{eq:distance-lambda-wr}
        \end{align}
\end{lemma}
\begin{proof}
	Define $\kappa_N(z):=\alpha_{r,N}(z) - \delta_N(z)$ where we recall that $\delta_N(z) = \sigma c_N m_N(z) = \frac{\sigma}{N} \, \Tr(\T_N(z))$. 
	Since $\delta_N(z)$ and $\alpha_{r,N}(z)$ are Stieltjes transforms of positive measures 
	carried by $\Cbb\backslash\supp(\phi)$, $\kappa_N$ is holomorphic on $\Cbb\backslash\supp(\phi)$ and satisfies 
	\begin{align}
	        |\kappa_N(z)|\leq \frac{C}{\drm(z,\supp(\phi))}.
	        \notag
	\end{align}
	This implies that the sequence $(\kappa_N)$ is uniformly bounded on each compact subset of $\Cbb\backslash\supp(\phi)$. 
	By Montel's theorem, $(\kappa_N)$ is a normal family. 
	Let $(\kappa_{\psi(N)})$ a subsequence of $(\kappa_N)$ which converges uniformly to $\kappa$ on each compact subset of $\Cbb\backslash\supp(\phi))$ . 
	Then $\kappa$ is holomorphic on $\Cbb\backslash\supp(\phi)$. 
	From \cite[Prop.6]{vallet2010sub},
	$
	        \Ebb\left[\frac{1}{N} \Tr \Q_N(z)\right] - \frac{1}{N} \Tr \T_N(z) \xrightarrow[N]{} 0 
	$
	for $z \in \Cbb\backslash\Rbb^+$ and since $\chi_N \to_N 1$ a.s., dominated convergence theorem implies 
	\begin{align}
	        \kappa_N(z) = \Ebb\left[\frac{\sigma}{N} \Tr \Q_N(z) \chi_N\right] - \frac{\sigma}{N} \Tr \T_N(z) \xrightarrow[N]{} 0
	        \notag
	\end{align}
	for $z \in \Cbb\backslash\Rbb^+$.
	Thus, $\kappa(z) = 0$ for $z \in \Cbb\backslash\Rbb^+$, and by analytic continuation, $\kappa(z)=0$ for all $z \in \Cbb\backslash\supp(\phi)$.
	Therefore, all converging subsequences extracted from the normal family $(\kappa_N(z))$ converge to $0$ uniformly on each compact subset of 
	$\Cbb\backslash\supp(\phi)$.
	Consequently, the whole sequence $(\kappa_N)$ converges uniformly to $0$ on each compact subset of $\Cbb\backslash\supp(\phi)$. This completes the proof
	of (\ref{eq:cvuni-alphar}). 
	We also notice that 
	\begin{equation}
	        \label{eq:expre-tildealphar}
	        \tilde{\alpha}_{r,N}(z) = \alpha_{r,N}(z) - \frac{\sigma (1 - c_N)}{z} + \frac{\sigma (1 - c_N)}{z}  \left( 1 - \Ebb(\chi_N) \right)
	\end{equation}
	and recall that $\tilde{\delta}_N(z) = \delta_N(z) -  \frac{\sigma (1 - c_N)}{z}$. As $1 - \Ebb(\chi_N) = \Ocal\left(\frac{1}{N^{p}}\right)$ for each 
	integer $p$, \eqref{eq:cvuni-alphar} implies
	\begin{align}
	        \sup_{z\in {\cal K}} \left|z(\tilde{\alpha}_{r,N}(z)-\tilde{\delta}_N(z))\right| \rightarrow 0.
	        \notag
	\end{align}
	Hence, it holds that
	$$
	        \sup_{z\in {\cal K}} \left| w_{r,N}(z) - w_N(z) \right| \rightarrow 0.
	$$
	Thus, \eqref{eq:distance-lambda-wr} follows immediately from \eqref{eq:distancewN-S}.
\end{proof}
Lemma \ref{le:bornes-regularisees} immediately implies that the following uniform bounds hold.
\begin{lemma}
        \label{le:bornes-utiles-regularisees}
        Let ${\cal K}$ and $\tilde{{\cal K}}$ be compact subsets of $\Cbb - \supp(\phi)$ and  $\Cbb^{*} - \supp(\phi)$ respectively. 
        For $N$ large enough, we have
        \begin{align}
	\label{eq:bound1plussigmaalphareg}
	\sup_{z \in {\cal K}} \left| \frac{1}{1+\sigma \alpha_{r,N}(z)} \right| < C, \\
	\label{eq:boundNR}
	\sup_{z \in {\cal K}} \| \R_{r,N}(z) \| < C, \\
	\label{eq:boundNtildeR}
	\sup_{z \in \tilde{{\cal K}}} \| \tilde{\R}_{r,N}(z) \| < C, \\
	\label{eq:cvuniR-T}
	\sup_{z \in {\cal K}} \| \R_{r,N}(z) - \T_N(z) \| \rightarrow 0, \\
	\label{eq:cvunitildeR-tildeT}
	\sup_{z \in \tilde{{\cal K}}} \| \tilde{\R}_{r,N}(z) -\tilde{\T}_N(z) \|  \rightarrow 0.
\end{align}
\end{lemma}
\begin{proof}
        We first recall that inequality \eqref{eq:inegalite-Reb} holds. Therefore, the uniform convergence result \eqref{eq:cvuni-alphar}
        implies that 
        $$
	\inf_{z \in {\cal K}} \left| 1 + \sigma \alpha_{r,N}(z) \right| > \frac{1}{4}
        $$
        for $N$ large enough. This establishes \eqref{eq:bound1plussigmaalphareg} that holds for $N$ large enough. 
        In order to prove \eqref{eq:boundNR}, we express $\R_{r,N}(z)$ as 
        $$
	\R_{r,N}(z) = \left(1 + \sigma \alpha_{r,N}(z) \right) \left( \B_N \B_N^* - w_{r,N}(z) \right)^{-1}
        $$
        and use \eqref{eq:borne-alphar} and \eqref{eq:distance-lambda-wr}. The proof of \eqref{eq:boundNtildeR} is similar, and is based on 
        the identity 
        $$
	\tilde{\R}_{r,N}(z) = \left(1 + \sigma \tilde{\alpha}_{r,N}(z) \right) \left( \B_N^* \B_N - w_{r,N}(z) \right)^{-1}.
        $$
        We remark that function $\tilde{\alpha}_{r,N}(z)$ has a pole at $z=0$. Hence, any compact $\tilde{{\cal K}}$ over which 
        $\| \tilde{\R}_{r,N}(z) \|$ is supposed to be uniformly bounded should not contain $0$. The proof of (\ref{eq:cvuniR-T}) follows immediately from 
        \eqref{eq:cvuni-alphar} and from \eqref{eq:bound1plussigmaalphareg}, \eqref{eq:boundNR}, \eqref{eq:boundNtildeR}. 
        Finally, to establish \eqref{eq:cvunitildeR-tildeT}, we remark that 
        \begin{align}
	\tilde{\R}_{r,N}(z) 
	= \frac{\B_N^* \R_N(z) \B_N}{w_{r,N}(z)} - \frac{\I_N}{1 + \sigma \alpha_{r,N}(z)},
	\notag\\
	\tilde{\T}_{r,N}(z) 
	= \frac{\B_N^* \T_N(z) \B_N}{w_{N}(z)} - \frac{\I_N}{1 + \sigma \delta_{N}(z)},
	\notag
        \end{align}
        and that $|w_{r,N}(z)|$ and $|w_N(z)|$ are uniformly bounded from below by \eqref{eq:distancewN-S} and \eqref{eq:distance-lambda-wr} 
        (recall that $0$ is one of the eigenvalues of $\B_N \B_N^*$). 
\end{proof}
We now establish \eqref{eq:biaisfqchi} and \eqref{eq:biaismnchi}. 
In order to prove that $\alpha_N(z) - \delta_N(z) = \Ocal\left(\frac{1}{N^{2}}\right)$ on $\Cbb \backslash \Rbb^{+}$, \cite{dumont2009capacity} and 
\cite{vallet2010sub} used the integration by parts formula (see e.g. \cite{pastur2005simple}) and the Poincaré inequality to show that the entries of 
$\Ebb[\Q_N(z)]$ are close from the entries of $\R_N(z)$ (see the fundamental equation \eqref{eq:master-1}). 
Then, $\alpha_N(z) - \delta_N(z)$ was evaluated by solving a linear system whose determinant $\Delta_N(z)$ given by \eqref{eq:def-determinant} was shown to be bounded from below.  
Lemma \ref{le:bornes-utiles-regularisees} allows to follow exactly the same approach to establish \eqref{eq:biaisfqchi} and \eqref{eq:biaismnchi}. 
However, functions $\alpha_N, \tilde{\alpha}_N, \R_N, \tilde{\R}_N$ have to be replaced by their regularized versions. 
The following results show that the presence of the regularization term $\chi_N$ does not modify essentially the calculations of \cite{dumont2009capacity} and \cite{vallet2010sub}. 
We first indicate how the integration by parts formula is modified. $\Vec(.)$ denotes the column by column vectorization operator of a matrix. 
\begin{lemma}
        \label{lemma:IPPder}
	Let $(f_N)_{N   \geq 1}$ be a sequence of continuously differentiable functions defined on $\Cbb^{M(M+N)}$ with 
	polynomially bounded partial derivatives satisfying the condition 
	$$
	        \sup_{z \in \partial \Rcal_y} \left| f_N\left(\Vec\left(\Q_N(z)\right),\Vec(\Sigmabs_N)\right) \chi_N \right| < C. 
	$$
	Then, for all $p \in \Nbb$, we have
	\begin{align}
		\Ebb\left[f\left(\Vec\left(\Q_N(z)\right),\Vec(\Sigmabs_N)\right) \chi_N \right]
		=
		\Ebb\left[f\left(\Vec\left(\Q_N(z)\right),\Vec(\Sigmabs_N)\right)\chi_N^k\right]
		+
		\frac{\epsilon_{1,N}(z)}{N^p}.
		\label{eq:IPPder1}
	\end{align}
	for all $k \in \Nbb^*$, and	
	\begin{align}
		\Ebb\left[ W_{ij,N} f\left(\Vec\left(\Q_N(z)\right),\Vec(\Sigmabs_N)\right) \chi_N \right]
		&=
		\frac{\sigma^2}{N} \Ebb\left[\frac{\partial f\left(\Vec\left(\Q_N(z)\right),\Vec(\Sigmabs_N)\right)}{\partial \overline{W}_{ij,N}} \chi_N\right]
		+
		\frac{\epsilon_{2,N}(z)}{N^p},
		\label{eq:IPPder2}\\
		\Ebb\left[ \overline{W}_{ij,N} f\left(\Vec\left(\Q_N(z)\right),\Vec(\Sigmabs_N)\right) \chi_N \right]
		&=
		\frac{\sigma^2}{N} \Ebb\left[\frac{\partial f\left(\Vec\left(\Q_N(z)\right),\Vec(\Sigmabs_N)\right)}{\partial W_{ij,N}} \chi_N\right]
		+
		\frac{\epsilon_{3,N}(z)}{N^p},
		\label{eq:IPPder3}
	\end{align}
	with $\sup_{z \in \partial\Rcal_y}|\epsilon_{i,N}(z)| \leq C < \infty$.	
\end{lemma}
As for the use of the Poincaré inequality, we have:
\begin{lemma}
	\label{lemma:var_reg}
	Let $\left(\M_N(z)\right)$ a sequence of deterministic complex $M \times M$ matrix-valued functions defined on $\Cbb\backslash\Rbb$ such that 
	\begin{align}
		\sup_{z \in \partial\Rcal_y}\|\M_N(z)\| \leq C.
		\notag
	\end{align}
	Then, 
	\begin{align}
		\sup_{z \in \partial\Rcal_y} \Var\left[\frac{1}{N} \Tr \Q_N(z) \M_N(z) \chi_N\right] 
		&\leq \frac{C}{N^2},
		\notag
	\end{align}
	and for $\a_N \in \Cbb^M$ such that $\sup_N \|\a_N\| < \infty$,
	\begin{align}
		\sup_{z \in \partial\Rcal_y}\Var\left[\a_N^* \Q_N(z) \M_N(z) \a_N \chi_N\right] 
		&\leq \frac{C}{N}.
		\notag
	\end{align}
	Moreover, the same kind of uniform bounds still hold when $\Q_N(z)$ is replaced by $\Q_N(z)^2$.
\end{lemma}
The proofs of these results are based on elementary arguments, and are thus omitted.  Following the calculations
of \cite{dumont2009capacity} and \cite{vallet2010sub}, we obtain that 
\begin{align}
	\Ebb\left[\Q_N(z)\chi_N\right] =
	\R_{r,N}(z) + \Deltabs_{r,N}(z) \R_{r,N}(z) + \Ebb\left[\Q_N(z)\chi_N\right] \R_{r,N}(z) \frac{\sigma^2}{N} \Tr \Deltabs_{r,N}(z) + \Thetabs_N(z)\R_{r,N}(z)
	\label{eq:mastereq}
\end{align}
for each $z \in \Cbb \backslash\supp(\phi)$ where $\Thetabs_N(z)$ is a matrix whose elements are uniformly bounded on $\partial \Rcal_y$ by
$\frac{C}{N^{p}}$ for each $p$, and where $\Deltabs_{r,N}(z)$ is the regularized version of matrix $\Deltabs_N(z)$ introduced in lemma \ref{lemma:alphaminustrR} 
defined by
\begin{align}
	\Deltabs_{r,N}(z) &=
	-\frac{1}{(1+\sigma\alpha_{r,N}(z))^2}\Ebb\left[\Q_N(z)\chi_N\right]
	\Ebb
	\left[
	        \left(\frac{\sigma^2}{N}\Tr \Q_N(z)\chi_N - \Ebb\left[\frac{\sigma^2}{N}\Tr \Q_N(z)\chi_N\right]\right)
	        \frac{\sigma^2}{N}\Tr \Sigmabs_N^*\Q_N(z)\B_N\chi_N
	\right]
	\notag\\
	&\quad
	+
	\frac{1}{1+\sigma\alpha_{r,N}(z)}
	\Ebb
	\left[
	        \left(\frac{\sigma^2}{N}\Tr\Sigmabs_N^*\Q_N(z)\B_N \chi_N - \Ebb\left[\frac{\sigma^2}{N}\Tr\Sigmabs_N^*\Q_N(z)\B_N \chi_N\right]\right)           
	        \Q_N(z)\chi_N
	\right]
	\notag\\
	&\quad
	+ \frac{1}{1+\sigma\alpha_{r,N}(z)}
	\Ebb
	\left[
	        \left(\frac{\sigma^2}{N} \Tr\Q_N(z)\chi_N - \Ebb\left[\frac{\sigma^2}{N} \Tr\Q_N(z)\chi_N\right]\right) 
	        \Q_N(z)\Sigmabs_N\Sigmabs_N^*\chi_N
	\right],
	\label{eq:defDeltaz}
\end{align}
After some calculations using Lemmas \ref{le:bornes-utiles-regularisees}, \ref{lemma:IPPder}, \ref{lemma:var_reg}, we eventually obtain that 
\begin{align}
	\sup_{z \in \partial\Rcal_y}\left|\a_N^*\left(\Ebb[\Q_N(z)\chi_N] - \R_{r,N}(z)\right)\a_N\right| \leq \frac{C}{N^{3/2}},
	\notag\\
	\sup_{z \in \partial\Rcal_y}\left| \alpha_{r,N}(z) - \frac{\sigma}{N} \Tr(\R_{r,N}(z)) \right| & \leq \frac{C}{N^{2}} ,
	\label{eq:boundTrEQminusRreg} \\
	\sup_{z \in \partial\Rcal_y}\left| \tilde{\alpha}_{r,N}(z) - \frac{\sigma}{N} \Tr(\tilde{\R}_{r,N}(z)) \right| & \leq \frac{C}{N^{2}}.
	\label{eq:boundTrEtildeQminustildeRreg}	
\end{align}
for all large $N$. In order to prove \eqref{eq:biaisfqchi} and \eqref{eq:biaismnchi}, it remains to handle the terms involving the difference $\R_{r,N}(z)-\T_N(z)$. 
We  show in the following that 
\begin{align}
	\sup_{z \in \partial\Rcal_y} \left|\a_N^*\left(\R_{r,N}(z) - \T_N(z)\right)\a_N\right| 
	\leq \frac{C}{N^2}
	\label{eq:boundRminusTreg}
\end{align}
for all large $N$.
We start as usual with the identity $\R_{r,N}(z) - \T_N(z) = \R_{r,N}(z)\left(\T_N(z)^{-1} - \R_{r,N}(z)^{-1}\right)\T_N(z)$, to get
\begin{align}
	\a_N^*\left(\R_{r,N}(z) - \T_N(z)\right)\a_N
	&=
	\sigma \frac{\alpha_{r,N}(z)-\delta_N(z)}{\left(1+\sigma\alpha_{r,N}(z)\right)\left(1+\sigma\delta_N(z)\right)}\a_N^*\R_{r,N}(z)\B_N\B_N^*\T_N(z) \a_N
	\notag\\
	&+z\sigma\left(\tilde{\alpha}_{r,N}(z) - \tilde{\delta}_N(z)\right) \a_N^* \R_{r,N}(z)\T_N(z)\a_N.
	\notag
\end{align}
The expression \eqref{eq:expre-tildealphar} of $\tilde{\alpha}_{r,N}$ implies that 
$z (\tilde{\alpha}_{r,N}(z) - \tilde{\delta}_N(z)) = z(\alpha_{r,N}(z)-\delta_N(z)) + \Ocal\left(\frac{1}{N^{p}}\right)$ for each integer $p$. 
Thus, to prove \eqref{eq:biaisfqchi} and \eqref{eq:biaismnchi}, it is sufficient to check that
\begin{align}
        \sup_{z \in \partial\Rcal_y}\left|\alpha_{r,N}(z)-\delta_N(z)\right| \leq \frac{C}{N^2}.
        \notag
\end{align}
We will use the same ideas as in Section \ref{subsec:proof-derivee} and remark that $(\alpha_{r,N}(z)-\delta_N(z), \tilde{\alpha}_{r,N}(z) - \tilde{\delta}_N(z))$ can be interpreted as
the solution  of a $2 \times 2$ linear system whose determinant is a regularized version of \eqref{eq:def-determinant}, and appears  uniformly bounded away from zero 
on $\partial \Rcal_y$.

Using again the previous expression of  $\R_{r,N}(z)-\T_N(z)$ together with \eqref{eq:boundTrEQminusRreg}, \eqref{eq:boundTrEtildeQminustildeRreg} and repeating the procedure for 
$\tilde{\R}_{r,N}(z)-\tilde{\T}_N(z)$, we obtain
\begin{align}
	\begin{bmatrix}
		\alpha_{r,N}(z) - \delta_N(z)
		\\
		\tilde{\alpha}_{r,N}(z) - \tilde{\delta}_N(z)
	\end{bmatrix}
	=
	\begin{bmatrix}
	u_{r,N}(z) & z v_{r,N}(z)
		\\
		z \tilde{v}_{r,N}(z) & u_{r,N}(z)
	\end{bmatrix}
	\begin{bmatrix}
		\alpha_{r,N}(z) - \delta_N(z)
		\\
		\tilde{\alpha}_{r,N}(z) - \tilde{\delta}_N(z)
	\end{bmatrix}
	+
	\frac{1}{N^2}
	\begin{bmatrix}
		\epsilon_{N}(z)
		\\
		\tilde{\epsilon}_{N}(z)
	\end{bmatrix},
	\label{eq:linsysalphadelta}
\end{align}
with 
$u_{r,N}(z) = \frac{\sigma^2}{N} \Tr \frac{\R_{r,N}(z)\B_N^*\B_N\T_N(z)}{(1+\sigma\alpha_{r,N}(z))(1+\sigma\delta_N(z))}$,
$v_{r,N}(z) = \frac{\sigma^2}{N} \Tr \R_{r,N}(z)\T_N(z)$ and $\tilde{v}_{r,N}(z) = \frac{\sigma^2}{N} \Tr \tilde{\R}_{r,N}(z)\tilde{\T}_N(z)$. 
The quantities $\epsilon_N(z)$, $\tilde{\epsilon}_N(z)$  satisfy 
$\sup_{z \in \partial \Rcal_y} |\epsilon_N(z)| < C, \sup_{z \in \partial \Rcal_y} |\tilde{\epsilon}_N(z)| < C$.
The determinant of the system is given by 
\begin{align}
        \Delta_{r,N}(z) = \left(1-u_{r,N}(z)\right)^2 - z^2 v_{r,N}(z) \tilde{v}_{r,N}(z).
        \notag
\end{align}
Lemma \ref{le:bornes-utiles-regularisees} implies that for all large $N$,
$u_{r,N}(z)$, $v_{r,N}(z)$ and $\tilde{v}_{r,N}(z)$ are uniformly bounded on $\partial\Rcal_y$. 
Therefore, to conclude the proof of \eqref{eq:boundRminusTreg}, it remains to check that for all large $N$,
\begin{align}
	\inf_{z \in \partial\Rcal_y} \left|\Delta_{r,N}(z)\right| \geq C > 0.
	\notag
\end{align}
Consider the function $\check{\Delta}_{N}(z)$ where we have replaced the matrix $\R_{r,N}(z)$ and $\tilde{\R}_{r,N}(z)$ by $\T_N(z)$ and $\tilde{\T}_N(z)$, i.e
\begin{align}
        \check{\Delta}_N(z) = \left(1-\check{u}_N(z)\right)^2 - z^2 \check{v}_N(z) \check{\tilde{v}}_N(z),
        \notag
\end{align}
with
$\check{u}_N(z) = \frac{\sigma^2}{N} \Tr \frac{\T_N(z)\B_N^*\B_N\T_N(z)}{(1+\sigma\delta_N(z))(1+\sigma\delta_N(z))}$,
$\check{v}_N(z) = \frac{\sigma^2}{N} \Tr \T_N(z)^2$, and $\check{\tilde{v}}_N(z) = \frac{\sigma^2}{N} \Tr \tilde{\T}_N(z)^2$.
Denote by $h_N(z) = \Delta_N(z) - \check{\Delta}_N(z)$.
Lemmas \ref{le:bornes-regularisees}, \ref{le:bornes-utiles-regularisees} imply that $\left|\check{u}_N(z) - u_N(z)\right|$, $\left|v_N(z)-\check{v}_N(z)\right|$ and $\left|\tilde{v}_N(z)-\check{\tilde{v}}_N(z)\right|$
converge to $0$ uniformly on $\partial\Rcal_y$ which of course implies
\begin{align}
        \sup_{z \in \partial\Rcal_y} |h_N(z)| \xrightarrow[N\to\infty]{} 0.
        \label{eq:convunifdet}
\end{align}
Using Cauchy-Schwarz inequality, we get
\begin{align}
	\left|\check{\Delta}_N(z)\right| \geq \underline{\Delta}_N(z):=
	\left(1-\underline{u}_N(z)\right)^2 - |z|^2 \underline{v}_N(z)\underline{\tilde{v}}_N(z),
	\notag
\end{align}
with $\underline{u}_N(z) = \frac{\sigma^2}{N}\Tr \frac{\T_N(z)\B_N\B_N^*\T_N(z)^*}{\left|1+\sigma\delta_N(z)\right|^2}$, 
$\underline{v}_N(z) = \frac{\sigma^2}{N}\Tr\T_N(z)\T_N(z)^*$ and $\underline{\tilde{v}}_N(z) = \frac{\sigma^2}{N}\Tr\tilde{\T}_N(z)\tilde{\T}_N(z)^*$.
Now, we use the following lemma.
\begin{lemma}
	There exists a constant $C > 0$ independent of $N$ such that
	\begin{align}
		\inf_{z \in \partial\Rcal_y} \underline{\Delta}_N(z) \geq C.
		\notag
	\end{align}
\end{lemma}
\begin{proof}
	It is shown in \cite{vallet2010sub} and \cite{hachem2010bilinear} that $\underline{\Delta}_N(z)$ is the determinant of the following $2 \times 2$ linear system
	\begin{align}
		\begin{bmatrix}
			\Im\left(\delta_N(z)\right)
			\\
			\Im\left(z \tilde{\delta}_N(z)\right)
		\end{bmatrix}
		=
		\begin{bmatrix}
			\underline{u}_N(z) & \underline{v}_N(z)
			\\
			|z|^{2} \underline{\tilde{v}}_N(z) & \underline{u}_N(z)
		\end{bmatrix}
		\begin{bmatrix}
			\Im\left(\delta_N(z)\right)
			\\
			\Im\left(z \tilde{\delta}_N(z)\right)
		\end{bmatrix}
		+
		\frac{\Im\left(z\right)}{\sigma}
		\begin{bmatrix}
			\underline{v}_N(z)
			\\
			\underline{u}_N(z)
		\end{bmatrix},
		\label{eq:linsysimdelta}
	\end{align}
	and that for $z \in \Cbb\backslash\Rbb$, $\underline{\Delta}_N(z) > 0$. Solving the system, and looking at the corresponding expression of 
	$\Im(\delta_N(z))$, we easily get that 
	\begin{align}
		\underline{\Delta}_N(z) = \frac{\Im(z)}{\Im\left(\delta_N(z)\right)} \frac{\sigma}{N} \Tr \T_N(z)\T_N(z)^*.
		\notag
	\end{align}
	for $z \in  \Cbb\backslash\Rbb$. 
	Expressing $\T_N(z)-\T_N(z)^*$ as $\T_N(z) \left(\T_N(z)^{-*} - \T_N(z)^{-1} \right) \T_N(z)^*$, and using the equation $\delta_N(z) = \frac{\sigma}{N} \Tr(\T_N(z))$, 
	we obtain that 
	\begin{align}
		        \Im\left(\delta_N(z)\right) = \Im\left(w_N(z)\right) \frac{\sigma}{N} \Tr \T_N(z)\T_N(z)^*.
		        \notag
	\end{align}
	and  deduce the useful formula
	\begin{align}
		\label{eq:formule-Delta}
		\underline{\Delta}_N(z) = \frac{\Im(z)}{\Im\left(w_N(z)\right)}.
	\end{align}
	Using the integral representation $\delta_N(z) = \int_{\Scal_N} \frac{\drm \mu_N(\lambda)}{\lambda - z}$, we obtain after straightforward computations that
	the expression $\frac{\Im(z)}{\Im\left(w_N(z)\right)}$ extends to $\Cbb\backslash\Scal_N$, and therefore to $\Cbb\backslash\supp(\phi)$, and satisfies
	\begin{align}
		\sup_{z \in \partial \Rcal_y} \frac{\left| \Im(w_N(z)) \right|}{|\Im(z)|} < C
		\notag
	\end{align}
	Eq. \eqref{eq:formule-Delta} thus implies that 
	\begin{align}
		\sup_{z \in \partial\Rcal_y} \left|\underline{\Delta}_N(z)\right|^{-1} \leq C.
		\notag
	\end{align}
	which concludes the proof.
\end{proof}
We deduce from this that $\inf_{z \in \partial\Rcal_y}\left|\Delta_N(z)\right| \geq C > 0$ for all large $N$.
Therefore, we can invert the system \eqref{eq:linsysalphadelta} and obtain
\begin{align}
        \sup_{z\in\partial\Rcal_y} \left|\alpha_{r,N}(z)-\delta_N(z)\right| \leq \frac{C}{N^2},
        \notag
\end{align}
for all large $N$. This establishes \eqref{eq:biaismnchi} and completes the proof of \eqref{eq:biaisfqchi}.

The proof of \eqref{eq:biaismn'chi} is similar to the proof of Lemma  \ref{le:expre-Ehatm'N}, but as above, $\alpha_N(z), \tilde{\alpha}_N(z), 
\R_N(z)$ and $\tilde{\R}_N(z)$ have to be replaced by their regularized versions  $\alpha_{r,N}(z), \tilde{\alpha}_{r,N}(z), 
\R_{r,N}(z)$ and $\tilde{\R}_{r,N}N(z)$. The reader can check that the properties of these regularized functions allow
to follow the various steps of the proof of Lemma \ref{le:expre-Ehatm'N}.

\bibliographystyle{plain}
\bibliography{biblio} 

\begin{thebibliography}{10}

\bibitem{bai1998no}
Z.D. Bai and J.W. Silverstein.
\newblock {No eigenvalues outside the support of the limiting spectral
  distribution of large-dimensional sample covariance matrices}.
\newblock {\em Annals of Probability}, 26(1):316--345, 1998.

\bibitem{bai1999exact}
Z.D. Bai and J.W. Silverstein.
\newblock {Exact separation of eigenvalues of large dimensional sample
  covariance matrices}.
\newblock {\em Annals of Probability}, 27(3):1536--1555, 1999.

\bibitem{benaych2011singular}
F.~Benaych-Georges and R.R. Nadakuditi.
\newblock {The singular values and vectors of low rank perturbations of large
  rectangular random matrices}.
\newblock {\em submitted}.
\newblock arXiv:1103.2221.

\bibitem{capitaine2007freeness}
M.~Capitaine and C.~Donati-Martin.
\newblock {Strong asymptotic freeness of Wigner and Wishart matrices}.
\newblock {\em Indiana Univ. Math. Journal}, 56:295--309, 2007.

\bibitem{capitaine2009largest}
M.~Capitaine, C.~Donati-Martin, and D.~F{\'e}ral.
\newblock {The largest eigenvalue of finite rank deformation of large Wigner
  matrices: convergence and non-universality of the fluctuations}.
\newblock {\em Annals of Probability}, 37(1):1--47, 2009.

\bibitem{chen1982inequality}
L.H.Y. Chen.
\newblock {An inequality for the multivariate normal distribution}.
\newblock {\em Journal of Multivariate Analysis}, 12(2):306--315, 1982.

\bibitem{davidson2001local}
K.R. Davidson and S.J. Szarek.
\newblock {Local operator theory, random matrices and Banach spaces}.
\newblock {\em Handbook of the geometry of Banach spaces}, 1:317--366, 2001.

\bibitem{dozier2007analysis}
R.B. Dozier and J.W. Silverstein.
\newblock {Analysis of the limiting spectral distribution of large dimensional
  information plus noise type matrices}.
\newblock {\em Journal of Multivariate Analysis}, 98(6):1099--1122, 2007.

\bibitem{dozier2007empirical}
R.B. Dozier and J.W. Silverstein.
\newblock {On the empirical distribution of eigenvalues of large dimensional
  information plus noise type matrices}.
\newblock {\em Journal of Multivariate Analysis}, 98(4):678--694, 2007.

\bibitem{dumont2009capacity}
J.~Dumont, W.~Hachem, S.~Lasaulce, P.~Loubaton, and J.~Najim.
\newblock {On the Capacity Achieving Covariance Matrix for Rician MIMO
  Channels: An Asymptotic Approach}.
\newblock {\em IEEE Transactions on Information Theory}, 56(3):1048--1069,
  2010.

\bibitem{girko2001canonical}
V.~L. Girko.
\newblock {\em Theory of stochastic canonical equations. {V}ol. {I}}, volume
  535 of {\em Mathematics and its Applications}.
\newblock Kluwer Academic Publishers, Dordrecht, 2001.

\bibitem{haagerup2005new}
U.~Haagerup and S.~Thorbjornsen.
\newblock {A new application of random matrices: $\mathrm{Ext}(C^*_{red}(F_2))$
  is not a group}.
\newblock {\em Annals of Mathematics}, 162(2):711, 2005.

\bibitem{hachem2007deterministic}
W.~Hachem, P.~Loubaton, and J.~Najim.
\newblock {Deterministic equivalents for certain functionals of large random
  matrices}.
\newblock {\em Annals of Applied Probability}, 17(3):875--930, 2007.

\bibitem{hachem2010bilinear}
W.~Hachem, P.~Loubaton, J.~Najim, and P.~Vallet.
\newblock On bilinear forms based on the resolvent of large random matrices.
\newblock {\em submitted}, 2010.
\newblock arXiv:1004.3848.

\bibitem{hannan1973estimation}
E.J. Hannan.
\newblock The estimation of frequency.
\newblock {\em Journal of Applied probability}, 10(3):510--519, 1973.

\bibitem{james1964distributions}
A.T. James.
\newblock Distributions of matrix variates and latent roots derived from normal
  samples.
\newblock {\em The Annals of Mathematical Statistics}, 35(2):475--501, 1964.

\bibitem{loubaton2010ejpsub}
P.~Loubaton and P.~Vallet.
\newblock {Almost sure localization of the eigenvalues in a gaussian
  information plus noise model. Application to the spiked models.}
\newblock {\em submitted}, 2010.
\newblock arXiv: 1009.5807.

\bibitem{mestre2008modified}
X.~Mestre and M.A. Lagunas.
\newblock {Modified subspace algorithms for DoA estimation with large arrays}.
\newblock {\em IEEE Transactions on Signal Processing}, 56(2):598, 2008.

\bibitem{pastur2005simple}
L.~Pastur.
\newblock {A simple approach to the global regime of Gaussian ensembles of
  random matrices}.
\newblock {\em Ukrainian Math. J.}, 57(6):936--966, 2005.

\bibitem{silversteinchoi1995}
J.W. Silverstein and S.~Choi.
\newblock {Analysis of the limiting spectral distribution of large dimensional
  random matrices}.
\newblock {\em Journal of Multivariate Analysis}, 52(2):175--192, 1995.

\bibitem{stoica1989music}
P.~Stoica and N.~Arye.
\newblock Music, maximum likelihood, and cramer-rao bound.
\newblock {\em IEEE Transactions on Acoustics, Speech and Signal Processing},
  37(5):720--741, 1989.

\bibitem{vallet2011improved}
P.~Vallet, W.~Hachem, P.~Loubaton, X.~Mestre, and J.~Najim.
\newblock {An improved MUSIC algorithm based on low rank perturbation of large
  random matrices}.
\newblock In {\em Proceedings of the 2011 IEEE Workshop on Statistical Signal
  Processing}. IEEE Signal Processing Society, 2011.

\bibitem{vallet2010sub}
P.~Vallet, P.~Loubaton, and X.~Mestre.
\newblock {Improved Subspace Estimation for Multivariate Observations of High
  Dimension: The Deterministic Signal Case}.
\newblock {\em accepted for publication in IEEE Transactions on Information
  Theory}.
\newblock arXiv: 1002.3234.

\end{thebibliography}

\end{document}